\documentclass[12pt,a4paper,reqno]{amsart}
\usepackage{color}
\usepackage{amsfonts,amsmath,amssymb,amsxtra,url,float} 
\allowdisplaybreaks[4] 
\usepackage[colorlinks,linkcolor=RoyalBlue,anchorcolor=Periwinkle,citecolor=Orange,urlcolor=Green]{hyperref} 
\usepackage[usenames,dvipsnames]{xcolor} 
\usepackage{enumitem}
\usepackage{geometry} 
\usepackage{rotating} 
\usepackage{lscape} 
\DeclareSymbolFont{largesymbol}{OMX}{yhex}{m}{n}
\DeclareMathAccent{\Widehat}{\mathord}{largesymbol}{"62} 
\usepackage{multirow}
\usepackage{graphicx} 
\usepackage{subfigure} 
\usepackage{tikz}
\usetikzlibrary{matrix}
\usetikzlibrary{trees}
\usepgflibrary{decorations.pathmorphing}
\usetikzlibrary{decorations.pathmorphing}
\usetikzlibrary{decorations.markings}
 \tikzset{
  on each segment/.style={
    decorate,
    decoration={
      show path construction,
      moveto code={},
      lineto code={
        \path [#1]
        (\tikzinputsegmentfirst) -- (\tikzinputsegmentlast);
      },
      curveto code={
        \path [#1] (\tikzinputsegmentfirst)
        .. controls
        (\tikzinputsegmentsupporta) and (\tikzinputsegmentsupportb)
        ..
        (\tikzinputsegmentlast);
      },
      closepath code={
        \path [#1]
        (\tikzinputsegmentfirst) -- (\tikzinputsegmentlast);
      },
    },
  },
  mid arrow/.style={postaction={decorate,decoration={
        markings,
        mark=at position .5 with {\arrow[#1]{stealth}}
      }}},
}
\usetikzlibrary{arrows}

\usetikzlibrary{patterns}
\usetikzlibrary{shadings} 
\usepackage[all]{xy} 
\usepackage{mathdots} 
\usepackage{dsfont} 
\usepackage{cite}
\usepackage{mathrsfs} 
\usepackage{multicol} 
\numberwithin{figure}{section}
\usepackage{marginnote} 
\usepackage{graphicx} 
\usepackage{multicol} 

\usepackage{fancyhdr} 
\newtheorem{theorem}{Theorem}[section]
\newtheorem{lemma}[theorem]{Lemma}
\newtheorem{corollary}[theorem]{Corollary}
\newtheorem{main theorem}[theorem]{Main Theorem}
\newtheorem{proposition}[theorem]{Proposition}
\newtheorem{definition}[theorem]{Definition}

\newtheorem{remark}[theorem]{Remark}
\newtheorem{example}[theorem]{Example}
\newtheorem{notation}[theorem]{Notation}
\newtheorem{question}[theorem]{Question}

\usetikzlibrary{arrows}

\numberwithin{equation}{section}

\begin{document}

\def\headertitle{Homological dimensions over almost gentle algebras}

\title[\headertitle]{Homological dimensions over almost gentle algebras}

\def\fstpage{1} 
\def\page{$\begin{matrix} {\color{white}0} \\ \thepage \end{matrix}$} 

\pagestyle{fancy}
\fancyhead[LO]{ }
\fancyhead[RO]{ }
\fancyhead[CO]{\ifthenelse{\value{page}=\fstpage}{\ }{\scriptsize{\headertitle}}}
\fancyhead[LE]{ }
\fancyhead[RE]{ }
\fancyhead[CE]{\scriptsize{Demin WANG, Zhaoyong HUANG, Yu-Zhe LIU}}
\fancyfoot[L]{ }
\fancyfoot[C]{\page} 
\fancyfoot[R]{ }
\renewcommand{\headrulewidth}{0.5pt} 

\thanks{$^{\ast}$Corresponding author.}


\author{Demin Wang}
\address{School of Mathematics, Nanjing University, Nanjing 210093, China}
\email{602023210012@smail.nju.edu.cn}

\author{Zhaoyong Huang}
\address{School of Mathematics, Nanjing University, Nanjing 210093, China}
\email{huangzy@nju.edu.cn}

\author{Yu-Zhe Liu$^{*}$}
\address{School of Mathematics and Statistics, Guizhou University, Guiyang 550025, Guizhou, P. R. China}
\email{liuyz@gzu.edu.cn / yzliu3@163.com (Y.-Z. Liu)}



\maketitle

\begin{abstract}
  We provide a method for computing the global dimension and self-injective dimension of almost gentle algebras,
  and prove that an almost gentle algebra is Gorenstein if it satisfies the Auslander condition.
\end{abstract}

\tableofcontents 


\section{Introduction}

In \cite{AS1987}, Assem and Skowro\'{n}ski introduced gentle algebras and use them to study tilted algebras and derived equivalence.
Afterwards, (skew-)gentle algebras play an important role in representation theory;
and theirs module categories and derived categories were described by using string algebras, Auslander--Reiten theory, combinatorial methods
and geometric models, see \cite{BR1987, BCS2021,ALP2016,BD2017,OPS2018,Z2019,QZZ2022} and so on.
In particular, the global dimension and self-injective dimension of (locally) gentle algebras have been studied,
and (locally) gentle algebras were shown to be Gorenstein by using different methods \cite{GR2005,LGH2024,FOZ2024}.

As a generalization of gentle algebras, almost gentle algebras are important finite-dimensional algebras introduced
by Green and Schroll \cite{GS2018AG}, which are monomial special multiserial algebras,
and are closely related to some hypergraphs and Brauer configuration algebras. In \cite{FGR2022} and \cite{Jaw2024},
complexes on almost gentle algebras and the trivial extension of almost gentle algebras were studied, respectively.
Based on the above, we have the following questions.

\begin{question} \
\begin{itemize}
\item[\rm(1)] How to characterize the global dimension and self-injective dimension of almost gentle algebras?
\item[\rm(2)] Whether or when are almost gentle algebras Gorenstein?
\end{itemize}
\end{question}


\newcommand{\defines}{\it}


\def\BLUE{blue} 
\def\RED{red} 


\def\<{\langle} 
\def\>{\rangle} 
\def\NN{\mathbb{N}} 
\def\ZZ{\mathbb{Z}} 
\def\QQ{\mathbb{Q}} 
\def\RR{\mathbb{R}} 
\def\II{\mathbb{I}} 

\newcommand{\Pic}{F{\tiny{IGURE}}\ }
\newcommand{\modcat}{\mathsf{mod}}
\newcommand{\repcat}{\mathsf{rep}}
\newcommand{\proj}{\mathsf{proj}}
\newcommand{\rmproj}{\mathrm{proj}}
\newcommand{\Dcat}{\mathcal{D}}
\newcommand{\ind}{\mathsf{ind}}
\newcommand{\per}{\mathsf{per}} 
\newcommand{\Hom}{\mathrm{Hom}} %
\newcommand{\End}{\mathrm{End}} %
\newcommand{\Ext}{\mathrm{Ext}} %
\newcommand{\op}{\mathrm{op}}
\newcommand{\newarrow}{{{\Leftarrow \mkern-8.5mu\raisebox{0.36em}[]{$\shortmid$}}{\mkern-10.5mu\raisebox{-0.098em}[]{$-$}}
{\mkern-4.1mu\raisebox{0.49em}[]{$_{|}$}}}} 
\newcommand{\tildearrow}{\sim\mkern-8.5mu\raisebox{-0.36em}{$^{\succ}$}}
\newcommand{\shadow}[1]{{\color{#1}$\blacksquare\!\!\blacksquare$}}
\newcommand{\sectcolor}{\color{blue}}
\newcommand{\To}[2]{\mathop{-\!\!\!-\!\!\!\longrightarrow}\limits^{#1}_{#2}}
\newcommand{\oT}[2]{\mathop{\longleftarrow\!\!\!-\!\!\!-}\limits^{#1}_{#2}}


\def\kk{\mathds{k}} 
\def\Q{\mathcal{Q}} 
\def\I{\mathcal{I}} 
\def\source{\mathfrak{s}}
\def\target{\mathfrak{t}}
\def\Str{\mathrm{Str}}
\def\Band{\mathrm{Ban}}
\def\M{\mathds{M}}
\def\ind{\mathrm{ind}}
\def\fdim{\mathrm{fl.dim}}
\def\fidim{\mathrm{fi.dim}}
\def\pdim{\mathrm{proj.dim}}
\def\idim{\mathrm{inj.dim}}
\def\gldim{\mathrm{gl.dim}}
\def\Ker{\mathrm{Ker}}
\def\Image{\mathrm{Im}}
\def\top{\mathrm{top}}
\def\op{\mathrm{op}}
\def\claw{{\ \shortmid\mkern-13mu\wedge\ }}
\def\aclaw{{\ \shortmid\mkern-13mu\vee\ }}
\def\clawnota{\kappa}
\def\aclawnota{\xi}
\def\C{\mathscr{C}}
\def\da{\mathscr{\downarrow}}
\def\bfda{\pmb{\downarrow}}
\def\bfua{\pmb{\uparrow}}
\def\ds{\delta} 
\def\forb{\mathcal{F}}
\newcommand{\rel}[2]{({#1},{#2})} 
\def\Left{\mathrm{L}}
\def\Right{\mathrm{R}}
\def\inner{\mathrm{in}}
\def\out{\mathrm{out}}
\def\spds{\mathit{\Psi}}
\def\rad{\mathrm{rad}}
\def\soc{\mathrm{soc}}
\def\noname{invalid\ }


Let $\kk$ be an algebraically closed field.
In this paper, we provide a method for computing the global dimension and self-injective dimension of almost gentle algebras
by using forbidden paths introduced by Avella-Alaminos and Gei\ss \cite{AG2008} (see Section \ref{Sect:Syzygy(ds)}).
We define claw, written as $\clawnota$, and anti-claw, written as $\aclawnota$ (Definition \ref{def:claw}),
to describe indecomposable projective modules and injective modules,
and introduce tow sets $\forb(v)$ and $\forb(\aclawnota)$ whose elements are some special forbidden paths
on the bound quiver $(\Q,\I)$ of an almost gentle algebra $A=\kk\Q/\I$, which are decided by some claws,
where $v$ is any vertex of the quiver $\Q$, see Notations \ref{nota:F(v)}, \ref{nota:F(ds)}, and Definition \ref{def:F(aclawnota)}.
By \cite[Corollary 2.4]{GKK91}, we have that the finitistic dimension conjecture holds true for monomial algebras.
Observe that an algebra is monomial if and only if so is its opposite algebra. It then follows from
\cite[Proposition 6.10]{AR1991Nakayama} that the left and right self-injective dimensions
of a monomial algebra, and hence, of an almost gentle algebra $A$, are identical. For simplicity, we call this common quantity
the {\it self-injective dimension} of $A$. The global dimension $\gldim A$ and self-injective dimension $\idim A$
of an almost gentle algebra $A$ can be described by the following results.

\begin{theorem} \label{thm-1.1}\
\begin{itemize}
  \item[\rm(1)] {\rm(Theorem \ref{thm:gldim})}
    \[\gldim A = \sup_{F \in \forb} \ell(F), \]
    where $\forb:=\bigcup\limits_{v\in\Q_0}\forb(v)$ and $\ell(F)$ is the length of $F$.
  \item[\rm(2)] {\rm(Theorem \ref{thm:selfdim})}
    \[ \idim A = \sup_{F\in \forb_{\mathrm{a}}} \ell(F), \]
    where $\forb_{\mathrm{a}} := \bigcup_{v\in\Q_0} \forb(\aclawnota_v)$ with
    $\aclawnota_v$ the anti-claw corresponding to the indecomposable injective module $E(v)$.
\end{itemize}
\end{theorem}

Theorem \ref{thm-1.1}(2) provides a method for judging the Gorensteiness of almost gentle algebras as follows.

\begin{theorem} For an almost gentle algebra $A=\kk\Q/\I$, the following statements are equivalent.
\begin{itemize}
  \item[\rm(1)] $\idim A = \infty$.
  \item[\rm(2)] {\rm(A direct consequence of Theorem \ref{thm:selfdim})} There exists an anti-claw $\aclawnota$
  corresponding to some indecomposable injective module such that $\forb(\aclawnota)$ contains a forbidden cycle whose length is infinite.
  \item[\rm(3)] {\rm(Theorem \ref{thm:selfdim II})}
    $(\Q,\I)$ contains an oriented cycle $\C=a_0a_1\cdots a_{\ell-1}=$
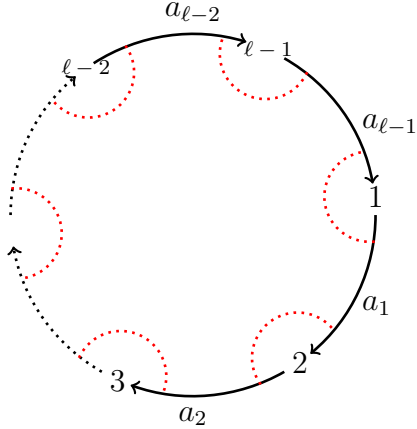
\begin{center}
\begin{tikzpicture}[scale=1.2]
\draw[rotate=   0][->] (2,0) arc(0:-50:2) [line width=1pt];
\draw[rotate= -60][->] (2,0) arc(0:-50:2) [line width=1pt];
\draw[rotate=-120][->] (2,0) arc(0:-50:2) [line width=1pt][dotted];
\draw[rotate=-180][->] (2,0) arc(0:-50:2) [line width=1pt][dotted];
\draw[rotate=-237][->] (2,0) arc(0:-50:2) [line width=1pt];
\draw[rotate=-300][->] (2,0) arc(0:-50:2) [line width=1pt];
\draw[rotate= 120]     (2,0.2) node{\tiny$\ell-2$};
\draw[rotate=  60]     (2,0.2) node{\tiny$\ell-1$};
\draw[rotate=   0]     (2,0.2) node{$1$};
\draw[rotate= -60]     (2,0.2) node{$2$};
\draw[rotate=-120]     (2,0.2) node{$3$};
\draw[rotate=  60] (1.73,1) node[above]{$a_{\ell-2}$};
\draw[rotate=   0] (1.73,1) node[right]{$a_{\ell-1}$};
\draw[rotate= -60] (1.73,1) node[right]{$a_1$};
\draw[rotate=-120] (1.73,1) node[below]{$a_2$};
\draw[red][rotate=6-  0] (1.95,-0.5 ) arc(-90:-270:0.5) [line width=1pt][dotted];
\draw[red][rotate=6- 60] (1.95,-0.5 ) arc(-90:-270:0.5) [line width=1pt][dotted];
\draw[red][rotate=6-120] (1.95,-0.5 ) arc(-90:-270:0.5) [line width=1pt][dotted];
\draw[red][rotate=6-180] (1.95,-0.5 ) arc(-90:-270:0.5) [line width=1pt][dotted];
\draw[red][rotate=6-240] (1.95,-0.5 ) arc(-90:-270:0.5) [line width=1pt][dotted];
\draw[red][rotate=6-300] (1.95,-0.5 ) arc(-90:-270:0.5) [line width=1pt][dotted];
\end{tikzpicture}
\end{center}
such that $a_{\overline{i}}a_{\overline{i+1}} \in \I$ {\rm(}for any $x\in\NN$, $\overline{x}$
is defined as $x$ modulo $\ell${\rm)}, and there is a vertex $v$ on $\C$ such that one of the following condition holds.
    \begin{itemize}
      \item[\rm(A)] there is an arrow $\alpha$ {\rm(}$\ne a_{\overline{v-1}}${\rm)} ending at $v$ satisfying $\alpha a_v \in \I$;
      \item[\rm(B)] there is an arrow $\beta$ {\rm(}$\ne a_{v}${\rm)} starting at $v$ satisfying $a_{\overline{v-1}}\beta \in \I$.
    \end{itemize}
    The oriented cycle as above is said to be a forbidden cycle.

  \item[\rm(4)] {\rm(Theorem \ref{thm:forb cycle}(1))} There is a vertex $v$ on some forbidden cycle $\C$ such that $v$ is not
  \noname {\rm(}see Definition \ref{def:noname} for the definition of \noname vertices{\rm)}.
\end{itemize}
\end{theorem}

Recall that a left and right Noetherian ring $R$ is called {\it Gorenstein} if its left and right self-injective dimensions
are finite, and $R$ is said to satisfy the {\it Auslander condition} if the flat dimension of the $i$-th term in a
minimal injective coresolution of $R$ as a left $R$-module is at most
$i-1$ for any $i \geqslant 1$. The Auslander condition is left and right symmetric \cite[Theorem 3.7]{FGR75}.
Bass \cite{Bass1963} proved that a commutative Noetherian ring $R$
is Gorenstein if and only if it satisfied the Auslander condition.
Based on it, Auslander and Reiten \cite{Aus1994} conjectured that an Artin algebra is Gorenstein if it
satisfies the Auslander condition. We call this conjecture {\bf ARC} for short.
It is situated between the well known Nakayama conjecture and the generalized Nakayama conjecture
\cite[p.2]{Aus1994}. All these conjectures remain still open. As an application of the above results,
we obtain the following result.


\begin{theorem} \
\begin{itemize}
\item[\rm(1)] {\rm(Corollary \ref{coro:injenv(P)})}
For an almost gentle algebra $A=\kk\Q/\I$,
the projective dimension of the injective envelope of $A_A$ is infinite.

\item[\rm(2)] {\rm(Theorem \ref{thm:AGC})} {\bf ARC} holds true for almost gentle algebras.
\end{itemize}
\end{theorem}

\section{Preliminaries}

Throughout this paper, assume that $\Q = (\Q_0,\Q_1,\source,\target)$ is a finite connected quiver.
Here, $\Q_0$ and $\Q_1$ respectively are the vertex and arrow sets of $\Q$,
and $\source$ and $\target$ are functions from $\Q_1$ to $\Q_0$ respectively sending each arrow to its source and sink.
For arbitrary two arrows $\alpha$ and $\beta$ of the quiver $\Q$,
if $\target(\alpha)=\source(\beta)$, then the composition of $\alpha$ and $\beta$ is denoted by $\alpha\beta$.
For any $v\in\Q_0$, we use $e_v$ to denote the path of length zero corresponding to $v$.
For a path on $\Q$, we use $\ell(p)=n$ to denote the length of $p$.
Let $A$ be a finite-dimensional algebra. We use $\modcat(A)$ to denote the category of finitely generated right $A$-modules.
For arbitrary two modules $M$ and $N$, if $M$ is a direct summand of $N$, then we write $M\leq_{\oplus}N$.
Moreover, we use $S(v_0), P(v_0)$ and $E(v_0)$ to denote the simple, indecomposable projective and indecomposable injective modules
corresponding to the vertex $v_0\in\Q_0$, respectively.

\subsection{Almost gentle algebras}

Recall from \cite{GS2018AG} that a bound quiver $(\Q,\I)$ is called an {\defines almost gentle pair}
if it satisfies the following conditions:
\begin{itemize}
  \item[(1)] $I$, the admissible ideal of the path algebra $\kk\Q := \mathrm{span}_{\kk}(\Q_l \mid l\in\NN )$,
    is a subspace of the $\kk$-linear space $\kk\Q$ which is generated by some paths on $\Q$ of length two;
  \item[(2)] for any arrow $a\in\Q_1$, there is at most one arrow $b\in\Q_1$ such that $ab\notin \I$
    and at most one arrow $c\in\Q_1$ such that $ca\notin \I$.
\end{itemize}

\begin{definition}  \rm \cite{GS2018AG}
A finite-dimensional algebra $A$ is called an {\defines almost gentle algebra} if its bound quiver is an almost gentle pair.
\end{definition}

\begin{example} \label{exp:almost gent} \rm
Let $\Q$ be the quiver shown in \Pic \ref{fig:almost gent}, and let $\I$ be the admissible ideal of $\kk\Q$ given by the paths
$a_{1,2_{\Left}}a_{2_{\Left},3_{\Left}}$, $a_{2,3}a_{3,4}$, $a_{1,2}a_{2,3}$, $a_{3,4}a_{4,5}$,
$a_{2,3'}a_{3',4}$, $a_{2_{\Right}, 3_{\Right}}a_{3_{\Right}, 5}$, $a_{1,2}a_{2,3'}$, $a_{3',4}a_{4,5}$,
$a_{1,2}a_{2,4}$, $a_{3_{\Left},4_{\Left}}b_{4_{\Left},5}$, $b_{1,2_{\Right}}a_{2_{\Right},3_{\Right}}$.
\begin{figure}[htbp]
\centering
\begin{tikzpicture}[scale=1.3]
\draw ( 0  , 3  ) node{$1$};
\draw ( 0  , 1.5) node{$2$};
\draw (-1.2, 0  ) node{$3$};
\draw ( 1.2, 0  ) node{$3'$};
\draw ( 0  ,-1.5) node{$4$};
\draw ( 0  ,-3  ) node{$5$};
\draw (-2.1, 2.1) node{$2_{\Left}$};
\draw (-3.0, 0  ) node{$3_{\Left}$};
\draw (-2.1,-2.1) node{$4_{\Left}$};
\draw ( 2.1, 2.1) node{$2_{\Right}$};
\draw ( 3.0, 0  ) node{$3_{\Right}$};
\draw ( 2.1,-2.1) node{$4_{\Right}$};
\draw[line width = 1pt][->] ( 0  ,-1.7) -- (   0,-2.8); \draw ( 0  ,-2.2) node[right]{$a_{4,5}$};
\draw[line width = 1pt][->] ( 0  , 1.2) -- (   0,-1.2); \draw ( 0  , 0  ) node[ left]{$a_{2,4}$};
\draw[line width = 1pt][->] ( 0  , 2.8) -- (   0, 1.7); \draw ( 0  , 2.2) node[ left]{$a_{1,2}$};
\draw[line width = 1pt][->] (-0.2, 1.5) to[out= 180, in=  90] (-1.2, 0.2);
\draw[line width = 1pt][->] (-1.2,-0.2) to[out= -90, in= 180] (-0.2,-1.5);
\draw[line width = 1pt][->] ( 0.2, 1.5) to[out=   0, in=  90] ( 1.2, 0.2);
\draw[line width = 1pt][->] ( 1.2,-0.2) to[out= -90, in=   0] ( 0.2,-1.5);
\draw (-1. , 1.0) node[ left]{$a_{2,3}$};
\draw ( 1. , 1.0) node[right]{$a_{2,3'}$};
\draw (-1. ,-1.0) node[ left]{$a_{3,4}$};
\draw ( 1. ,-1.0) node[right]{$a_{3',4}$};
\draw[line width = 1pt][->][rotate= 5+  0] (0,3) arc(90:125:3);
\draw[line width = 1pt][->][rotate= 5+ 45] (0,3) arc(90:125:3);
\draw[line width = 1pt][->][rotate= 5+ 90] (0,3) arc(90:125:3);
\draw[line width = 1pt][->][rotate= 5+135] (0,3) arc(90:125:3);
\draw[line width = 1pt][->][rotate=-5-  0] (0,3) arc(90:55:3);
\draw[line width = 1pt][->][rotate=-5- 45] (0,3) arc(90:55:3);
\draw[line width = 1pt][->][rotate=-5- 90] (0,3) arc(90:55:3);
\draw[line width = 1pt][->][rotate=-5-135] (0,3) arc(90:55:3);
\draw[line width = 1pt][->][rotate=-5+  0] (0,2.8) arc(90:55:2.8);
\draw[line width = 1pt][<-][rotate=-5+180] (0,2.8) arc(90:55:2.8);
\draw[rotate=   0] ( 1.06, 3.21) node{$a_{1,2_{\Right}}$};
\draw[rotate=   0] ( 0.96, 2.35) node{$b_{1,2_{\Right}}$};
\draw[rotate=  45] ( 1.06, 3.21) node{$a_{1,2_{\Left}}$};
\draw[rotate=  90] ( 1.06, 3.21) node{$a_{2_{\Left},3_{\Left}}$};
\draw[rotate= 135] ( 1.06, 3.21) node{$a_{3_{\Left},4_{\Left}}$};
\draw[rotate= 180] ( 1.06, 3.21) node{$a_{4_{\Left},5}$};
\draw[rotate= -45] ( 1.06, 3.21) node{$a_{2_{\Right},3_{\Right}}$};
\draw[rotate= -90] ( 1.06, 3.21) node{$a_{3_{\Right},4_{\Right}}$};
\draw[rotate=-135] ( 1.06, 3.21) node{$a_{4_{\Right},5}$};
\draw[rotate= 180] ( 0.96, 2.35) node{$b_{4_{\Left},5}$};
\draw[ red][rotate=   0] (-1.72, 2.45) arc(  45:-135:0.5) [line width = 1pt][dotted];
\draw[ red][rotate=   0] (-1.21, 0.50) arc(  90: 270:0.5) [line width = 1pt][dotted];
\draw[ red][rotate=   0] ( 0.00, 2.00) arc(  90: 185:0.5) [line width = 1pt][dotted];
\draw[blue][rotate=   0] ( 0.00, 2.20) arc(  90:  -5:0.7) [line width = 1pt][dotted];
\draw[blue][rotate= 180] (-1.72, 2.45) arc(  45:-135:0.5) [line width = 1pt][dotted];
\draw[blue][rotate= 180] (-1.21, 0.50) arc(  90: 270:0.5) [line width = 1pt][dotted];
\draw[blue][rotate= 180] ( 0.00, 2.00) arc(  90: 185:0.5) [line width = 1pt][dotted];
\draw[ red][rotate= 180] ( 0.00, 2.20) arc(  90:  -5:0.7) [line width = 1pt][dotted];
\draw[green][rotate=   0] (-1.62,-2.25) arc(-25: 135:0.5) [line width = 1pt][dotted];
\draw[green][rotate= 180] (-1.62,-2.25) arc(-25: 135:0.5) [line width = 1pt][dotted];
\draw[green] ( 0.00, 2.10) to[out=0,in=0] ( 0.00, 0.90) [line width = 1pt][dotted];
\end{tikzpicture}
\caption{An almost gentle pair}
\label{fig:almost gent}
\end{figure}
Then $(\Q,\I)$ is an almost gentle pair and $A=\kk\Q/\I$ is an almost gentle algebra.
\end{example}

\subsection{Strings and string modules}

For any arrow $a\in \Q_1$, Butler and Ringel \cite{BR1987} introduced the {\defines formal inverse} $a$, written as $a^{-1}$,
and, naturally, define $\target(a^{-1})=\source(a)$ and $\source(a^{-1})=\target(a)$.
We denote by $\Q_1^{-1}=\{a^{-1}\mid a\in \Q_1\}$ the set of all formal inverses of arrows.
Then any path $p=a_1a_2\cdots a_n$ in $(\Q,I)$ naturally provides a formal inverse path
$p^{-1} = a_{n}^{-1}a_{n-1}^{-1}\cdots a_1^{-1}$ of $p$.
For any path $p$, we define $(p^{-1})^{-1}=p$, and for any path $e_v$ of length zero corresponding to $v\in \Q_0$, we define $e_v^{-1}=e_v$.

In a bound quiver $(\Q,\I)$ of a finite-dimensional algebra $A$, a {\defines string} (of length $m$) on $(\Q,\I)$
is a sequence $s=(a_1,\ldots,a_m)$ such that
\begin{itemize}
  \item each $a_i$ ($1\leqslant i\leqslant m$) is either an arrow or a formal inverse of the arrow;
  \item if $a_i\in\Q_1$ and $a_{i+1}\in\Q_1^{-1}$, then $a_i\notin a_{i+1}^{-1}$;
  \item $\target(a_i)=\source(a_{i+1})$ holds for all $1\leqslant i\leqslant n-1$.
\end{itemize}
In particular, a string of length zero is said to be a {\defines simple string}; and
a string which is one of the forms
$\bullet \longrightarrow \bullet \longrightarrow \cdots \longrightarrow \bullet$
and $\bullet \longleftarrow \bullet \longleftarrow \cdots \longleftarrow \bullet$ (length $\geqslant 0$)
is said to be a {\defines directed string};
and we define $0$, the zero vector in $\kk\Q$, is also a string which is said to be {\defines a trivial string}.

Two strings $s_1$ and $s_2$ are called {\defines equivalent} if $s_1=s_2^{-1}$ or $s_1=s_2$.

We denote by $\Str(A)$ the set of all equivalent classes of strings.
Now, similar to \cite{BR1987}, we define a string module with respect to a given string.

\begin{definition}[{String modules}] \rm
Suppose that $s=a_1a_2\cdots a_n$ is a string on $\Q$, where $a_1,\ldots, a_n$ are arrows such that
$\source(a_i)=v_i$ ($1\leqslant i\leqslant n$) and $\target(a_n)=v_{n+1}$.
It induces a {\defines string module}, denote it by $\M(s)$, which satisfies the following conditions:
\begin{itemize}
  \item for $v\in \Q_0$, we have that $\dim_k(\M(s)e_v)$ equals to the multiplicities of $s$ transverses $v$;
  \item any $a_i\in \Q_1$ on $s$ provides an identity from $\kk_i$ to $\kk_{i+1}$, where $\kk_i$ and $\kk_{i+1}$
  are copies of $\kk$ which, as two $\kk$-vector spaces, are direct summands of $\M(s)e_{v_i}$ and $\M(s)e_{v_{i+1}}$, respectively;
  \item any arrow $a\in \Q_1$ which does not on $s$ provides a zero action $\M(s)a=0$.
\end{itemize}
For a directed string $s=a_1a_2\cdots a_n$ with $a_1,a_2,\ldots, a_n\in \Q_1$,
if either $\target (s)$ is a sink or for each arrow $\alpha$ with $\source(\alpha)=\target(s)$,
and the concatenation $s\alpha=0$, then $s$ is said to be a {\defines right maximal directed string}.
Dually, a {\defines left maximal directed string} is defined.
Furthermore, the string module induced by a (right maximal) directed string $s$
is said to be a {\defines {\rm(}right maximal{\rm)} directed string module}.
\end{definition}

\begin{remark} \rm
(1) A {\defines band} (of length $m$) is a string $b=(b_1,\ldots,b_m)$ with $\target(b_m)=\source(b_1)$ such that
\begin{itemize}
  \item $b^2$ is a string of length $2m$;
  \item $b$ is not a non-trivial power of some string,
    that is, there is no string $s$ such that $b=s^n$ for some $n\geqslant 2$.
\end{itemize}
Two bands $b=b_1\cdots b_m$ and $b'= b_1'\cdots b_m'$ are called {\defines equivalent} if either
$b[t] = b'$ or $b[t]^{-1}=b'$, where $b[t] = b_{1+t}b_{2+t}\cdots b_{n-1}b_{n}b_1 \cdots b_t$.
We denote by $\Band(A)$ the set of all equivalent classes of bands on the bound quiver of $A$.

Suppose that $b=b_1b_2\cdots b_n$ is a band on $\Q$, where $\source(a_i)=v_i$ ($1\leqslant i\leqslant n$).
For any $u\in\NN^+$ and $0\ne \lambda\in \kk$, it induces a {\defines band module},
denote it by $\M(s,J_n(\lambda))$, which satisfies the following conditions:
\begin{itemize}
  \item for $v\in \Q_0$, we have $\dim_k(\M(s)e_v)=tu$, where $u$ is the multiplicities of $s$ transverses $v$;
  \item any $b_i\in \Q_1$ with $1<i\leqslant n$ provides an identity from $\kk_i^{\oplus u}$ to $\kk_{i+1}^{\oplus u}$,
  where $\kk_i$ and $\kk_{i+1}$ are copies of $\kk$ which, as two $\kk$-vector spaces, are direct summands of
  $\M(s,J_n(\lambda))e_{v_i}$ and $\M(s,\pmb{J}_n(\lambda))e_{v_{i+1}}$, respectively;
  \item the arrow $b_1$ provides a $\kk$-linear map
    $\pmb{J}_n(\lambda): = \left(\begin{smallmatrix}
    1 & & & \\
    \lambda &1 & & \\
     & \ddots & \ddots &\\
     & & \lambda & 1
    \end{smallmatrix}\right)$
    from $\kk_1^{\oplus u}$ to $\kk_{2}^{\oplus u}$;
  \item any right $A$-action $\M(s)a$ given by the arrow $a\in \Q_1$ does not on $s$ is a zero action.
\end{itemize}
All indecomposable modules defined on a string algebra have been described  in \cite{BR1987} by the following bijection:
\[ \M: \Str(A) \cup (\Band(A)\times\mathscr{J}) \to \ind(\modcat(A)). \]
Here, $\mathscr{J}$ is the set of all Jordan blocks $J_n(\lambda)$ with $\lambda\ne 0$ and $\ind(\modcat(A))$
is the set of all isoclasses of indecomposable modules over a string algebra $A$.

(2)  Strings and bands originated from V-sequences and primitive V-sequences introduced by Wald and Waschb\"{u}sch \cite{WW1985}.
(Primitive) V-sequences are used to describe the module categories of biserial algebras.
\end{remark}

\begin{example}\rm
Let $A=\kk\Q/\I$ be the almost gentle algebra given in Example \ref{exp:almost gent}.
Then $s_1 = a_{1,2_{\Right}} b_{1,2_{\Right}}^{-1}$ and $s_2 = a_{2,4} a_{4,5} a_{4_{\Left},5}^{-1} b_{4_{\Left},5} a_{4,5}^{-1}$
are strings, but they are not directed strings;
and $s_3 = a_{2,4}a_{4,5}$ is a direct string.
The string modules $\M(s_1)$, $\M(s_2)$ and $\M(s_3)$ are shown in \Pic \ref{fig:str mods}.
\begin{figure}[htbp]
\centering
\begin{tikzpicture}[scale=0.8]
\draw ( 0  , 3  ) node{$\kk^{2}$};
\draw ( 0  , 1.5) node{$0$};
\draw (-1.2, 0  ) node{$0$};
\draw ( 1.2, 0  ) node{$0$};
\draw ( 0  ,-1.5) node{$0$};
\draw ( 0  ,-3  ) node{$0$};
\draw (-2.1, 2.1) node{$0$};
\draw (-3.0, 0  ) node{$0$};
\draw (-2.1,-2.1) node{$0$};
\draw ( 2.1, 2.1) node{$\kk$};
\draw ( 3.0, 0  ) node{$0$};
\draw ( 2.1,-2.1) node{$0$};
\draw[line width = 1pt][->] ( 0  ,-1.7) -- (   0,-2.8);
\draw[line width = 1pt][->] ( 0  , 1.2) -- (   0,-1.2);
\draw[line width = 1pt][->] ( 0  , 2.8) -- (   0, 1.7);
\draw[line width = 1pt][->] (-0.2, 1.5) to[out= 180, in=  90] (-1.2, 0.2);
\draw[line width = 1pt][->] (-1.2,-0.2) to[out= -90, in= 180] (-0.2,-1.5);
\draw[line width = 1pt][->] ( 0.2, 1.5) to[out=   0, in=  90] ( 1.2, 0.2);
\draw[line width = 1pt][->] ( 1.2,-0.2) to[out= -90, in=   0] ( 0.2,-1.5);
\draw[line width = 1pt][->][rotate= 5+  0] (0,3) arc(90:125:3);
\draw[line width = 1pt][->][rotate= 5+ 45] (0,3) arc(90:125:3);
\draw[line width = 1pt][->][rotate= 5+ 90] (0,3) arc(90:125:3);
\draw[line width = 1pt][->][rotate= 5+135] (0,3) arc(90:125:3);
\draw[line width = 1pt][->][rotate=-5-  0] (0,3) arc(90:55:3);
\draw[line width = 1pt][->][rotate=-5- 45] (0,3) arc(90:55:3);
\draw[line width = 1pt][->][rotate=-5- 90] (0,3) arc(90:55:3);
\draw[line width = 1pt][->][rotate=-5-135] (0,3) arc(90:55:3);
\draw[line width = 1pt][->][rotate=-5+  0] (0,2.8) arc(90:55:2.8);
\draw[line width = 1pt][<-][rotate=-5+180] (0,2.8) arc(90:55:2.8);
\draw[rotate=   0] ( 1.06, 3.21) node{\tiny$[0\ 1]$};
\draw[rotate=   0] ( 0.96, 2.35) node{\tiny$[1\ 0]$};
\draw[ red][rotate=   0] (-1.72, 2.45) arc(  45:-135:0.5) [line width = 1pt][dotted];
\draw[ red][rotate=   0] (-1.21, 0.50) arc(  90: 270:0.5) [line width = 1pt][dotted];
\draw[ red][rotate=   0] ( 0.00, 2.00) arc(  90: 185:0.5) [line width = 1pt][dotted];
\draw[blue][rotate=   0] ( 0.00, 2.20) arc(  90:  -5:0.7) [line width = 1pt][dotted];
\draw[blue][rotate= 180] (-1.72, 2.45) arc(  45:-135:0.5) [line width = 1pt][dotted];
\draw[blue][rotate= 180] (-1.21, 0.50) arc(  90: 270:0.5) [line width = 1pt][dotted];
\draw[blue][rotate= 180] ( 0.00, 2.00) arc(  90: 185:0.5) [line width = 1pt][dotted];
\draw[ red][rotate= 180] ( 0.00, 2.20) arc(  90:  -5:0.7) [line width = 1pt][dotted];
\draw[green][rotate=   0] (-1.62,-2.25) arc(-25: 135:0.5) [line width = 1pt][dotted];
\draw[green][rotate= 180] (-1.62,-2.25) arc(-25: 135:0.5) [line width = 1pt][dotted];
\draw[green] ( 0.00, 2.10) to[out=0,in=0] ( 0.00, 0.90) [line width = 1pt][dotted];
\end{tikzpicture}
\\
\begin{tikzpicture}[scale=0.8]
\draw ( 0  , 3  ) node{$0$};
\draw ( 0  , 1.5) node{$\kk$};
\draw (-1.2, 0  ) node{$0$};
\draw ( 1.2, 0  ) node{$0$};
\draw ( 0  ,-1.5) node{$\kk^2$};
\draw ( 0  ,-3  ) node{$\kk^2$};
\draw (-2.1, 2.1) node{$0$};
\draw (-3.0, 0  ) node{$0$};
\draw (-2.1,-2.1) node{$\kk$};
\draw ( 2.1, 2.1) node{$0$};
\draw ( 3.0, 0  ) node{$0$};
\draw ( 2.1,-2.1) node{$0$};
\draw[line width = 1pt][->] ( 0  ,-1.7) -- (   0,-2.8);
\draw ( 0  ,-2.2) node[right]{$[{_0^1}\ {_1^0}]$};
\draw[line width = 1pt][->] ( 0  , 1.2) -- (   0,-1.2);
\draw ( 0  , 0  ) node[ left]{$[{_1^0}]$};
\draw[line width = 1pt][->] ( 0  , 2.8) -- (   0, 1.7);
\draw[line width = 1pt][->] (-0.2, 1.5) to[out= 180, in=  90] (-1.2, 0.2);
\draw[line width = 1pt][->] (-1.2,-0.2) to[out= -90, in= 180] (-0.2,-1.5);
\draw[line width = 1pt][->] ( 0.2, 1.5) to[out=   0, in=  90] ( 1.2, 0.2);
\draw[line width = 1pt][->] ( 1.2,-0.2) to[out= -90, in=   0] ( 0.2,-1.5);
\draw[line width = 1pt][->][rotate= 5+  0] (0,3) arc(90:125:3);
\draw[line width = 1pt][->][rotate= 5+ 45] (0,3) arc(90:125:3);
\draw[line width = 1pt][->][rotate= 5+ 90] (0,3) arc(90:125:3);
\draw[line width = 1pt][->][rotate= 5+135] (0,3) arc(90:125:3);
\draw[line width = 1pt][->][rotate=-5-  0] (0,3) arc(90:55:3);
\draw[line width = 1pt][->][rotate=-5- 45] (0,3) arc(90:55:3);
\draw[line width = 1pt][->][rotate=-5- 90] (0,3) arc(90:55:3);
\draw[line width = 1pt][->][rotate=-5-135] (0,3) arc(90:55:3);
\draw[line width = 1pt][->][rotate=-5+  0] (0,2.8) arc(90:55:2.8);
\draw[line width = 1pt][<-][rotate=-5+180] (0,2.8) arc(90:55:2.8);
%
\draw[rotate= 180] ( 1.06, 3.21) node{$[^0_1]$};
\draw[rotate= 180] ( 0.96, 2.35) node{$[_0^1]$};
\draw[ red][rotate=   0] (-1.72, 2.45) arc(  45:-135:0.5) [line width = 1pt][dotted];
\draw[ red][rotate=   0] (-1.21, 0.50) arc(  90: 270:0.5) [line width = 1pt][dotted];
\draw[ red][rotate=   0] ( 0.00, 2.00) arc(  90: 185:0.5) [line width = 1pt][dotted];
\draw[blue][rotate=   0] ( 0.00, 2.20) arc(  90:  -5:0.7) [line width = 1pt][dotted];
\draw[blue][rotate= 180] (-1.72, 2.45) arc(  45:-135:0.5) [line width = 1pt][dotted];
\draw[blue][rotate= 180] (-1.21, 0.50) arc(  90: 270:0.5) [line width = 1pt][dotted];
\draw[blue][rotate= 180] ( 0.00, 2.00) arc(  90: 185:0.5) [line width = 1pt][dotted];
\draw[ red][rotate= 180] ( 0.00, 2.20) arc(  90:  -5:0.7) [line width = 1pt][dotted];
\draw[green][rotate=   0] (-1.62,-2.25) arc(-25: 135:0.5) [line width = 1pt][dotted];
\draw[green][rotate= 180] (-1.62,-2.25) arc(-25: 135:0.5) [line width = 1pt][dotted];
\draw[green] ( 0.00, 2.10) to[out=0,in=0] ( 0.00, 0.90) [line width = 1pt][dotted];
\end{tikzpicture}
\ \
\begin{tikzpicture}[scale=0.8]
\draw ( 0  , 3  ) node{$0$};
\draw ( 0  , 1.5) node{$\kk$};
\draw (-1.2, 0  ) node{$0$};
\draw ( 1.2, 0  ) node{$0$};
\draw ( 0  ,-1.5) node{$\kk$};
\draw ( 0  ,-3  ) node{$\kk$};
\draw (-2.1, 2.1) node{$0$};
\draw (-3.0, 0  ) node{$0$};
\draw (-2.1,-2.1) node{$0$};
\draw ( 2.1, 2.1) node{$0$};
\draw ( 3.0, 0  ) node{$0$};
\draw ( 2.1,-2.1) node{$0$};
\draw[line width = 1pt][->] ( 0  ,-1.7) -- (   0,-2.8);
\draw ( 0  ,-2.2) node[right]{$1$};
\draw[line width = 1pt][->] ( 0  , 1.2) -- (   0,-1.2);
\draw ( 0  , 0  ) node[ left]{$1$};
\draw[line width = 1pt][->] ( 0  , 2.8) -- (   0, 1.7);
\draw[line width = 1pt][->] (-0.2, 1.5) to[out= 180, in=  90] (-1.2, 0.2);
\draw[line width = 1pt][->] (-1.2,-0.2) to[out= -90, in= 180] (-0.2,-1.5);
\draw[line width = 1pt][->] ( 0.2, 1.5) to[out=   0, in=  90] ( 1.2, 0.2);
\draw[line width = 1pt][->] ( 1.2,-0.2) to[out= -90, in=   0] ( 0.2,-1.5);
\draw[line width = 1pt][->][rotate= 5+  0] (0,3) arc(90:125:3);
\draw[line width = 1pt][->][rotate= 5+ 45] (0,3) arc(90:125:3);
\draw[line width = 1pt][->][rotate= 5+ 90] (0,3) arc(90:125:3);
\draw[line width = 1pt][->][rotate= 5+135] (0,3) arc(90:125:3);
\draw[line width = 1pt][->][rotate=-5-  0] (0,3) arc(90:55:3);
\draw[line width = 1pt][->][rotate=-5- 45] (0,3) arc(90:55:3);
\draw[line width = 1pt][->][rotate=-5- 90] (0,3) arc(90:55:3);
\draw[line width = 1pt][->][rotate=-5-135] (0,3) arc(90:55:3);
\draw[line width = 1pt][->][rotate=-5+  0] (0,2.8) arc(90:55:2.8);
\draw[line width = 1pt][<-][rotate=-5+180] (0,2.8) arc(90:55:2.8);
%
\draw[ red][rotate=   0] (-1.72, 2.45) arc(  45:-135:0.5) [line width = 1pt][dotted];
\draw[ red][rotate=   0] (-1.21, 0.50) arc(  90: 270:0.5) [line width = 1pt][dotted];
\draw[ red][rotate=   0] ( 0.00, 2.00) arc(  90: 185:0.5) [line width = 1pt][dotted];
\draw[blue][rotate=   0] ( 0.00, 2.20) arc(  90:  -5:0.7) [line width = 1pt][dotted];
\draw[blue][rotate= 180] (-1.72, 2.45) arc(  45:-135:0.5) [line width = 1pt][dotted];
\draw[blue][rotate= 180] (-1.21, 0.50) arc(  90: 270:0.5) [line width = 1pt][dotted];
\draw[blue][rotate= 180] ( 0.00, 2.00) arc(  90: 185:0.5) [line width = 1pt][dotted];
\draw[ red][rotate= 180] ( 0.00, 2.20) arc(  90:  -5:0.7) [line width = 1pt][dotted];
\draw[green][rotate=   0] (-1.62,-2.25) arc(-25: 135:0.5) [line width = 1pt][dotted];
\draw[green][rotate= 180] (-1.62,-2.25) arc(-25: 135:0.5) [line width = 1pt][dotted];
\draw[green] ( 0.00, 2.10) to[out=0,in=0] ( 0.00, 0.90) [line width = 1pt][dotted];
\end{tikzpicture}
\caption{String modules}
\label{fig:str mods}
\end{figure}
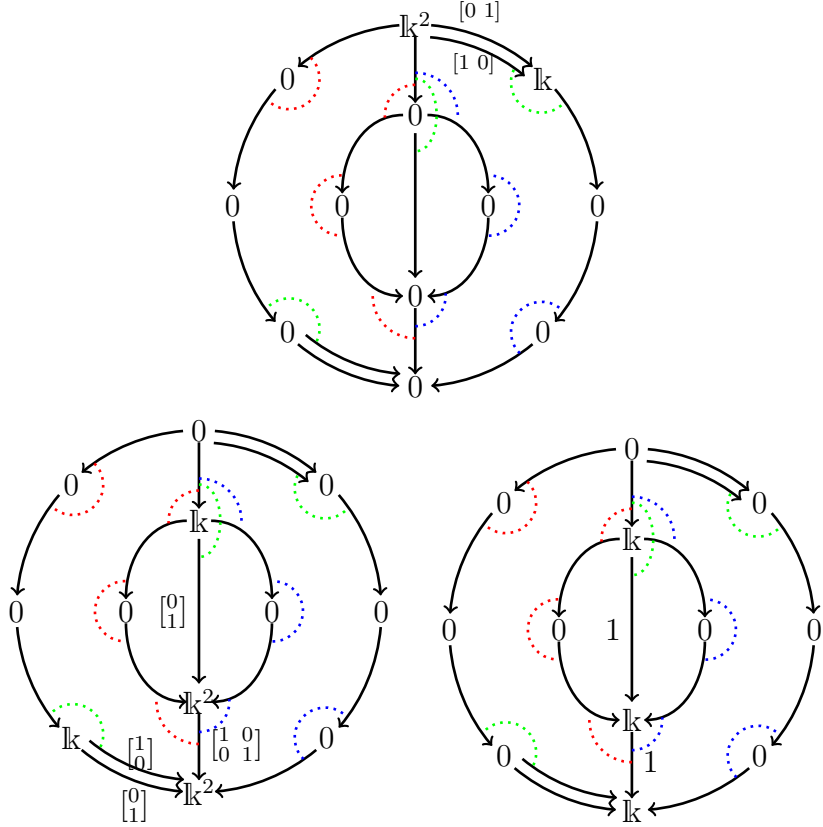

The string $s_1$ with a Jordan block $\pmb{J}_n(\lambda)$ ($\lambda \ne 0$), i.e., the pair $(s_1, \pmb{J}_n(\lambda)$, is a band,
which describe the band module $\M(s_1, \pmb{J}_n(\lambda))$ shown in \Pic \ref{fig:band mod}.
\begin{figure}[htbp]
\begin{tikzpicture}[scale=0.8]
\draw ( 0  , 3  ) node{$\kk^n$};
\draw ( 0  , 1.5) node{$0$};
\draw (-1.2, 0  ) node{$0$};
\draw ( 1.2, 0  ) node{$0$};
\draw ( 0  ,-1.5) node{$0$};
\draw ( 0  ,-3  ) node{$0$};
\draw (-2.1, 2.1) node{$0$};
\draw (-3.0, 0  ) node{$0$};
\draw (-2.1,-2.1) node{$0$};
\draw ( 2.1, 2.1) node{$\kk^n$};
\draw ( 3.0, 0  ) node{$0$};
\draw ( 2.1,-2.1) node{$0$};
\draw[line width = 1pt][->] ( 0  ,-1.7) -- (   0,-2.8);
\draw[line width = 1pt][->] ( 0  , 1.2) -- (   0,-1.2);
\draw[line width = 1pt][->] ( 0  , 2.8) -- (   0, 1.7);
\draw[line width = 1pt][->] (-0.2, 1.5) to[out= 180, in=  90] (-1.2, 0.2);
\draw[line width = 1pt][->] (-1.2,-0.2) to[out= -90, in= 180] (-0.2,-1.5);
\draw[line width = 1pt][->] ( 0.2, 1.5) to[out=   0, in=  90] ( 1.2, 0.2);
\draw[line width = 1pt][->] ( 1.2,-0.2) to[out= -90, in=   0] ( 0.2,-1.5);
\draw[line width = 1pt][->][rotate= 5+  0] (0,3) arc(90:125:3);
\draw[line width = 1pt][->][rotate= 5+ 45] (0,3) arc(90:125:3);
\draw[line width = 1pt][->][rotate= 5+ 90] (0,3) arc(90:125:3);
\draw[line width = 1pt][->][rotate= 5+135] (0,3) arc(90:125:3);
\draw[line width = 1pt][->][rotate=-5-  0] (0,3) arc(90:55:3);
\draw[line width = 1pt][->][rotate=-5- 45] (0,3) arc(90:55:3);
\draw[line width = 1pt][->][rotate=-5- 90] (0,3) arc(90:55:3);
\draw[line width = 1pt][->][rotate=-5-135] (0,3) arc(90:55:3);
\draw[line width = 1pt][->][rotate=-5+  0] (0,2.8) arc(90:55:2.8);
\draw[line width = 1pt][<-][rotate=-5+180] (0,2.8) arc(90:55:2.8);
\draw[rotate=   0] ( 1.06, 3.21) node{$\mathrm{id}$};
\draw[rotate=   0] ( 0.96, 2.35) node{\tiny$\pmb{J}_n(\lambda)$};
\draw[ red][rotate=   0] (-1.72, 2.45) arc(  45:-135:0.5) [line width = 1pt][dotted];
\draw[ red][rotate=   0] (-1.21, 0.50) arc(  90: 270:0.5) [line width = 1pt][dotted];
\draw[ red][rotate=   0] ( 0.00, 2.00) arc(  90: 185:0.5) [line width = 1pt][dotted];
\draw[blue][rotate=   0] ( 0.00, 2.20) arc(  90:  -5:0.7) [line width = 1pt][dotted];
\draw[blue][rotate= 180] (-1.72, 2.45) arc(  45:-135:0.5) [line width = 1pt][dotted];
\draw[blue][rotate= 180] (-1.21, 0.50) arc(  90: 270:0.5) [line width = 1pt][dotted];
\draw[blue][rotate= 180] ( 0.00, 2.00) arc(  90: 185:0.5) [line width = 1pt][dotted];
\draw[ red][rotate= 180] ( 0.00, 2.20) arc(  90:  -5:0.7) [line width = 1pt][dotted];
\draw[green][rotate=   0] (-1.62,-2.25) arc(-25: 135:0.5) [line width = 1pt][dotted];
\draw[green][rotate= 180] (-1.62,-2.25) arc(-25: 135:0.5) [line width = 1pt][dotted];
\draw[green] ( 0.00, 2.10) to[out=0,in=0] ( 0.00, 0.90) [line width = 1pt][dotted];
\end{tikzpicture}
\caption{A band module}
\label{fig:band mod}
\end{figure}
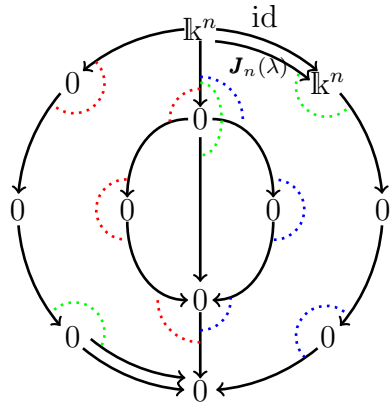

All indecomposable modules over a string algebra can be divided into two classes, one is a collection of
all string modules and the other one is a collection of all band modules. But, it does not hold true
over almost gentle algebras in general, see the indecomposable module shown in \Pic \ref{fig:ind mod}.
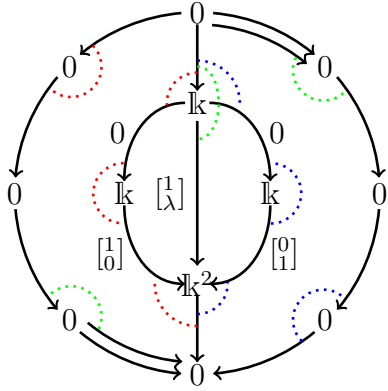
\begin{figure}[H]
\centering
\begin{tikzpicture}[scale=0.8]
\draw ( 0  , 3  ) node{$0$};
\draw ( 0  , 1.5) node{$\kk$};
\draw (-1.2, 0  ) node{$\kk$};
\draw ( 1.2, 0  ) node{$\kk$};
\draw ( 0  ,-1.5) node{$\kk^2$};
\draw ( 0  ,-3  ) node{$0$};
\draw (-2.1, 2.1) node{$0$};
\draw (-3.0, 0  ) node{$0$};
\draw (-2.1,-2.1) node{$0$};
\draw ( 2.1, 2.1) node{$0$};
\draw ( 3.0, 0  ) node{$0$};
\draw ( 2.1,-2.1) node{$0$};
\draw[line width = 1pt][->] ( 0  ,-1.7) -- (   0,-2.8);
\draw[line width = 1pt][->] ( 0  , 1.2) -- (   0,-1.2);
\draw ( 0  , 0  ) node[ left]{$[^1_{\lambda}]$};
\draw[line width = 1pt][->] ( 0  , 2.8) -- (   0, 1.7);
\draw[line width = 1pt][->] (-0.2, 1.5) to[out= 180, in=  90] (-1.2, 0.2);
\draw[line width = 1pt][->] (-1.2,-0.2) to[out= -90, in= 180] (-0.2,-1.5);
\draw[line width = 1pt][->] ( 0.2, 1.5) to[out=   0, in=  90] ( 1.2, 0.2);
\draw[line width = 1pt][->] ( 1.2,-0.2) to[out= -90, in=   0] ( 0.2,-1.5);
\draw (-1. , 1.0) node[ left]{$0$};
\draw ( 1. , 1.0) node[right]{$0$};
\draw (-1. ,-1.0) node[ left]{$[^1_0]$};
\draw ( 1. ,-1.0) node[right]{$[^0_1]$};
\draw[line width = 1pt][->][rotate= 5+  0] (0,3) arc(90:125:3);
\draw[line width = 1pt][->][rotate= 5+ 45] (0,3) arc(90:125:3);
\draw[line width = 1pt][->][rotate= 5+ 90] (0,3) arc(90:125:3);
\draw[line width = 1pt][->][rotate= 5+135] (0,3) arc(90:125:3);
\draw[line width = 1pt][->][rotate=-5-  0] (0,3) arc(90:55:3);
\draw[line width = 1pt][->][rotate=-5- 45] (0,3) arc(90:55:3);
\draw[line width = 1pt][->][rotate=-5- 90] (0,3) arc(90:55:3);
\draw[line width = 1pt][->][rotate=-5-135] (0,3) arc(90:55:3);
\draw[line width = 1pt][->][rotate=-5+  0] (0,2.8) arc(90:55:2.8);
\draw[line width = 1pt][<-][rotate=-5+180] (0,2.8) arc(90:55:2.8);
%
\draw[ red][rotate=   0] (-1.72, 2.45) arc(  45:-135:0.5) [line width = 1pt][dotted];
\draw[ red][rotate=   0] (-1.21, 0.50) arc(  90: 270:0.5) [line width = 1pt][dotted];
\draw[ red][rotate=   0] ( 0.00, 2.00) arc(  90: 185:0.5) [line width = 1pt][dotted];
\draw[blue][rotate=   0] ( 0.00, 2.20) arc(  90:  -5:0.7) [line width = 1pt][dotted];
\draw[blue][rotate= 180] (-1.72, 2.45) arc(  45:-135:0.5) [line width = 1pt][dotted];
\draw[blue][rotate= 180] (-1.21, 0.50) arc(  90: 270:0.5) [line width = 1pt][dotted];
\draw[blue][rotate= 180] ( 0.00, 2.00) arc(  90: 185:0.5) [line width = 1pt][dotted];
\draw[ red][rotate= 180] ( 0.00, 2.20) arc(  90:  -5:0.7) [line width = 1pt][dotted];
\draw[green][rotate=   0] (-1.62,-2.25) arc(-25: 135:0.5) [line width = 1pt][dotted];
\draw[green][rotate= 180] (-1.62,-2.25) arc(-25: 135:0.5) [line width = 1pt][dotted];
\draw[green] ( 0.00, 2.10) to[out=0,in=0] ( 0.00, 0.90) [line width = 1pt][dotted];
\end{tikzpicture}
\caption{An indecomposable module which is neither string nor band ($\lambda \ne 0$)}
\label{fig:ind mod}
\end{figure}
\end{example}

\section{Global dimension of almost gentle algebras}

In this section, we provide a description of the global dimensions of almost gentle algebras.
To do this, we need to compute the projective resolution of any simple module.

\subsection{Syzygies of directed string modules} \label{Sect:Syzygy(ds)}


\begin{definition}\label{def:claw} \rm \
\begin{itemize}
\item[(1)]
A {\defines claw} is a sequence of some directed strings $s_1,\cdots,s_n$, written as
\[s_1\claw s_2\claw\cdots\claw s_n,\]
whose sources coincide (cf. \Pic \ref{fig:claw}).
For simplicity, we use the symbol $s_1\claw s_1 \claw \cdots \claw s_n$ to emphasize that $s_1, \ldots, s_n$ have the same sources.

\item[(2)]
Dually, we can define that an {\defines anti-claw} is a sequence of some directed strings $s_1,\cdots,s_n$, written as
\[s_1\aclaw s_2\aclaw\cdots\aclaw s_n, \]
whose sinks coincide (cf. \Pic \ref{fig:aclaw}).
For simplicity, we use the symbol $s_1\aclaw s_2 \aclaw \cdots \aclaw s_n$ to emphasize that $s_1, \ldots, s_n$ have the same sink.
\end{itemize}
\end{definition}

\begin{figure}[htbp]
\centering
\begin{tikzpicture}
\draw (0,0) node{
\xymatrix@C=1.5cm{
& & v_0
 \ar@{~>}[lld]_{s_1}
 \ar@{~>}[ld]^{s_2}
 \ar@{~>}[rd]_{s_{n-1}}
 \ar@{~>}[rrd]^{s_n}
& & \\
v_1 & v_2 & \cdots & v_{n-1} & v_n
 }
};
\end{tikzpicture}
\caption{The claw $s_1\claw s_2\claw\cdots\claw s_n$ given by  directed strings $s_1, s_2, \ldots, s_n$}
\label{fig:claw}
\end{figure}

\begin{figure}[htbp]
\centering
\begin{tikzpicture}
\draw (0,0) node{
\xymatrix@C=1.5cm{
v_1 & v_2 & \cdots & v_{n-1} & v_n \\
& & v_0
 \ar@{<~}[llu]^{s_1}
 \ar@{<~}[lu]_{s_2}
 \ar@{<~}[ru]^{s_{n-1}}
 \ar@{<~}[rru]_{s_n}
& &
 }
};
\end{tikzpicture}
\caption{The anti-claw $s_1\aclaw s_2\aclaw\cdots\aclaw s_n$ given by directed strings $s_1, s_2, \ldots, s_n$}
\label{fig:aclaw}
\end{figure}

From the definition of the almost gentle algebras,
it is obvious that claws and anti-claws can be used to describe the indecomposable projective and injective modules, respectively.
For any module $M$ and $n\geqslant 1$, we use $\Omega_n(M)$ to denote the $n$-th syzygy of $M$, in particular, we write $\Omega_0(M)=M$.

\begin{lemma} \label{lemm:1st-syzygy}
Let $A$ be an almost gentle algebra. Take $M$ to be a directed string module given by some directed string $\ds$.
Then $\Omega_1(M)$ is a direct sum of some directed string modules, that is,
\begin{align}\label{formula:1st-syzygy}
  \Omega_1(M)= \Omega_1(\M(\ds)) \cong \bigoplus_{i=1}^r \M(s_i)
\end{align}
for some directed strings $s_1, \ldots, s_r$.
\end{lemma}

\begin{proof}
First of all, assume that the source of $\ds$ is $v_0$. Then the top of $\M(\ds)$ is isomorphic to the simple module $S(v_0)$.
It follows that the projective cover of $\M(\ds)$ is isomorphic to the indecomposable projective module $P(v_0)$.
Now, let $s_1'\claw s_2' \claw\cdots \claw s_n'$ be the claw corresponding to $P(v_0)$,
and, for each $1\leqslant t\leqslant n$, suppose $s_t' = a_{t,1}a_{t,2}\cdots a_{t,m_t}$ ($a_{t,1}$, $a_{t,2}$, $\ldots$, $a_{t,m_t}$ are arrows).
Then there are two integers $1\leqslant j\leqslant n$ and $1\leqslant l\leqslant m_j$ such that $\ds = a_{j,1}a_{j,2}\cdots a_{j,l}$.
One can check that
\[\Omega_1(M) \cong \M(a_{j,l+1}\cdots a_{j,m_j}) \oplus
  \bigoplus_{
  \begin{smallmatrix}
  1\leqslant t\leqslant n\\
  t\ne j
  \end{smallmatrix}
  } \M(a_{t,2}\cdots a_{t,m_t}). \]
Cf. \Pic \ref{fig:projcover}.
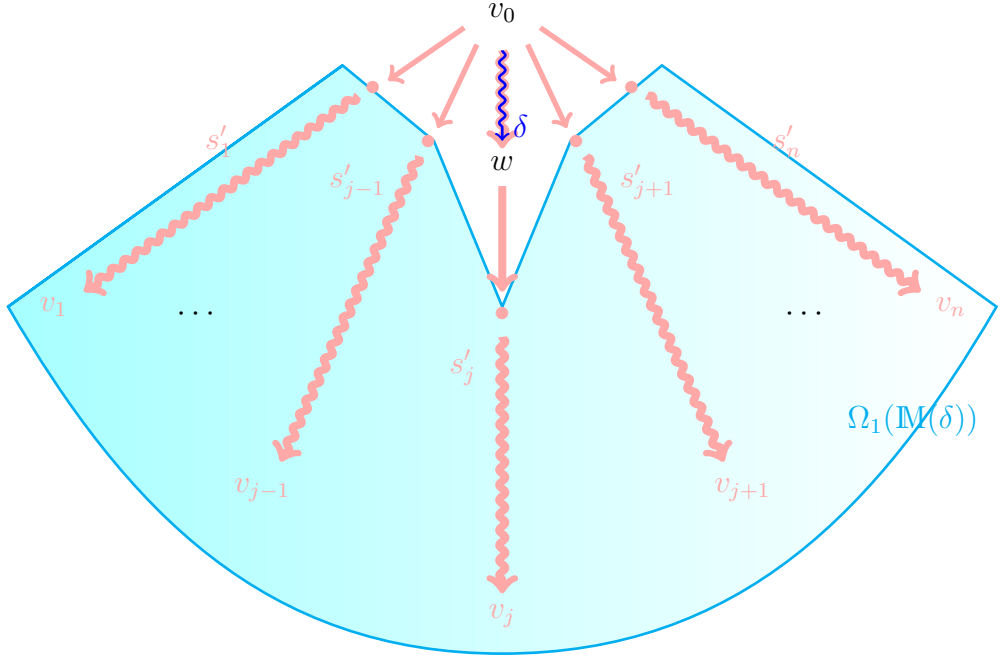
\begin{figure}[htbp]
\centering
\begin{tikzpicture}[scale=2]
\fill[left color=cyan!35, right color=white]
  (-3.25,-1.95) -- (-1.05,-0.35) -- (-0.45,-0.85) -- (0,-1.95)
  -- ( 0.45,-0.85) -- ( 1.05,-0.35) -- ( 3.25,-1.95)
  to[out=-120,in=0] (0,-4.25) to[out=180,in=-60] (-3.25,-1.95);
\draw[cyan][line width=1pt]
  (-3.25,-1.95) -- (-1.05,-0.35) -- (-0.45,-0.85) -- (0,-1.95)
  -- ( 0.45,-0.85) -- ( 1.05,-0.35) -- ( 3.25,-1.95)
  to[out=-120,in=0] (0,-4.25) to[out=180,in=-60] (-3.25,-1.95)
  -- (-1.05,-0.35);
\draw[cyan] (2.7,-2.7) node{$\Omega_1(\M(\ds))$};
\draw[->][line width=3pt][red!34][decorate, decoration
      ={snake,amplitude=.4mm,segment length=2mm,post length=1mm}] (0,-0.25) -- (0,-0.90);
\draw[->][line width=3pt][red!34] (0,-1.15) -- (0,-1.85);
\draw[red!34] (0,-2) node{$\bullet$};
\draw[->][line width=3pt][red!34][decorate, decoration
      ={snake,amplitude=.4mm,segment length=2mm,post length=1mm}] (0,-2.15) -- (0,-3.85);
\draw[red!34] (0,-4) node{$v_j$};
\draw[->][line width=2pt][red!34] (-0.25,-0.1) -- (-0.75,-0.45);
\draw[red!34] (-0.85,-0.50) node{$\bullet$};
\draw[->][line width=3pt][red!34][decorate, decoration
      ={snake,amplitude=.4mm,segment length=2mm,post length=1mm}] (-0.95,-0.55) -- (-2.75,-1.85);
\draw[red!34] (-2.95,-1.95) node{$v_1$};
\draw[rotate=30][->][line width=2pt][red!34] (-0.25,-0.1) -- (-0.75,-0.45);
\draw[rotate=30][red!34] (-0.85,-0.50) node{$\bullet$};
\draw[rotate=30][->][line width=3pt][red!34][decorate, decoration
      ={snake,amplitude=.4mm,segment length=2mm,post length=1mm}] (-0.95,-0.55) -- (-2.75,-1.85);
\draw[rotate=30][red!34] (-2.95,-1.95) node{$v_{j-1}$};
\draw[rotate=-30][->][line width=2pt][red!34] ( 0.25,-0.1) -- ( 0.75,-0.45);
\draw[rotate=-30][red!34] ( 0.85,-0.50) node{$\bullet$};
\draw[rotate=-30][->][line width=3pt][red!34][decorate, decoration
      ={snake,amplitude=.4mm,segment length=2mm,post length=1mm}] ( 0.95,-0.55) -- ( 2.75,-1.85);
\draw[rotate=-30][red!34] ( 2.95,-1.95) node{$v_{j+1}$};
\draw[->][line width=2pt][red!34] ( 0.25,-0.1) -- ( 0.75,-0.45);
\draw[red!34] ( 0.85,-0.50) node{$\bullet$};
\draw[->][line width=3pt][red!34][decorate, decoration
      ={snake,amplitude=.4mm,segment length=2mm,post length=1mm}] ( 0.95,-0.55) -- ( 2.75,-1.85);
\draw[red!34] ( 2.95,-1.95) node{$v_{n}$};
\draw[->][line width=1pt][blue][decorate, decoration
      ={snake,amplitude=.4mm,segment length=2mm,post length=1mm}] (0,-0.25) -- (0,-0.85);
\draw[blue] (0,-0.75) node[right]{$\ds$};
\draw ( 0, 0) node{$v_0$};
\draw ( 0,-1) node{$w$};
\draw (-2,-2) node{$\cdots$};
\draw ( 2,-2) node{$\cdots$};
\draw[red!34] (-1.7, -0.85) node[ left]{$s_1'$};
\draw[red!34] (-0.7, -1.12) node[ left]{$s_{j-1}'$};
\draw[red!34] (-0.1, -2.35) node[ left]{$s_{j}'$};
\draw[red!34] ( 0.7, -1.12) node[right]{$s_{j+1}'$};
\draw[red!34] ( 1.7, -0.85) node[right]{$s_n'$};
\end{tikzpicture}
\caption{The projective cover $P(v_0)$ and $1$-st syzygy $\Omega_1(\M(\ds))$ of $\M(\ds)$}
\label{fig:projcover}
\end{figure}
In this case, it is easy to see that each indecomposable direct summand of $\Omega_1(\M(\ds))$ is a directed string module.
\end{proof}

\begin{lemma}
\label{lemm:nth-syzygy}
Keep the notations in Lemma \ref{lemm:1st-syzygy}.
Then, for each $n\geqslant 1$, the $n$-th syzygy $\Omega_n(M)$ of $M$ is a direct sum of some directed string modules.
\end{lemma}

\begin{proof}
By Lemma \ref{lemm:1st-syzygy}, we have that $\Omega_1(M)$ is a direct sum of some directed string modules.
It is clear that if $\Omega_{n}(M)$ is a direct sum of some directed string modules, then so is
$\Omega_{n+1}(M)=\Omega_{1}(\Omega_{n}(M))$, thus the assertion follows.
\end{proof}


Recall that a {\defines forbidden path} (of length $n\geqslant 1$) on a gentle pair $(\Q,\I)$ is path $a_1\cdots a_n$
such that $a_{i}a_{i+1} \in \I$ holds for all $1\leqslant i \leqslant n-1$.
Sometimes an oriented cycle $\C=a_1a_2\cdots a_n$ with $a_1a_2, \ldots, a_{n-1}a_n, a_na_1 \in \I$ is called
a {\defines forbidden cycle} or a {\defines full relational orientated cycle} \cite[Section 2.2]{AG2008}.
Naturally, we can define the forbidden path on any almost gentle pair in the same way.
A {\defines forbidden path} of length zero is a path $e_v$ corresponding to a vertex $v\in\Q_0$ such that one of the following
conditions is satisfied:
\begin{itemize}
\item[(1)] a path $e_v$ of length zero corresponding to $v\in\Q_0$ such that
\begin{itemize}
  \item there is a unique arrow $a$ with $\target(a)=v$ and there is a unique arrow $b$ with $\source(a)=v$;
  \item $ab\in\I$;
\end{itemize}
\item[(2)] a path $e_v$ of length zero corresponding to a source $v\in\Q_0$
  such that $\{a\in\Q_1 \mid \source(a)=v\}$ contains only one element;
\item[(3)] a path $e_v$ of length zero corresponding to a sink $v\in\Q_0$
  such that $\{a\in\Q_1 \mid \target(a)=v\}$ contains only one element.
\end{itemize}

In \cite{AG2008}, forbidden paths were used to compute the so-called AG-invariants of gentle algebras
which describe the derived equivalence of gentle one-cycle algebras.
In \cite{LZ2019, LZ2021}, the standard forms of gentle one-cycle algebras and their geometric models by using forbidden paths and AG-invariants
were provided.

Let $\alpha$ be an arrow and $v$ the vertex $\target(\alpha)$ (resp. $\source(\alpha)$).
A vertex $v\in\Q_0$ is called an {\defines $(\alpha, \bfda)$-relational vertex} (resp. {\defines $(\alpha, \bfua)$-relational vertex}) if there is an arrow $\beta$ with $\source(\beta)=v$ (resp. $\target(\beta)=v$) such that $\alpha\beta\in\I$ (resp. $\beta\alpha\in\I$).
A vertex $v\in\Q_0$ is called a {\defines relational vertex} if it is either $(\alpha, \bfda)$-relational or $(\alpha', \bfua)$-relational for some arrows $\alpha$ and $\alpha'$,
that is, there are two arrows $a$ and $b$ with $\target(a)=v=\source(b)$ such that $ab\in\I$.

\begin{example} \rm
Consider the bound quiver $(\Q,\I)$ given in Example \ref{exp:almost gent}.
The vertex $2$ is an $(a_{1,2},\bfda)$-relational vertex because $a_{2,3}$, $a_{2,4}$, and $a_{2,3'}$ are arrows starting at $2$ such that $a_{1,2}a_{2,3}$, $a_{1,2}a_{2,4}$, and $a_{1,2}a_{2,3'}$ are belong to $\I$.
Moreover, $2$ is also an $(a_{2,4},\bfua)$-relational vertex by $a_{1,2}a_{2,4} \in \I$.
Similarly, $2_{\Right}$ is both a $(b_{1,2_{\Right}}, \bfda)$-relational vertex and an $(a_{2_{\Right},3_{\Right}},\bfua)$-relational vertex,
but it is not an $(a_{1,2_{\Right}},\bfda)$-relational vertex since $a_{1,2_{\Right}}a_{2_{\Right},3{\Right}}$ does not lie in $\I$.
\end{example}

Clearly, in the case for $(\Q,\I)$ to be an almost gentle pair and $\Q$ to be a quiver of type $\mathbb{A}$,
the path $e_v$ corresponding to a vertex $v\in\Q_0$ is a forbidden path if and only if $v$ is a relational vertex.
However, in the case for $(\Q,\I)$ to be an almost gentle pair, $v$ to be relational does not admit that $e_v$ is forbidden.

The following result can be shown by using Proposition \ref{lemm:nth-syzygy}.

\begin{proposition}\label{prop:nth-syzygy}
Keep the notations in Lemma \ref{lemm:1st-syzygy}.
Then, for each $n\geqslant 1$, every indecomposable direct summand of $\Omega_n(M)$ is one of the following
\begin{itemize}
\item[\rm(1)] a right maximal directed string module;
\item[\rm(2)] a simple module corresponding to some relational vertex;
\item[\rm(3)] a simple module corresponding to some sink.
\end{itemize}
\end{proposition}

\begin{proof}
Assume that $D \le_{\oplus} \Omega_n(M)$ is an indecomposable direct summand.
Then, by Proposition \ref{lemm:nth-syzygy}, we have $D\cong \M(s)$ for some directed string.
If $s = a_1a_2\cdots a_l$ is not right maximal ($a_1, \ldots, a_l$ are arrows, $l\geqslant 0$,
and we have $s=e_v$ for some $v\in\Q_0$ and $D\cong\M(s)\cong S(v)$ in the case for $l=0$),
then there is an arrow $\alpha$ such that $s\alpha = a_1a_2\cdots a_l\alpha$, as a path on $\Q$, does not belong to $\I$.
Now, assume that $s\alpha_1\cdots\alpha_m =  a_1a_2\cdots a_l\alpha_1\alpha_2\cdots\alpha_m$ is right maximal. Then we can choose these arrows
$\alpha_1\ldots\alpha_m\in \Q_1$ such that $s\alpha_1\ldots \alpha_m$ is a right maximal directed string
since $A$ is an almost gentle algebra, there is no oriented cycle $\alpha_1\alpha_2\ldots\alpha_m$ over $\Q$.
Then the projective cover
\[\tilde{p}_n: P(\Omega_{n-1}(M)) \to \Omega_{n-1}(M)\]
of $\Omega_{n-1}(M)$ satisfies the following conditions:
\begin{itemize}
\item $P(\Omega_{n-1}(M))$ has an indecomposable projective direct summand $P(v)$ such that $v$
is a source of some right maximal directed string $s'$ which is of the form
\[s' = b_1b_2\cdots b_k s  \alpha_1\alpha_2\cdots\alpha_m
     = b_1b_2\cdots b_ka_1a_2\cdots a_l\alpha_1\alpha_2\cdots\alpha_m; \]

\item $P(v)$ can be described as the claw $s_1'\claw s_2'\claw \cdots\claw s_r'$ such that
$s'=s_j'$ for some $1\leqslant j\leqslant r$.
\end{itemize}
In this case, we have $D \cong \M(a_1a_2\cdots a_l\alpha_1\alpha_2\cdots\alpha_m) \le_{\oplus} \Ker(\tilde{p}_n)$ ($=\Omega_n(M)$),
which contradicts that $D\cong \M(s)$ in the case of $l\geqslant 1$.

Therefore, we have $l=0$, i.e., $s=e_v$. In this case, $l=m=0$ holds, and then $s' = b_1b_2\cdots b_k$ is right maximal.
By the above argument, we have that $\target(s') = \target(b_k)$ and $v$ coincide.
Thus $v$ is either a relational vertex or a sink.
\end{proof}

\begin{example} \label{exp:proj resol} \rm
Consider the almost gentle algebra $A=\kk\Q/\I$ given in Example \ref{exp:almost gent},
the projective cover of the directed string module $\M(a_{1,2})$ is
\[p_0 : P_0 = P(1) \To{}{} \M(a_{1,2}) = ({_2^1}) \]
whose kernel is given by
\[ \tilde{p}_0: \Ker(p_0) = (2_{\Left}) \oplus (2_{\Right}) \oplus
\left(\begin{smallmatrix}
2_{\Right} \\ 3_{\Right} \\ 4_{\Right}
\end{smallmatrix}\right) \To{}{} P(1), \]
where $ P(1) = \left(\begin{smallmatrix}
  & & 1 & & \\
 2_{\Left} & 2 &  & 2_{\Right} & 2_{\Right} \\
  & & & & 3_{\Right} \\
  & & & & 4_{\Right} \\
\end{smallmatrix}\right) $
is described by the claw
$$\clawnota = a_{1,2_{\Left}} \claw a_{1,2} \claw b_{1,2_{\Right}}
\claw a_{1,2_{\Right}}a_{2_{\Right},3_{\Right}}a_{3_{\Right},4_{\Right}}, $$
and $\Omega_1(\M(a_{1,2})) = \left(\begin{smallmatrix}
2_{\Right} \\ 3_{\Right} \\ 4_{\Right}
\end{smallmatrix}\right)
\cong \M(e_{2_{\Right}}) \oplus \M(e_2) \oplus \M(a_{2_{\Right},3_{\Right}} a_{3_{\Right},4_{\Right}})$.
Each indecomposable direct summand of $\Omega_1(\M(a_{1,2}))$ is a directed string module.
It is clear that $\M(e_{2_{\Right}}) \cong S(2_{\Right})$ and $\M(e_2) \cong S(2)$ are simple and $\M(a_{2_{\Right},3_{\Right}} a_{3_{\Right},4_{\Right}})$ is a right maximal directed string module.

Furthermore, we get the minimal projective resolution of $\M(a_{1,2})$ as follows:
\begin{align*}
 0 & \To{}{} P_2 = P(3_{\Left}) \oplus P(3_{\Right})
\To{p_2}{} P_1 = P(2_{\Left})\oplus P(2_{\Right}) \oplus P(2_{\Right}) \\
   & \To{p_1}{} P_0=P(1) \To{p_0}{} \M(a_{1,2}) = (_2^1) \To{}{} 0,
\end{align*}
where $\Omega_2(\M(a_{1,2})) \cong
(^{3_{\Left}}_{4_{\Left}}) \oplus (^{3_{\Right}}_{4_{\Right}}) = P(3_{\Left}) \oplus P(4_{\Right}) \cong \M(a_{3_{\Left},4_{\Left}}) \oplus \M(a_{3_{\Right},4_{\Right}})$ is a direct sum of two right maximal directed string modules.
\end{example}

\subsection{Projective dimension of right maximal directed string modules}

For any vertex $v$ of an almost gentle pair $(\Q,\I)$, we define
\[ v^{\da} := \{ w \in \Q_0 \mid \text{there is an arrow } a\in \Q_1 \text{ such that } \source(a) = v \text{ and } \target(a) = w \}. \]
Let $\clawnota = s_1\claw s_2\claw \cdots \claw s_n$ be a claw,
where $s_i = a_{i,1}a_{i,2}\cdots a_{i,l_i}$ ($l_1 \geqslant 1$ and $a_{i,1}$, $\ldots$, $a_{i,l_i}$ are arrows).
A vertex $\target(a_{i,j})$ ($1\leqslant j\leqslant l_i$) is said to be a {\defines $\rel{\clawnota}{i}$-relational vertex}
if there is an arrow $\alpha$ with $\source(\alpha)=\target(a_{i,j})$ such that $a_{i,j}\alpha\in\I$.
All $\rel{\clawnota}{i}$-relational vertices are called {\defines $\clawnota$-relational vertices} for simplicity.

\begin{lemma} \label{lemm:pdim dire str mod}
Let $M$ be an indecomposable module over an almost gentle algebra $A$
corresponding to a right maximal directed string $\ds =a_1a_2\cdots a_l$
{\rm(}$l\geqslant 1$, and $a_1$, $a_2$, $\ldots$, $a_l$ are arrows{\rm)}.
Then $\Omega_1(M)$ is projective if and only if all vertices in $\source(\ds)^{\da} \backslash \{\target(a_1)\}$ are not $\clawnota$-relational,
where $\clawnota$ is the claw corresponding to $P(\source(\ds))$.
\end{lemma}

\begin{proof}
By the definition of almost gentle algebras, the indecomposable projective module
$P_0$ ($\cong P(\source(\ds))$) given by the projective cover $p_0: P_0 \to M$ of $M$ can be described by the claw
\[ \clawnota = s_1 \claw s_2 \claw \cdots \claw s_m, \]
where $s_j = a_1a_2\cdots a_l = \ds$ for some $1\leqslant j\leqslant m$, and $s_1$, $\ldots$,
$s_{j-1}$, $s_{j+1}$, $\ldots$, $s_m$ are paths of length $\geqslant 1$.
Notice that $s_j$ is right maximal, then we have
\[\Omega_1(M) \cong \bigoplus_{\begin{smallmatrix}
1\le i\le m \\
i\ne j
\end{smallmatrix}} X_i \]
where each $X_i$ is a directed string module, see Lemma \ref{lemm:1st-syzygy}.
For simplicity, assume $s_i = a_{i,1}\cdots a_{i, t_i}$ for all $1\leqslant i\leqslant m$, $i\ne j$,
here $t_i \geqslant 1$. Then we obtain
\begin{center}
$ X_i \cong \M(a_{i,2}\cdots a_{i,t_i})$

(note that $X_i$ is isomorphic to $S(\target(a_{1,i}))$ in the case for $t_i=1$).
\end{center}

Assume that $\Omega_1(M)$ is projective. If there exists a vertex $\target(a_{i,1})$ in $\source(\ds)^{\da}$
which is $\clawnota$-relational, then we have
$t_i = 1$, $X_i \cong S(\target(a_{1,i}))$,
and there is a right maximal directed string $s'$ of length $\geqslant 1$ with $\source(s') = \target(s_i) = \target(a_{i,1})$.
It is easy to see that $S(\target(a_{1,i}))$ $(\le_{\oplus} \Omega_1(M))$  is not projective by the existence of the directed string $s'$ of length $\geqslant 1$.
Thus, $\Omega_1(M)$ is not a projective module, a contradiction.

On the other hand, if all vertices in $\source(\clawnota)^{\da} \backslash \{\target(a_1)\}$ are not $\clawnota$-relational,
then for any $1\leqslant i \leqslant m$, $i\ne j$, one of the following conditions is satisfied:
\begin{itemize}
  \item[(1)] $t_i = 1$, and $\target(s_i)$ ($= \target(a_{i,1})$) is a sink of $(\Q,\I)$;
  \item[(2)] $t_i \geqslant 2$.
\end{itemize}
Thus, $X_i$ is isomorphic to either the simple module $S(\target(s_i))$ corresponding to $\target(s_i)\in\Q_0$
or the right maximal directed string module $\M(s_i')$  with $\source(s_i') = \target(s_i)$.

In the case (1), since $\target(s_i)$ is a sink of $(\Q,\I)$, we obtain that $S(\target(s_i))$ is a simple projective module.

In the case (2), assume $s_i'=b_{i,1}\cdots b_{i,\ell_i}$. Then there is no arrow $b\ne b_{i,1}$ with $\source(b_{i,1})=\source(b) = \target(s_i)$.
Otherwise, by the definition of almost gentle algebras, we have $a_{i,1}b \in \I$,
it follows that $\target(s_i) = \target(a_{i,1})$ is a $\rel{\clawnota}{i}$-relational vertex, a contradiction.

Then, by the right maximality of $s_i'$, we have that $\M(s_i')$ is projective.
Therefore, $X_i$ is a projective module, and so is $\Omega_1(M)$, as required.
\end{proof}

Similarly, one gets the following result.

\begin{lemma}\label{lemm:1st-syzygy of simp}
Let $S(v)$ be a simple module over an almost gentle algebra $A$ corresponding to the vertex $v\in\Q_0$.
Then $\Omega_1(S(v))$ is projective if and only if all vertices in $v^\da$ are not $\clawnota$-relational,
where $\clawnota$ is the claw corresponding to $P(v)$.
\end{lemma}

\begin{example} \rm
Let $A=\kk\Q/\I$ be the almost gentle algebra given in Example \ref{exp:almost gent}.
Keep the notations in Example \ref{exp:proj resol}. We claim that the $2$-nd syzygy
\[\Omega_2(\M(a_{1,2}))
\cong \Omega_1(S(2_{\Left})) \oplus
\Omega_1(S(2_{\Right})) \oplus
\Omega_1\left(
\left(\begin{smallmatrix}
2_{\Right} \\ 3_{\Right} \\ 4_{\Right}
\end{smallmatrix}\right)
\right)\]
of $\M(a_{1,2})$ is projective. 
Let $\clawnota_{11}$, $\clawnota_{12}$, and $\clawnota_{13}$ be the claws corresponding to
the indecomposable projective module $P(2_{\Left})$, $P(2_{\Right})$, and $P(2_{\Right})$, respectively.
For the vertex $2_{\Left}$, we have $2_{\Left}^{\da} = \{3_{\Left}\}$ which is a set containing only one element $3_{\Left}$.
Since $3_{\Left}$ is not a $\clawnota_{11}$-relational vertex, we have that
$\Omega_1(S_{2_{\Left}})$ is projective by Lemma \ref{lemm:1st-syzygy of simp}.
Similarly, $\Omega_1(S(2_{\Right}))$ is projective.
On the other hand, we have $1^{\da} = \{ 2_{\Left}, 2, 2_{\Right} \}$ where $2 = \target(a_{1,2})$.
Then $1^{\da}\backslash\{\target(a_{1,2})\} = \{2_{\Left}, 2_{\Right}\}$.
Notice that $2_{\Left}$ is $\clawnota_0$-relational ($\clawnota_0$ is the claw $\clawnota$ given in Example \ref{exp:proj resol})
and $2_{\Right}$ is not $\clawnota_0$-relational,
thus $\Omega_1(\M(a_{1,2}))$ is not projective by Lemma \ref{lemm:pdim dire str mod}.
In this case, the indecomposable direct summand
$\left(\begin{smallmatrix}
2_{\Right} \\ 3_{\Right} \\ 4_{\Right}
\end{smallmatrix}\right) \le_{\oplus} \Omega_1(\M(a_{1,2}))$
is projective, it follows that $\Omega_1\left(
\left(\begin{smallmatrix}
2_{\Right} \\ 3_{\Right} \\ 4_{\Right}
\end{smallmatrix}\right)
\right)$ is zero.
Therefore, $\Omega_2(\M(a_{1,2}))$ is projective. The claim is proved.
\end{example}

\subsection{Global dimension}

In this subsection, we describe the global dimension of almost gentle algebras.
To do this, we provide a description for the $n$-th syzygy of the simple module to be non-projective.

\begin{proposition} \label{prop:forb-Omega}
Keep the notations in Lemma \ref{lemm:1st-syzygy of simp}.
Then $\Omega_{n-1}(S(v))$ is not projective if and only if there is a forbidden path $p$ of length $n$ starting at $v$.
\end{proposition}

\begin{proof}
Let $p = b_1b_2\cdots b_n$ be a forbidden path on $(\Q,\I)$,
where $v=v_0$, $\source(b_i) = v_{i-1}$ ($1\leqslant i \leqslant n$) and $\target(b_n) = v_n$,
and let $\clawnota_0$ be the claw corresponding to the indecomposable projective module $P(v_0) = P_0$
given by the projective cover $P_0 \to M$ of $M$.
Then $v_1$ is a $\clawnota_0$-relational vertex. By Lemma \ref{lemm:1st-syzygy of simp},
we have that $\Omega_1(S(v))$ is not projective. To be more precisely,
by Proposition \ref{prop:nth-syzygy}, we have that $\Omega_1(S(v))$ has a non-projective indecomposable direct summand $X_1$
which is either a right maximal directed string module or a simple module such that $\top(X_1) \cong S(v_1)$.

Let $\clawnota_1$ be the claw corresponding to the indecomposable projective module $P(v_1)$.
Then $v_2 \in v_1^{\da}$ is a $\clawnota_1$-relational vertex, it follows that $\Omega_1(X_1)$
is not projective by using Lemma \ref{lemm:pdim dire str mod}.
In this case, $\Omega_2(S(v))$ is not projective since $\Omega_1(X_1) \le_{\oplus} \Omega_2(S(v))$,
and $b_1b_2$ is a forbidden path end with $v_2$ corresponding to it since $v_2$ is $\clawnota_1$-relational.
Then, we can fix a family of right maximal directed string modules $X_2$, $\ldots$, $X_n$ such that
$X_t \le_{\oplus} \Omega_t(S(v))$ holds for all $1\leqslant t\leqslant n$ and
$X_1$, $X_2$, $\ldots$, $X_{n-1}$ are not projective by induction.
Thus, $\Omega_{n-1}(S(v))$ is non-projective. Moreover, we obtain a correspondence
\begin{align} \label{formula:forb-Omega}
 \Omega_t(S(v)) \mapsto v_t
\end{align}
such that $\top(X_t)\cong S(v_t)$ holds for all $1\leqslant t\leqslant n$.

On the other hand, for the case for $\Omega_{n-1}(S(v))$ to be non-projective,
we suppose that the lengths of all forbidden paths starting at $v$ are less than or equal to $n-1$.
Consider all right maximal forbidden paths $p_1$, $\ldots$, $p_m$ start at $v$, and assume
\[ p_j = a_{j,1}a_{j,2}\cdots a_{j,l_j}
 = \ \  v_{j,0} \To{a_{j,1}}{} v_{j,1} \To{a_{j,2}}{} \cdots \To{a_{j,l_j}}{} v_{j,l_j} \]
\begin{center}
    ($a_{j,k}a_{j,k+1} \in \I$ holds for all $1\leqslant k < l_j \leqslant n-1$).
\end{center}
Here, a {\defines right maximal forbidden path} is a forbidden path $p=a_1\cdots a_l$
($a_1,\ldots, a_l$ are arrows, and $a_{k}a_{k+1} \in \I$)
such that $pb\notin \I$ holds for all arrows $b\in\Q_1$ with $\target(p)=\source(b)$.
In this case, if $t\leqslant l_j$, then one can check that there is a  module $X_{j,t} \le_{\oplus} \Omega_t(S(v))$
which is either a right maximal directed string or a simple module such that $\top(X_{j,t}) = v_{j,t}$ by Proposition \ref{prop:nth-syzygy}.
By the right maximality of $p_j$, all vertices in $(v_{j,l_j})^{\da}$ are not $\clawnota_{j,l_j}$-relational,
where $\clawnota_{j,l_j}$ is the claw corresponding to $P(v_{j,l_j})$.
Applying Lemma \ref{lemm:pdim dire str mod} to $\Omega_{1}(\Omega_{j-1}(S(v)))$,
all $X_{j,l_j}$ are projective since each indecomposable direct summand of $\Omega_{j-1}(S(v))$
is either a right maximal directed string module or a simple module by Proposition \ref{prop:nth-syzygy} again.
Thus, for $L=\max\{l_j \mid 1\leqslant j\leqslant m\}$,
we obtain that $\bigoplus_{l_j=L} X_{j,l_j}$ is projective which is isomorphic to $\Omega_L(S(v))$.
In the case for $L< n-1$, we have that $\Omega_n(S(v)) = 0$ is projective, a contradiction.
In the case for $L=n-1$, we have that $\Omega_{n-1}(S(v)) = \Omega_L(S(v)) \cong \bigoplus_{l_j=L} X_{j,l_j}$ is projective, a contradiction.
\end{proof}

\begin{notation}\label{nota:F(v)}
For any vertex $v$ of an almost gentle pair $(\Q,\I)$,
we use $\forb(v)$ to denote the set of all forbidden paths starting at $v\in\Q_0$.
\end{notation}

For a module $M$, we use $\pdim M$ to denote the projective dimension of $M$.
We are now in a position to prove the following result.

\begin{theorem}\label{thm:pdim simp}
For any simple module $S(v)$ over an almost gentle algebra $A$, we have
\begin{align}\label{formula:pdim simp}
  \pdim S(v) = \sup_{F \in \forb(v)} \ell(F),
\end{align}
where $\ell(F)$ is the length of $F$.
\end{theorem}

%
%

\begin{proof}
We only prove it in the case for $\pdim S(v) = n$. The proof of the case for $\pdim S(v) = \infty$ is similar.

By Proposition \ref{prop:forb-Omega}, there exists a forbidden path $F_0 \in \forb(v)$
such that $\ell(F_0) = n$ since $\Omega_{n-1}(S(v))$ is not projective.
Notice that if there is a forbidden path $F_0'$ in $\forb(v)$ that satisfies $\ell(F_0')\geqslant n+1$,
then $\Omega_n(S(v))$ is not projective by Proposition \ref{prop:forb-Omega} again. It contradicts that $\pdim S(v)=n$.
Thus, $F_0$ is the forbidden path in $\forb(v)$ such that
\[ \pdim S(v) = \ell(F_0) = \sup_{F\in\forb(v)} \ell(F), \]
as required.
\end{proof}

The following result provides a description of the global dimension of almost gentle algebras.

\begin{theorem}\label{thm:gldim}
Let $A$ be an almost gentle algebra with the bound quiver $(\Q,\I)$. Then
\[\gldim A = \sup_{F \in \forb} \ell(F), \]
where $\forb = \bigcup\limits_{v\in\Q_0}\forb(v)$ is the set of all forbidden paths.
\end{theorem}

\begin{proof}
Notice a forbidden path $F$ starting at a vertex $v$ yields that $\Omega_{\ell(F)}(S(v))$ is non-zero, then we have
\[\gldim \mathit{\Lambda} = \sup_{v\in\Q_0} \pdim S(v)
 = \sup_{v\in\Q_0} \sup_{F \in \forb(v)} \ell(F)
 = \sup_{F \in \bigcup\limits_{v\in\Q_0}\forb(v)} \ell(F)
 = \sup_{F \in \forb} \ell(F),\]
as required.
\end{proof}

\begin{example} \rm
The global dimension of the almost gentle algebra $A=\kk\Q/\I$ given in Example \ref{exp:almost gent} is $4$ which is decided by the forbidden paths
\[ a_{1,2} a_{2,3'}a_{3',4}a_{4,5} \text{ and } a_{1,2}a_{2,3}a_{3,4}a_{4,5}.\]
To be more precisely, the above two forbidden paths show that the projective dimension $\pdim S(1)$ of the simple module $S(1)$ is $4$,
and one can check that $\pdim S(4) \geqslant \pdim S(v)$ holds for all $v\in\Q_0$.
\end{example}

\section{Self-injective dimension of almost gentle algebras}

If $\mathit{\Lambda}=\kk\Q/\I$ is a finite-dimensional algebra, then
\[ \idim \mathit{\Lambda} = \sup_{i\in\Q_0}\pdim (D(\mathit{\Lambda} e_i)), \]
where $D(\mathit{\Lambda} e_i) := \Hom_{\kk}(\mathit{\Lambda} e_i, \kk)$ is the indecomposable injective
$\mathit{\Lambda}$-module, written as $E(i)$, corresponding to the vertex $i\in\Q_0$.
Thus, we can compute the self-injective dimension $\idim \mathit{\Lambda}$ of $\mathit{\Lambda}$ by
using minimal projective resolutions of injective modules.

\subsection{Syzygies of injective modules} \label{subsect:syzygies of E}

An indecomposable injective module of an algebra $A=\kk\Q/\I$ corresponding to $v\in\Q_0$
is isomorphic to $D(Ae_v)$ which describes all paths on the bound quiver $(\Q,\I)$ end at $v$.
Thus, each indecomposable injective module can be described by an anti-claw (see Definition \ref{def:claw}).
In this subsection, we compute the $1$-st syzygy of an indecomposable injective module over an almost gentle algebra.

\begin{lemma} \label{lemm:vert}
Let $(\Q, \I)$ be an almost gentle pair, and let $v\in\Q_0$ be a {\defines $(c^{\inner},d^{\out})$-type vertex},
i.e., the vertex $v$ satisfying $\target^{-1}(v)=c$ and $\source^{-1}(v)=d$.
Then, in this bound quiver $(\Q,\I)$, there is an integer $t$ with $0 \leqslant t \leqslant \min\{c,d\}$
such that the number of all non-zero paths crossing $v$ of length two is $t$.
\end{lemma}

(Here, we say a path $p =\ 1\To{a_1}{} 2 \To{}{} \cdots \To{}{} n \To{a_n}{} n+1 $
crosses a vertex $v$ if $v$ is a vertex lying in $\{2,3,\ldots,n\}$.)

\begin{proof}
Assume $v$ is shown as in \Pic \ref{fig:c-in d-out}.
\begin{figure}[htbp]
\centering
\begin{tikzpicture}
\draw (0,0) node{
\xymatrix{
  u_1 \ar@{->}[rrd]_{\alpha_1}
& u_2 \ar@{->}[rd]^{\alpha_2}
& \cdots
& u_{c-1} \ar@{->}[ld]_{\alpha_{c-1}}
& u_c \ar@{->}[lld]^{\alpha_c}
\\
  & &
v \ar@{->}[lld]_{\beta_1}
  \ar@{->}[ld]^{\beta_2}
  \ar@{->}[rd]_{\beta_{d-1}}
  \ar@{->}[rrd]^{\beta_d}
  & &
\\
w_1 & w_2 & \cdots & w_{d-1} & w_d
}
};
\end{tikzpicture}
\caption{$(c^{\inner},d^{\out})$-type vertex}
\label{fig:c-in d-out}
\end{figure}
Then, by the definition of almost gentle pairs, for any $\alpha_i$ ($1\leqslant i\leqslant c$),
there is at most one arrow $\beta_{j}$ ($1\leqslant j\leqslant d$) such that $\alpha_i\beta_j \notin \I$.
In this case, if $\beta_j$ exists, then for any $\imath \ne i$, we have $\alpha_{\imath}\beta_j \in \I$,
and for any $\jmath \ne j$, we have $\alpha_i\beta_{\jmath} \in \I$.
Thus, without loss of generality, we assume
\begin{align} \label{formula:c-in d-out}
\alpha_i\beta_j \
\begin{cases}
\notin \I & \text{ if } 1\leqslant i=j\leqslant t; \\
\in \I & \text{ if } t < i \leqslant c.
\end{cases}
\end{align}

Suppose that the number of all paths crossing $v$ of length at least two is $T$. We claim that $T=t$.
By the first condition of (\ref{formula:c-in d-out}), we have $t$ paths
$\alpha_1\beta_1$, $\ldots$, $\alpha_t\beta_t$ $\notin \I$ crossing $v$, then $T\geqslant t$.
In the following, we prove $T\leqslant t$.
If $T>t$, then there is a path $\alpha_k\beta_l \notin \I$ crossing $v$ such that
$\alpha_k\beta_l \ne \alpha_i\beta_i$ for any $1\leqslant i\leqslant t$.
In this case, by the second condition of (\ref{formula:c-in d-out}), we have $k\leqslant t$.
Then $\alpha_k\beta_j\notin \I$ for some $j$ with $j\ne k$ and $1\leqslant j\leqslant d$.
It follows that both $\beta_k$ and $\beta_j$ are arrows starting at $v$ such that
$\alpha_k\beta_k$ and $\alpha_k\beta_j$ are paths not lying in $\I$,
this contradicts the definition of almost gentle pairs. The claim is proved.
\end{proof}

If there are $t$ non-zero paths of length two on an almost gentle pair $(\Q, \I)$
crossing a $(c^{\inner},d^{\out})$-type vertex $v$,
then, using the notation from Lemma \ref{lemm:vert} and \Pic \ref{fig:c-in d-out},
the anti-claw corresponding to $E(v)$ is of the form
\[ s_1 \aclaw s_2 \aclaw \cdots \aclaw s_c, \]
where each $s_i$ is a left maximal directed string of the form $\cdots \alpha_i$,
and there are $d$ right maximal directed string $s_1'$, $\ldots$, $s_d'$ starting at $v$,
where each $s_j'$ is a path of the form $\beta_j\cdots$, see \Pic \ref{fig:c-in d-out II}.

\begin{figure}[htbp]
\centering
\begin{tikzpicture}
\fill [top color = white, middle color=blue!35]
  (-0.25,-0.15) to[out=180,in=180] (-4.20, 1.50)
  -- ( 4.20, 1.50) to[out=0,in=0] ( 0.25,-0.15);
\draw [blue]
  (-0.25,-0.15) to[out=180,in=180] (-4.20, 1.50)
  -- ( 4.20, 1.50) to[out=0,in=0] ( 0.25,-0.15) -- (-0.25,-0.15);
  [line width = 0.75pt];
\draw [blue] (0,1.75) node{$E(v)$};
\draw (0.1,0) node{
\xymatrix{
  x_1 \ar@{~>}[rrd]_{s_1}
& x_2 \ar@{~>}[rd]^{s_2}
& \cdots
& x_{c-1} \ar@{~>}[ld]_{s_{c-1}}
& x_c \ar@{~>}[lld]^{s_c}
\\
  & &
v \ar@{~>}[lld]_{s_1'}
  \ar@{~>}[ld]^{s_2'}
  \ar@{~>}[rd]_{s_{d-1}'}
  \ar@{~>}[rrd]^{s_d'}
  & &
\\
y_1 & y_2 & \cdots & y_{d-1} & y_d
}
};
\draw[blue] (0,-2) node{$s_i = \cdots \alpha_i$\ ($1\leqslant i\leqslant c$)};
\draw (0,-2.6) node{$s_j' = \beta_j\cdots $\ ($1\leqslant j\leqslant d$)};
\end{tikzpicture}
\caption{$(c^{\inner},d^{\out})$-type vertex $v$ on an almost gentle pair $(\Q,\I)$}
\begin{center}
  ($s_1',\ldots, s_d'$ are right maximal directed strings)
\end{center}
\label{fig:c-in d-out II}
\end{figure}

\begin{lemma} \label{lemm:1st-syzygy of E}
Let $v$ be a $(c^{\inner},d^{\out})$-type vertex on an almost gentle pair $(\Q, \I)$ is shown in \Pic \ref{fig:c-in d-out II}.
Then $\Omega_1(E(v))$ has a decomposition as follows:
\[ \Omega_1(E(v)) \cong \spds \oplus \bigoplus_{\jmath \in J} S_{\jmath} \oplus  \bigoplus_{\imath \in I} M_{\imath}, \]
where $I$ and $J$ are index sets, $\spds$ is either zero or a module whose top is a direct sum of some copies of $S(v)$,
all $S_{\jmath}$ are copies of $S(v)$, and all $M_{\imath}$ are either simple modules or right maximal directed string modules.
\end{lemma}

(Note that $\spds$ may be not indecomposable, and, in this case, $\spds$ may contain a direct summand which is indecomposable projective.)

\begin{proof}
By Lemma \ref{lemm:vert}, we may assume $\alpha_1\beta_1, \ldots, \alpha_t\beta_t \notin \I$.
Then we have $s_1s_1', \ldots, s_ts_t' \notin \I$ since $\I$ is an admissible ideal generated by some paths of length two,
and so, for any $1\leqslant i\leqslant c$, the claw corresponding to the indecomposable projective module $P(x_i)$ is of the form
\[ \clawnota = s_{i1} \claw s_{i2} \claw \cdots \claw s_{im_i}, \]
where
\begin{itemize}
  \item $s_{i1} =
  \begin{cases}
      s_is_i', & i\leqslant t; \\
      s_i, & i\geqslant t+1,
  \end{cases}$
  here, the path $s_i$ is of the form
  \begin{align} \label{formula:1st-syzygy of E I}
  v_{i,1,1} \To{a_{i,1,1}}{} v_{i,1,2} \To{a_{i,1,2}}{} \cdots \To{a_{i,1,\ell_{i1}}}{} v_{i,1,\ell_{i1}+1} \ \ (v_{i,1,1} = x_i, v_{i,1,\ell_{i1}}+1=v),
  \end{align}
  and the path $s_i'$ is of the form
  \begin{align} \label{formula:1st-syzygy of E II}
    v_{i,1,\ell_{i1}+1} \To{a_{i,1,\ell_{i1}+1}}{} v_{i,1,\ell_{i1}+2} \To{a_{i,1,l_{i1}+2}}{} \cdots \To{a_{i,1,l_{i1}}}{} v_{i,1,l_{i1}+1} \ \ (v_{i,1,\ell_{i1}} = v, v_{i,1,l_{i1}}=y_i);
  \end{align}
  \item $s_{ij}$ is a path
  \[ v_{i,j,1} \To{a_{i,j,1}}{} v_{i,j,2} \To{a_{i,j,2}}{} \cdots \To{a_{i,j,l_{ij}}}{} v_{i,j,l_{ij}+1} \ \ (x_i = v_{i,j,1})\]
  of length $l_{ij}$ ($1 < j \leqslant m_i$).
\end{itemize}
Then the projective cover of $E(v)$ is the homomorphism
\[ p_0: P_0 = \bigoplus_{i=1}^{c} P(x_i) \to E(v) \]
\begin{center}
 (the left picture of \Pic \ref{fig:c-in d-out III} shows the quiver representation of $E(v)$,

 and the right picture of \Pic \ref{fig:c-in d-out III} shows the anti-claw of $E(v)$),
\end{center}
\begin{figure}[htbp]
\centering
\begin{tikzpicture}[scale=1.2]
\draw [->] (-0.0, 0.8) -- (-0.0, 0.2);
\draw [->] (-0.0, 1.8) -- (-0.0, 1.2);
\draw [->] (-0.0, 2.8) -- (-0.0, 2.2);
\draw (0,2) node{$\kk_{i,1,2}^{\dag}$};
\draw (0,3) node{$\kk_{i,1,1}^{\dag}$};
\draw (0,1.8) -- (0,0.2) [dotted];
\draw (0,-0.1) node{$\kk_{i,1,\ell_{i1}+1}^{\dag}$};
%
\draw [rotate=45][->] (-0.0, 0.8) -- (-0.0, 0.2);
\draw [rotate=45][->] (-0.0, 1.8) -- (-0.0, 1.2);
\draw [rotate=45][->] (-0.0, 2.8) -- (-0.0, 2.2);
\draw [rotate=45] (0,2) node{$\kk_{1,1,2}^{\dag}$};
\draw [rotate=45] (0,3) node{$\kk_{1,1,1}^{\dag}$};
\draw [rotate=45] (0,1.8) -- (0,0.2) [dotted];
\draw [rotate=-45][->] (-0.0, 0.8) -- (-0.0, 0.2);
\draw [rotate=-45][->] (-0.0, 1.8) -- (-0.0, 1.2);
\draw [rotate=-45][->] (-0.0, 2.8) -- (-0.0, 2.2);
\draw [rotate=-45] (0,2) node{$\kk_{c,1,2}^{\dag}$};
\draw [rotate=-45] (0,3) node{$\kk_{c,1,1}^{\dag}$};
\draw [rotate=-45] (0,1.8) -- (0,0.2) [dotted];
\draw (-0.9,2) node{$\cdots$};
\draw ( 0.9,2) node{$\cdots$};
\draw (0,-1.1) node{Indecomposable injective module $E(v)$};
\draw (0,-1.5) node{(each $\kk_{i,j,k}^{\dag}$ is isomorphic to $\kk$)};
\end{tikzpicture}
\
\begin{tikzpicture}[scale=1.2]
\draw [->] (-0.0, 0.8) -- (-0.0, 0.2);
\draw [->] (-0.0, 1.8) -- (-0.0, 1.2);
\draw [->] (-0.0, 2.8) -- (-0.0, 2.2);
\draw      (-0.03,-0.2) -- (-0.03,-0.4);
\draw      ( 0.03,-0.2) -- ( 0.03,-0.4);
\draw      (-0.03, 3.4) -- (-0.03, 3.2);
\draw      ( 0.03, 3.4) -- ( 0.03, 3.2);
\draw (0,2) node{$v_{(i,1,2)}$};
\draw (0,3) node{$v_{(i,1,1)}$};
\draw (0,1.8) -- (0,0.2) [dotted];
\draw (0,0) node{$v_{(i,1,\ell_{i1}+1)}$};
\draw (0,-0.6) node{$v$};
\draw (0, 3.6) node{$x_i$};
\draw [rotate=45][->] (-0.0, 0.8) -- (-0.0, 0.2);
\draw [rotate=45][->] (-0.0, 1.8) -- (-0.0, 1.2);
\draw [rotate=45][->] (-0.0, 2.8) -- (-0.0, 2.2);
\draw [rotate=45] (0,2) node{$v_{(1,1,2)}$};
\draw [rotate=45] (0,3) node{$v_{(1,1,1)}$};
\draw [rotate=45] (0,1.8) -- (0,0.2) [dotted];
\draw [rotate=-45][->] (-0.0, 0.8) -- (-0.0, 0.2);
\draw [rotate=-45][->] (-0.0, 1.8) -- (-0.0, 1.2);
\draw [rotate=-45][->] (-0.0, 2.8) -- (-0.0, 2.2);
\draw [rotate=-45] (0,2) node{$v_{(c,1,2)}$};
\draw [rotate=-45] (0,3) node{$v_{(c,1,1)}$};
\draw [rotate=-45] (0,1.8) -- (0,0.2) [dotted];
\draw (-0.9,2) node{$\cdots$};
\draw ( 0.9,2) node{$\cdots$};
\draw (0,-1.5) node{The anti-claw of $E(v)$};
\end{tikzpicture}
\caption{The indecomposable injective module $E(v)$ and the anti-claw of it}
\label{fig:c-in d-out III}
\end{figure}
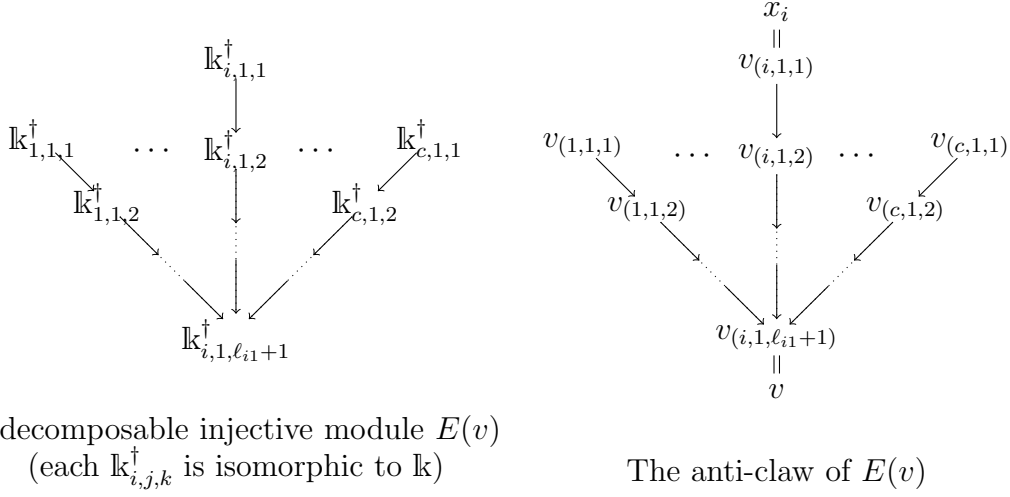
such that for each triple $(i,j,k)$ with $1\leqslant k \leqslant \ell_{ij}+1$, we have the following two facts (see \Pic \ref{fig:c-in d-out IV}, the right picture):
\begin{itemize}
  \item if $j=1$ and $1\leqslant k \leqslant \ell_{ij}+1$,
  that is, $v_{i,j,k} \in \{x_i=v_{i,1,1}, v_{i,1,2}, \ldots, v_{i,1,\ell_{ij}+1}\}$ is a vertex on the path $s_i$,
  then $p_0$ sends each one-dimensional $\kk$-linear space $\kk_{i,1,k}$,
  the direct summand of the $\kk$-linear space $P_0 e_{v_{i,1,k}}$ corresponding to the vertex $v_{i,1,k}$,
  to the one-dimensional direct summand, written as $\kk_{i,j,k}^{\dag}$, of the $\kk$-linear space $E(v)e_{v_{i,j,k}}$ corresponding to the vertex $v_{i,j,k}$ by using the identity
  \[(\mathrm{id}: x\mapsto x) \in \Hom_{\kk}(\kk,\kk) \cong \Hom_{\kk}(\kk_{i,1,k}, \kk_{i,j,k}^{\dag});\]

  \item if $(i,j,k)$ does not lie in the above two cases, then $p_0$ sends each $\kk_{i,j,k}$ to zero.
\end{itemize}

\begin{figure}[htbp]
\centering
\begin{tikzpicture}[xscale=1.2, yscale=1.45]
\fill [top color = white, middle color=blue!35]
  (-3.25,3.5) -- (3.25,3.5) -- (0,-0.7) -- (-3.25,3.5);
\fill [top color = white, middle color=orange!35]
  (-5.3,3.5) -- (-8.8,0.5) -- (-5.3,-4.85) -- (-5.0,-4.85) --
  (-5.0,3.5) -- (-5.3,3.5);
\draw [red!33][line width = 10pt] (0,3.8) -- (0,-0.28);
\draw [rotate around = { 36:(0,-0.28)}] [red!33][line width = 10pt] (0,4.5) -- (0,-0.28);
\draw [rotate around = {-36:(0,-0.28)}] [red!33][line width = 10pt] (0,4.5) -- (0,-0.28);
\draw [red] (0,3.8) node[above]{$s_i$};
\draw [rotate around = { 36:(0,-0.28)}][red] (0,4.5) node[above]{$s_1$};
\draw [rotate around = {-36:(0,-0.28)}][red] (0,4.5) node[above]{$s_{c}$};
\draw [blue] (0,-0.65) node[below]{$E(v)$};
\draw [->] (-0.0, 0.8) -- (-0.0, 0.2);
\draw [->] (-0.0, 1.8) -- (-0.0, 1.2);
\draw [->] (-0.0, 2.8) -- (-0.0, 2.2);
\draw (0,2) node{$\kk_{i,1,1}^{\dag}$};
\draw (0,3) node{$\kk_{i,1,2}^{\dag}$};
\draw (0,1.8) -- (0,0.2) [dotted];
\draw (0,-0.1) node{$\kk_{i,1,\ell_{i1}}^{\dag}$};
\draw[rotate= 38] [->] (-0.0, 0.8) -- (-0.0, 0.2);
\draw[rotate= 38] (0,1.13) -- (0,0.2) [dotted];
\draw[rotate=-38] [->] (-0.0, 0.8) -- (-0.0, 0.2);
\draw[rotate=-38] (0,1.13) -- (0,0.2) [dotted];
\draw (-0.8,1) node{$\cdots$};
\draw ( 0.8,1) node{$\cdots$};
\draw (0.61,-0.1) node[right]{($1\leqslant i \leqslant c$)};
%
%
\draw[shift={(-5.3,0)}] [->] (-0.0,-3.2) -- (-0.0,-3.8);
\draw[shift={(-5.3,0)}] [->] (-0.0,-2.2) -- (-0.0,-2.8);
\draw[shift={(-5.3,0)}] [->] (-0.0,-1.2) -- (-0.0,-1.8);
\draw[shift={(-5.3,0)}] [->] (-0.0,-0.2) -- (-0.0,-0.8);
\draw[shift={(-5.3,0)}] [->] (-0.0, 0.8) -- (-0.0, 0.2);
\draw[shift={(-5.3,0)}] [->] (-0.0, 1.8) -- (-0.0, 1.2);
\draw[shift={(-5.3,0)}] [->] (-0.0, 2.8) -- (-0.0, 2.2);
\draw[shift={(-5.3,0)}] (0,2) node{$\kk_{i,1,2}$};
\draw[shift={(-5.3,0)}] (0,3) node{$\kk_{i,1,1}$};
\draw[shift={(-5.3,0)}] (0,1.8) -- (0,0.2) [dotted];
\draw[shift={(-5.3,0)}] (0,-0.1) node{$\kk_{i,1,\ell_{i1}}$};
\draw[shift={(-5.3,0)}] (0.3,-0.28) node[right]{$\oplus\bigoplus\limits_{2\leqslant i\leqslant c}\kk_{i,1,\ell_{i1}}$};
\draw[shift={(-5.3,0)}] (0,-1) node{$\kk_{i,1,\ell_{i1}+1}$};
\draw[shift={(-5.3,0)}] (0,-1.5) -- (0,-2.5) [dotted];
\draw[shift={(-5.3,0)}] (0,-3) node{$\kk_{i,1,l_{i1}}$};
\draw[shift={(-5.3,0)}] (0,-4) node{$\kk_{i,1,l_{i1}+1}$};
\draw[shift={(-5.3,0)}][rotate around={-45:(0,3)}] [->] (-0.0, 2.8) -- (-0.0, 2.2);
\draw[shift={(-5.3,0)}][rotate around={-45:(0,3)}] [->] (-0.0, 1.8) -- (-0.0, 1.2);
\draw[shift={(-5.3,0)}][rotate around={-45:(0,3)}] (0,2) node{$\kk_{i,j,1}$};
\draw[shift={(-5.3,0)}][rotate around={-45:(0,3)}] (0,1.8) -- (0,0.5) [dotted];
\draw[shift={(-5.3,0)}][rotate around={-40:(0,3)}] (0,0.5) node[below left]{$(1\leqslant j\leqslant m_i)$};
\draw[shift={(-5.3,0)}] (-0.8,1) node{$\cdots$};
\draw[shift={(-5.3,0)}] (0,2.5) node[right]{$1$};
\draw[shift={(-5.3,0)}] (0,1.5) node[left]{$1$};
\draw[shift={(-5.3,0)}] (0,0.5) node[left]{$
  \left(\begin{smallmatrix}
    1 \\ 0 \\ \vdots \\ 0
  \end{smallmatrix}
  \right)$};
\draw[shift={(-5.3,0)}] (0,-0.5) node[left]{\tiny$(1\ 0 \cdots \ 0 )$};
\draw[shift={(-5.3,0)}] (0,-1.5) node[left]{$1$};
\draw[shift={(-5.3,0)}] (0,-2.5) node[left]{$1$};
\draw[shift={(-5.3,0)}] (0,-3.5) node[left]{$1$};
\draw[orange] (-5.3,-4.85) node[below]{$P(x_i)$};
%
%
\draw[cyan][shift={(0,0)}][->] (-4.8,0) -- (-0.8,0) [line width=1pt];
\draw[cyan][shift={(0,2)}][->] (-4.8,0) -- (-0.8,0) [line width=1pt];
\draw[cyan][shift={(0,3)}][->] (-4.8,0) -- (-0.8,0) [line width=1pt];
\draw[cyan] (-2.8, 0) node[above]{\tiny$\varphi_{i,1,\ell_{i1}}=(1 \ 1  \cdots \ 1)$};
\draw[cyan] (-2.8, 2) node[above left]{$\varphi_{i,1,2}= 1$};
\draw[cyan] (-2.8, 3) node[above left]{$\varphi_{i,1,1}= 1$};
\end{tikzpicture}
\caption{The direct summand $P(x_i)$ of the projective module $P_0 = \bigoplus\limits_{1\leqslant i\leqslant c}P(x_i)$ given by the projective cover of $E(v)$
 ({where $P_0e_v$, as a $\kk$-linear space, has a direct summand $\kk^{\oplus c} = \bigoplus\limits_{1\leqslant i\leqslant c} \kk_{i,1,\ell_{i1}}$})}
\label{fig:c-in d-out IV}
\end{figure}

Then the $1$-st syzygy $\Omega^{1}(E(v))$ is a direct sum of some modules of the following three classes:
\begin{itemize}
  \item the directed string module $M_{ij}$ corresponding to the string $\hat{s}_{ij}:= a_{i,j,2} \cdots a_{i,j,l_{ij}}$ which is obtained from $s_{ij}$ ($1 < j \leqslant m_i$); if $m_i=1$, then $\hat{s}_{ij} = e_{\target(a_{i,j,1})}$ is a simple string, that is, the above directed string module is simple;
  \item the simple module isomorphic to $S(v)$ which is given by the indecomposable projective module $P(x_i)$ with $i\geqslant t+1$
  because $s_is_{j}' = 0$ holds for all $1\leqslant j\leqslant d$;
  \item the module $\spds$ given by the $\kk$-linear maps $(\varphi_{i,1,1}, \varphi_{i,1,2}, \cdots, \varphi_{i,1,\ell_{i1}})_{1\leqslant i\leqslant t}$
  whose top is $S(v)^{\oplus (t-1)}$ (for $0 \leqslant t\leqslant 1$, we put $\spds=0$),
  and in this case, if $c=2$, then $\spds$ corresponds to a claw whose top is $S(v)$.
\end{itemize}
Therefore, we obtain
\begin{align} \label{formula:1st-syzygy of E III}
 \Omega_1(E(v)) \cong \spds \oplus S(v)^{\oplus (c-t)} \oplus
\bigg(
  \bigoplus_{(i,j)\in I} M_{ij}
\bigg),
\end{align}
where $I$ is some index set.
\end{proof}

\begin{corollary} \label{coro:1st-syzygy of E}
The top of the direct summand
\begin{center}
  $\displaystyle  \spds_0
  := \spds \oplus \bigoplus_{\jmath \in J} S_{\jmath}
  \ \le_{\oplus} \Omega_1(E(v))$
\end{center}
given in Lemma \ref{lemm:1st-syzygy of E} is isomorphic to $S(v)^{\oplus (c-1)}$.
\end{corollary}

\begin{proof}
It is a direct corollary of the formula (\ref{formula:1st-syzygy of E III}) in the proof of Lemma \ref{lemm:1st-syzygy of E}.
\end{proof}

\begin{example} \label{exp:p.res inj} \rm
Let $A=\kk\Q/\I$ be the almost gentle algebra given by Example \ref{exp:almost gent}.

(1) Consider the indecomposable injective module $E(2_{\Right}) = \left({^1}{_{2_{\Right}}}{^1}\right)$ and its projective cover $P(1)^{\oplus 2} \to E(2_{\Right})$.
The anti-claw corresponding to $E(2_{\Right})$ is
\[  a_{1,2_{\Right}} \aclaw b_{1,2_{\Right}}. \]
Here, $v=2_{\Right}$, $s_1=a_{1,2_{\Right}}$, and $s_2=b_{1,2_{\Right}}$.
The claw corresponding to $P(1)$ is
\[a_{1,2_{\Left}} \claw a_{1,2} \claw a_{1,2_{\Right}} a_{2_{\Right},3_{\Right}} a_{3_{\Right},4_{\Right}} \claw b_{1,2_{\Right}}. \]
Here, $s_1' = a_{2_{\Right},3_{\Right}} a_{3_{\Right},4_{\Right}}$.
We have
\begin{align}\label{formula: in exp:p.res inj}
 \Omega_1(E(2_{\Right})) \cong
\left( \begin{smallmatrix}
  2_{\Right} \\ 3_{\Right} \\ 4_{\Right}
\end{smallmatrix} \right)
\oplus
\left(
(2_{\Left}) \oplus (2) \oplus (2_{\Right}) \oplus
\left( \begin{smallmatrix}
  2_{\Right} \\ 3_{\Right} \\ 4_{\Right}
\end{smallmatrix} \right)
\oplus (2_{\Left}) \oplus (2)
\right),
\end{align}
\begin{figure}[htbp]
\centering
\begin{tikzpicture}[scale=1.3]
\draw ( 0  , 3  ) node{$0$}; 
\draw ( 0  , 1.5) node{$\kk^{\oplus 2}$}; 
\draw (-1.2, 0  ) node{$0$}; 
\draw ( 1.2, 0  ) node{$0$}; 
\draw ( 0  ,-1.5) node{$0$}; 
\draw ( 0  ,-3  ) node{$0$}; 
\draw (-2.1, 2.1) node{$\kk^{\oplus 2}$};  
\draw (-3.0, 0  ) node{$0$}; 
\draw (-2.1,-2.1) node{$0$}; 
\draw ( 2.1, 2.1) node{$\kk^{\oplus 3}$};  
\draw ( 3.0, 0  ) node{$\kk^{\oplus 2}$};  
\draw ( 2.1,-2.1) node{$\kk^{\oplus 2}$};  
\draw[line width = 1pt][->] ( 0  ,-1.7) -- (   0,-2.8); \draw ( 0  ,-2.2) node[right]{$0$}; 
\draw[line width = 1pt][->] ( 0  , 1.2) -- (   0,-1.2); \draw ( 0  , 0  ) node[ left]{$0$}; 
\draw[line width = 1pt][->] ( 0  , 2.8) -- (   0, 1.7); \draw ( 0  , 2.2) node[ left]{$0$}; 
\draw[line width = 1pt][->] (-0.2, 1.5) to[out= 180, in=  90] (-1.2, 0.2);
\draw[line width = 1pt][->] (-1.2,-0.2) to[out= -90, in= 180] (-0.2,-1.5);
\draw[line width = 1pt][->] ( 0.2, 1.5) to[out=   0, in=  90] ( 1.2, 0.2);
\draw[line width = 1pt][->] ( 1.2,-0.2) to[out= -90, in=   0] ( 0.2,-1.5);
\draw (-1. , 1.0) node[ left]{$0$}; 
\draw ( 1. , 1.0) node[right]{$0$}; 
\draw (-1. ,-1.0) node[ left]{$0$}; 
\draw ( 1. ,-1.0) node[right]{$0$}; 
\draw[line width = 1pt][->][rotate= 5+  0] (0,3) arc(90:125:3);
\draw[line width = 1pt][->][rotate= 5+ 45] (0,3) arc(90:125:3);
\draw[line width = 1pt][->][rotate= 5+ 90] (0,3) arc(90:125:3);
\draw[line width = 1pt][->][rotate= 5+135] (0,3) arc(90:125:3);
\draw[line width = 1pt][->][rotate=-5-  0] (0,3) arc(90:55:3);
\draw[line width = 1pt][->][rotate=-5- 45] (0,3) arc(90:55:3);
\draw[line width = 1pt][->][rotate=-5- 90] (0,3) arc(90:55:3);
\draw[line width = 1pt][->][rotate=-5-135] (0,3) arc(90:55:3);
\draw[line width = 1pt][->][rotate=-5+  0] (0,2.8) arc(90:55:2.8);
\draw[line width = 1pt][<-][rotate=-5+180] (0,2.8) arc(90:55:2.8);
%
\draw[rotate=   0] ( 1.06, 3.21-0.09) node{$0$};
\draw[rotate=   0] ( 0.96, 2.35) node{$0$};
\draw[rotate=  45] ( 1.06, 3.21) node{$0$};
\draw[rotate=  90] ( 1.06, 3.21) node{$0$};
\draw[rotate= 135] ( 1.06, 3.21) node{$0$};
\draw[rotate= 180] ( 1.06, 3.21-0.09) node{$0$};
\draw[rotate= -45] ( 1.06, 3.21) node{$
  \left[\begin{smallmatrix}
    1 & 0 & 0 \\
    0 & 1 & 0
  \end{smallmatrix}\right]$};
\draw[rotate= -90] ( 1.06, 3.21) node{$
  \left[\begin{smallmatrix}
    1 & 0 \\
    0 & 1
  \end{smallmatrix}\right]$};
\draw[rotate=-135] ( 1.06, 3.21) node{$0$};
\draw[rotate= 180] ( 0.96, 2.35) node{$0$};
\draw[ red][rotate=   0] (-1.72, 2.45) arc(  45:-135:0.5) [line width = 1pt][dotted];
\draw[ red][rotate=   0] (-1.21, 0.50) arc(  90: 270:0.5) [line width = 1pt][dotted];
\draw[ red][rotate=   0] ( 0.00, 2.00) arc(  90: 185:0.5) [line width = 1pt][dotted];
\draw[blue][rotate=   0] ( 0.00, 2.20) arc(  90:  -5:0.7) [line width = 1pt][dotted];
\draw[blue][rotate= 180] (-1.72, 2.45) arc(  45:-135:0.5) [line width = 1pt][dotted];
\draw[blue][rotate= 180] (-1.21, 0.50) arc(  90: 270:0.5) [line width = 1pt][dotted];
\draw[blue][rotate= 180] ( 0.00, 2.00) arc(  90: 185:0.5) [line width = 1pt][dotted];
\draw[ red][rotate= 180] ( 0.00, 2.20) arc(  90:  -5:0.7) [line width = 1pt][dotted];
\draw[green][rotate=   0] (-1.62,-2.25) arc(-25: 135:0.5) [line width = 1pt][dotted];
\draw[green][rotate= 180] (-1.62,-2.25) arc(-25: 135:0.5) [line width = 1pt][dotted];
\draw[green] ( 0.00, 2.10) to[out=0,in=0] ( 0.00, 0.90) [line width = 1pt][dotted];
\end{tikzpicture}
\caption{The $1$st-syzygy $\Omega_1(E(2_{\Right}))$ of the injective module $E(2_{\Right})$}
\label{fig: in exp:p.res inj}
\end{figure}
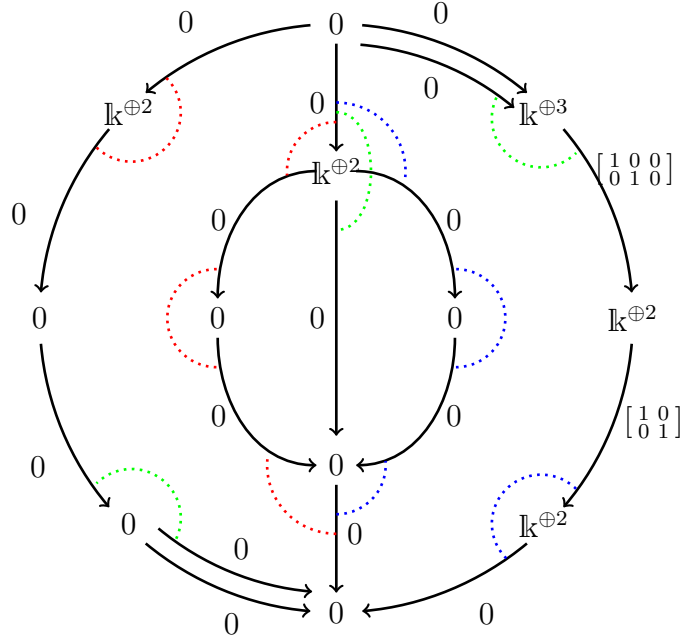
see \Pic \ref{fig: in exp:p.res inj}, where
$\spds \cong \left( \begin{smallmatrix}
  2_{\Right} \\ 3_{\Right} \\ 4_{\Right}
\end{smallmatrix} \right)$ is the first direct summand shown in (\ref{formula: in exp:p.res inj}),
$\bigoplus_{\jmath\in J} S_{\jmath} = 0$ (the index set $J$ is an empty set),
and $\bigoplus_{\imath\in I} M_{\imath} \cong (2_{\Left}) \oplus (2) \oplus (2_{\Right}) \oplus
\left( \begin{smallmatrix}
  2_{\Right} \\ 3_{\Right} \\ 4_{\Right}
\end{smallmatrix} \right)
\oplus (2_{\Left}) \oplus (2)$ which is a direct sum of some right maximal directed string modules.
Thus, we have
\[\spds_0 = \spds \cong \left( \begin{smallmatrix}
  2_{\Right} \\ 3_{\Right} \\ 4_{\Right}
\end{smallmatrix} \right). \]

(2) Consider the indecomposable injective module $E(4) = \left({^3}\ {^2_4}\ {^{3'}}\right)$ and its projective cover
$P(3)\oplus P(2) \oplus P(3') \to E(4)$.
The anti-claw corresponding to $E(4)$ is
\begin{center}
  $a_{3,4} \aclaw a_{2,4} \aclaw a_{3,4'}$ $=$
\begin{tikzpicture}[baseline=-0.25cm]
 \draw (0,0) node{
 $\xymatrix@R=1.24cm{
 3 \ar[rd]_{a_{3,4}} & 2 \ar[d]|{a_{2,4}} &  3' \ar[ld]^{a_{3,4'}} \\
 & 4  &
}$};
\end{tikzpicture},
\end{center}
and $P(3)$, $P(2)$ and $P(3)$ are indecomposable projective modules corresponding to the claws
\begin{center}
$a_{3,4}$,
\ \
$a_{2,3}\claw a_{2,4} \claw a_{2,3'} = $
\begin{tikzpicture}[baseline=-0.25cm]
 \draw (0,0) node{
 $\xymatrix@R=1.24cm{
 & 2 \ar[ld]_{a_{2,3}} \ar[d]|{a_{2,4}} \ar[rd]^{a_{2,3'}} & \\
 3 & 4\ar[d]|{a_{4,5}} & 3' \\
 & 5 &
}$};
\end{tikzpicture},
\ \
 and $a_{3',4}$,
\end{center}
respectively.
Then we have $c=3$, $d=1$, $s_1=a_{3,4}$, $s_2=a_{2,4}$, $s_3=a_{3,4'}$, and $s_1'=a_{2,4}a_{4,5}$. Thus,
\[ \Omega_1(E(4)) \cong ({^4_5}) \oplus (4) \oplus ((3) \oplus (3')), \]
where $\spds \cong ({^4_5})$, $\bigoplus_{\jmath\in J} S_{\jmath} \cong (4)$ (the index set $J$ contains only one element),
and $\bigoplus_{\imath\in I} M_{\imath} \cong (3) \oplus (3')$ (the index set $I$ is $\{1,2\}$, $M_1 \cong (3)$, and $M_2 \cong (3')$),
see \Pic \ref{fig: in exp:p.res inj 2}.
\begin{figure}[htbp]
\centering
\begin{tikzpicture}[scale=1.3]
\draw ( 0  , 3  ) node{$0$}; 
\draw ( 0  , 1.5) node{$0$}; 
\draw (-1.2, 0  ) node{$\kk$}; 
\draw ( 1.2, 0  ) node{$\kk$}; 
\draw ( 0  ,-1.5) node{$\kk\!{}^{\oplus 2}$}; 
\draw ( 0  ,-3  ) node{$\kk$}; 
\draw (-2.1, 2.1) node{$0$}; 
\draw (-3.0, 0  ) node{$0$}; 
\draw (-2.1,-2.1) node{$0$}; 
\draw ( 2.1, 2.1) node{$0$}; 
\draw ( 3.0, 0  ) node{$0$}; 
\draw ( 2.1,-2.1) node{$0$}; 
\draw[line width = 1pt][->] ( 0  ,-1.7) -- (   0,-2.8); \draw ( 0  ,-2.2) node[right]{$[1\ 0]$}; 
\draw[line width = 1pt][->] ( 0  , 1.2) -- (   0,-1.2); \draw ( 0  , 0  ) node[ left]{$0$}; 
\draw[line width = 1pt][->] ( 0  , 2.8) -- (   0, 1.7); \draw ( 0  , 2.2) node[ left]{$0$}; 
\draw[line width = 1pt][->] (-0.2, 1.5) to[out= 180, in=  90] (-1.2, 0.2);
\draw[line width = 1pt][->] (-1.2,-0.2) to[out= -90, in= 180] (-0.2,-1.5);
\draw[line width = 1pt][->] ( 0.2, 1.5) to[out=   0, in=  90] ( 1.2, 0.2);
\draw[line width = 1pt][->] ( 1.2,-0.2) to[out= -90, in=   0] ( 0.2,-1.5);
\draw (-1. , 1.0) node[ left]{$0$}; 
\draw ( 1. , 1.0) node[right]{$0$}; 
\draw (-1. ,-1.0) node[ left]{$0$}; 
\draw ( 1. ,-1.0) node[right]{$0$}; 
\draw[line width = 1pt][->][rotate= 5+  0] (0,3) arc(90:125:3);
\draw[line width = 1pt][->][rotate= 5+ 45] (0,3) arc(90:125:3);
\draw[line width = 1pt][->][rotate= 5+ 90] (0,3) arc(90:125:3);
\draw[line width = 1pt][->][rotate= 5+135] (0,3) arc(90:125:3);
\draw[line width = 1pt][->][rotate=-5-  0] (0,3) arc(90:55:3);
\draw[line width = 1pt][->][rotate=-5- 45] (0,3) arc(90:55:3);
\draw[line width = 1pt][->][rotate=-5- 90] (0,3) arc(90:55:3);
\draw[line width = 1pt][->][rotate=-5-135] (0,3) arc(90:55:3);
\draw[line width = 1pt][->][rotate=-5+  0] (0,2.8) arc(90:55:2.8);
\draw[line width = 1pt][<-][rotate=-5+180] (0,2.8) arc(90:55:2.8);
%
\draw[rotate=   0] ( 1.06, 3.21-0.09) node{$0$};
\draw[rotate=   0] ( 0.96, 2.35) node{$0$};
\draw[rotate=  45] ( 1.06, 3.21) node{$0$};
\draw[rotate=  90] ( 1.06, 3.21) node{$0$};
\draw[rotate= 135] ( 1.06, 3.21) node{$0$};
\draw[rotate= 180] ( 1.06, 3.21-0.09) node{$0$};
\draw[rotate= -45] ( 1.06, 3.21) node{$0$};
\draw[rotate= -90] ( 1.06, 3.21) node{$0$};
\draw[rotate=-135] ( 1.06, 3.21) node{$0$};
\draw[rotate= 180] ( 0.96, 2.35) node{$0$};
\draw[ red][rotate=   0] (-1.72, 2.45) arc(  45:-135:0.5) [line width = 1pt][dotted];
\draw[ red][rotate=   0] (-1.21, 0.50) arc(  90: 270:0.5) [line width = 1pt][dotted];
\draw[ red][rotate=   0] ( 0.00, 2.00) arc(  90: 185:0.5) [line width = 1pt][dotted];
\draw[blue][rotate=   0] ( 0.00, 2.20) arc(  90:  -5:0.7) [line width = 1pt][dotted];
\draw[blue][rotate= 180] (-1.72, 2.45) arc(  45:-135:0.5) [line width = 1pt][dotted];
\draw[blue][rotate= 180] (-1.21, 0.50) arc(  90: 270:0.5) [line width = 1pt][dotted];
\draw[blue][rotate= 180] ( 0.00, 2.00) arc(  90: 185:0.5) [line width = 1pt][dotted];
\draw[ red][rotate= 180] ( 0.00, 2.20) arc(  90:  -5:0.7) [line width = 1pt][dotted];
\draw[green][rotate=   0] (-1.62,-2.25) arc(-25: 135:0.5) [line width = 1pt][dotted];
\draw[green][rotate= 180] (-1.62,-2.25) arc(-25: 135:0.5) [line width = 1pt][dotted];
\draw[green] ( 0.00, 2.10) to[out=0,in=0] ( 0.00, 0.90) [line width = 1pt][dotted];
\end{tikzpicture}
\caption{The $1$st-syzygy $\Omega_1(E(4))$ of the injective module $E(4)$}
\label{fig: in exp:p.res inj 2}
\end{figure}
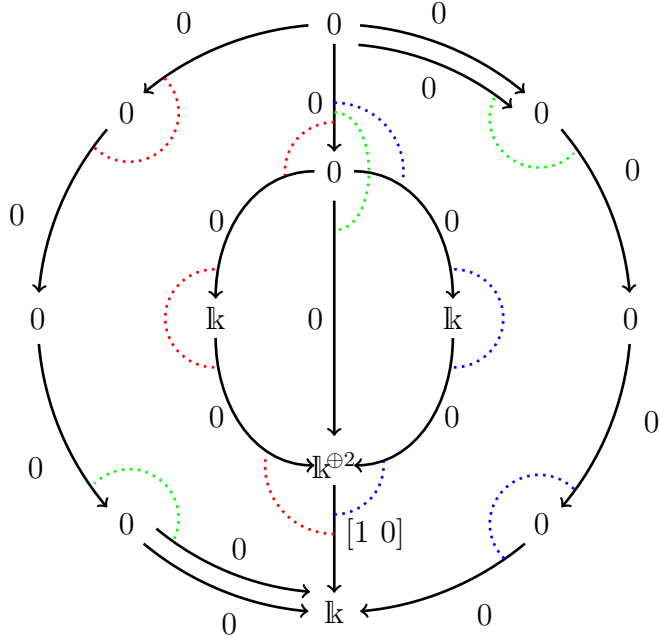
Thus, we have
\[ \spds_0 \cong ({^4_5}) \oplus (4). \]
\end{example}

\subsection{The projectivity of the direct summand $\spds_0$}

To compute $\pdim D(A)$, we need a method to describe the projectivity of $\spds_0$.
We say that a $(c^{\inner}, d^{\out})$-type vertex $v$ of $\Q$ is a {\defines gentle vertex} if the following conditions hold:
\begin{itemize}
  \item $c \leqslant 2$ and $d \leqslant 2$, in this case, we say that $v$ is a $((\leqslant 2)^{\inner}, (\leqslant 2)^{\out})$-type vertex for simplicity;
  \item the number of arrows ending at $v$ is less than or equal to $2$, and that of arrows starting at $v$ is less than or equal to $2$;
  \item there is at most one path of length two crossing $v$ such that it lies in $\I$,
  and there at most one path of length two crossing $v$ such that it does not lie in $\I$.
\end{itemize}
A bound quiver $(\Q,\I)$ is a gentle pair if $\I$ is generated by some paths of length two and all vertices of $\Q$ are gentle.

\begin{lemma} \label{lemm:spds I}
Keep the notation in Lemma \ref{lemm:1st-syzygy of E} and let $c=2$.
If $v$ is a gentle vertex, then the direct summand $\spds_0 \le_{\oplus} \Omega_1(E(v))$ is projective.
\end{lemma}

(Note that if $c\leqslant 1$, then $E(v)$ is a directed string module whose projective dimension
can be computed by using Lemma \ref{lemm:pdim dire str mod}.)

\begin{proof}
Since $v$ is a $(c^{\inner},d^{\out})$-type vertex, if $v$ is a gentle vertex, then we have $c\leqslant 2$ and $d \leqslant 2$.
In the following, we show that $\spds_0$ is projective.
\begin{itemize}
  \item[(1)] In the case for $d=0$, we have that $v$ is a sink of $\Q$,
    and $E(v)$ corresponds to a string which is of the following form
    \begin{center}
      $v_{1,1,1} \To{a_{1,1,1}}{} v_{1,1,2} \To{a_{1,1,2}}{}  \cdots \To{a_{1,1,\ell_{11}}}{}
      v \oT{a_{2,1,\ell_{21}}}{} \cdots \oT{a_{2,1,2}}{} v_{2,1,2} \oT{a_{2,1,1}}{} v_{2,1,1}$

      $(v_{1,1,\ell_{11}+1} = v = v_{2,1,\ell_{21}+1})$,
    \end{center}
    then, by Lemma \ref{lemm:1st-syzygy of E}, we have
    \[\Omega_1(E(v)) \cong P(v) \oplus \bigoplus_{\imath\in I} M_{\imath}, \]
    where $I$ is some index set, all $M_{\imath}$ are directed string modules, and $\spds_0 \cong S(v) \cong P(v)$ is both simple and projective.
  \item[(2)] In the case for $d=1$, there is a unique arrow $\alpha$ such that $\source(\alpha)=v$.
    Since $v$ is a gentle vertex, then either $a_{1,1,\ell_{11}}\alpha \in \I$ or $a_{2,1,\ell_{21}}\alpha \in \I$ holds.
    Assume $a_{1,1,\ell_{11}}\alpha \in \I$ and let $s'$ be the right maximal directed string given by the path
    \[ v_{2,1,\ell_{21}+1} \To{a_{2,1,\ell_{21}+1}}{}
       v_{2,1,\ell_{21}+2} \To{v_{2,1,\ell_{21}+2}}{}
       \cdots \To{a_{2,1,l_{21}}}{} v_{2,1,l_{21}+1} \]
    on $(\Q,\I)$ satisfying $\alpha = a_{2,1,\ell_{21}+1}$.
    Then $a_{2,1,\ell_{21}}\alpha \notin \I$, and we have
    \[\Omega_{1}(E(v)) \cong \M(s') \oplus \bigoplus_{\imath\in I} M_{\imath}  \]
    by Lemma \ref{lemm:1st-syzygy of E}, where $I$ is some index set, all $M_{\imath}$ are directed string modules, and
    \[ \spds_0 \cong \M(s') \cong P(v) \]
    is both a directed string module and a projective module
    since $\I$ is generated by some paths of length two.
  \item[(3)] In the case for $d=2$, there are two arrow $\alpha$ and $\beta$ such that $\source(\alpha) = v = \source(\beta)$,
    and, without loss of generality, we have
    $a_{1,1,\ell_{11}}\alpha\in\I$, $a_{1,1,\ell_{11}}\beta\notin\I$,
    $a_{2,1,\ell_{21}}\alpha\notin\I$, $a_{2,1,\ell_{21}}\beta\in\I$.
    Then we have two right maximal directed strings $s_1'$ and $s_2'$ on $(\Q,\I)$ which are given by the paths
    \begin{center}
      $ v_{1,1,\ell_{11}+1} \To{a_{1,1,\ell_{11}+1}}{} v_{1,1,\ell_{11}+2} \To{a_{1,1,\ell_{11}+2}}{} \cdots \To{a_{1,1,l_{11}}}{} v_{1,1,l_{11}+1} $

    and $ v_{2,1,l_{21}+1} \oT{a_{2,1,l_{21}}}{} \cdots \oT{a_{2,1,\ell_{21}+2}}{} v_{2,1,\ell_{21}+2} \oT{a_{2,1,\ell_{21}+1}}{} v_{2,1,\ell_{21}+1}$,
    \end{center}
    respectively. Here, $a_{1,1,\ell_{11}+1}=\beta$, $a_{2,1,\ell_{21}+1}=\alpha$.
    Thus
    \[\Omega_{1}(E(v)) \cong \M(s_1' (s_2')^{-1}) \oplus \bigoplus_{\imath\in I} M_{\imath} \]
    by Lemma \ref{lemm:1st-syzygy of E}, where $I$ is some index set, all $M_{\imath}$ are directed string modules, and
    \[ \spds_0 \cong \M(s_1' (s_2')^{-1}) \cong P(v) \]
    is projective.
\end{itemize}
Therefore, we conclude that $\spds_0$ is projective.
\end{proof}

\begin{example} \rm
Consider the $1$-st syzygy $\Omega_1(E(2_{\Right}))$ of the indecomposable injective module $E(2_{\Right})$ given in Example \ref{exp:p.res inj},
it follows from Lemma \ref{lemm:spds I} that $\spds_0$ is projective since $2_{\Right}$ is a gentle vertex.
Of course, we can directly check that $\spds_0 \cong P(2_{\Right})$ is projective by using the definition of projection modules.
\end{example}

\begin{lemma} \label{lemm:spds II}
Keep the notation in Lemma \ref{lemm:1st-syzygy of E} and Corollary \ref{coro:1st-syzygy of E}, and let $c>2$.
Then the direct summand $\spds_0 \le_{\oplus} \Omega_1(E(v))$ is projective if and only if $d=0$
{\rm(}where $t$ is the integer with $0 \leqslant t \leqslant \min\{c,d\}$ given in the proof of Lemma \ref{lemm:1st-syzygy of E}{\rm)}.
\end{lemma}

\begin{proof}
By Lemma \ref{lemm:1st-syzygy of E} and Corollary \ref{coro:1st-syzygy of E}, if $d>0$,
then $\spds_0 \le_{\oplus} \Omega_1(E(v))$ is a module such that the following conditions hold:
\begin{itemize}
  \item $\top(\spds_0) \cong S(v)^{\oplus (c-1)}$;
  \item $\displaystyle \rad(\spds_0) \cong \bigoplus_{1\leqslant i\leqslant t} \M(s_i'')$,
  where $s_i''$ is the path
  \[ v_{i,1,\ell_{i1}+2} \To{a_{i,1,l_{i1}+2}}{} \cdots \To{a_{i,1,l_{i1}}}{} v_{i,1,l_{i1}+1} \]
  satisfying $s' = a_{i,1,l_{i1}+1}s_i''$, see (\ref{formula:1st-syzygy of E II}).
\end{itemize}
Then the projective cover of $\spds_0$ is
\[ P(v)^{\oplus (c-1)} \to \spds_0 \]
satisfying
\begin{align*}
   & \dim_{\kk}(P(v)^{\oplus (c-1)}) - \dim_{\kk}\spds_0 \\
=\ & (c-1)\bigg(1+\sum_{i=1}^d \ell(s_i')\bigg)
   - \bigg((c-1)+\sum_{i=1}^t \ell(s_i') \bigg) \\
=\ & (c-2)\sum_{i=1}^t \ell(s_i') + (c-1)\sum_{i=t+1}^d \ell(s_i') > 0.
\end{align*}
Here, $t=0$ admits $\sum_{i=1}^t \ell(s_i') = 0$, and $d=t=1$ admits $\sum_{i=t+1}^d \ell(s_i') = 0$.
Thus, if $d>0$, then $\spds_0$ is non-projective.

If $d=0$, then $v$ is a sink which admits that $\spds$ is isomorphic to a direct sum of some copies of $S(v) \cong P(v)$.
\end{proof}

\begin{example} \rm
Consider the $1$-st syzygy $\Omega_1(E(4))$ of the indecomposable injective module $E(4)$ given in Example \ref{exp:p.res inj}.
We have that $\spds_0 \cong ({^4_5}) \oplus (4) $ is non-projective.
Here, we have $c=3>2$, $d=1\ne 0$, $\M(s_1')\cong P(4)$, $\ell(s_1')=1$, and
\[ \dim_{\kk} (P(4)^{\oplus 2}) - \dim_{\kk} \spds_0 = (3-1)(1+\ell(s_1')) - 3 = 1 > 0. \]
\end{example}

\begin{lemma} \label{lemm:spds III}
Keep the notation in Lemma \ref{lemm:1st-syzygy of E} and Corollary \ref{coro:1st-syzygy of E}, and let $d>2$.
Then, when $c\geqslant 1$, the direct summand $\spds_0 \le_{\oplus} \Omega_1(E(v))$ is projective
if and only if one of the following conditions holds.
\begin{itemize}
  \item[\rm(1)]
    $c=1$ {\rm(}in this case, the path $s_1$, see \Pic \ref{fig:c-in d-out II} or \Pic \ref{fig:c-in d-out IV}, is a unique path ending at $v${\rm)},
    and there is a unique right maximal path $s_j'$ $(1\leqslant j \leqslant d)$ such that
    $s_1s_j'\notin\I$, $\ell(s_j')=1$, and $\target(s_j')$ is a sink of the quiver $\Q$.
  \item[\rm(2)]
    $c=1$, and there is a unique right maximal path $s_j' = $
    \[ v=v_{1,1,\ell_{11}+1} \To{a_{1,1,\ell_{11}+1}}{}
         v_{1,1,\ell_{11}+2} \To{a_{1,1,\ell_{11}+2}}{}
         \cdots \To{a_{1,1,l_{11}}}{} v_{{1,1,l_{11}+1}}
         \ \ \text{{\rm(}cf. {\rm(}\ref{formula:1st-syzygy of E II}{\rm)}{\rm)}}
    \]
    $(1\leqslant j \leqslant d)$ such that $s_1s_j'\notin\I$, $\ell(s_j')\geqslant 2$,
    and the vertex $v_{1,1,\ell_{11}+2}$ $(\in v^{\da})$ on the path $s_j'$ is not an  $(a_{1,1,\ell_{11}+1},\bfda)$-relational vertex.
  \item[\rm(3)]
    $c=1$, and $s_1s_j'\in\I$ holds for all $1\leqslant j \leqslant d$.
\end{itemize}
\end{lemma}

\begin{proof}
First of all, if $c>2$, then $\spds_0$ is non-projective by Lemma \ref{lemm:spds II}.
Thus, we only need to consider the case for $1 \leqslant c \leqslant 2$.

If $c=2$, we have two subcases as follows:

\begin{itemize}
\item[(a)]
$t=1$. In this subcase, cf. \Pic \ref{fig:c-in d-out II}, there is a unique path lying in $\{s_1', s_2', \ldots, s_d'\}$,
assuming $s_1'$ without loss of generality, such that $s_1s_1'\notin \I$ holds.
Then, by Lemma \ref{lemm:1st-syzygy of E}, $\spds_0$ is a string module corresponding to the string $s_1'$,
then the projective cover is of the following form
\[ P(v) \to \spds_0 \]
which admits
\begin{align*}
   & \dim_{\kk} P(v) -  \dim_{\kk} \spds_0 \\
=\ & \bigg(1+\sum_{i=1}^d \ell(s_i') \bigg)
   - \bigg(1+ \ell(s_1') \bigg) \\
=\ & \sum_{i=2}^d \ell(s_i') > 0.\ \ \text{(for $d>2$)}
\end{align*}
It follows that the kernel of $P(v) \to \spds_0$ is not zero. Thus, $\spds_0$ is non-projective as required.
\item[(b)]
$t=2$. In this subcase, cf. \Pic \ref{fig:c-in d-out II}, there are two paths lying in $\{s_1', s_2', \ldots, s_d'\}$,
assuming $s_1'$ and $s_2'$ without loss of generality, such that
$s_1s_1'\notin \I$ and $s_2s_2'\notin \I$ hold.
Then, by Lemma \ref{lemm:1st-syzygy of E}, $\spds_0$ is a string module corresponding to the string $(s_1')^{-1}s_2'$,
and we obtain that the projective cover is of the following form
\[ P(v) \to \spds_0 \]
which admits
\begin{align*}
   & \dim_{\kk} P(v) -  \dim_{\kk} \spds_0 \\
=\ & \bigg(1+\sum_{i=1}^d \ell(s_i') \bigg)
   - \bigg(1+ \ell(s_1') + \ell(s_2') \bigg) \\
=\ & \sum_{i=3}^d \ell(s_i') > 0.\ \ \text{(by using $d>2$)}
\end{align*}
It follows that $\spds_0$ is non-projective by an argument similar to that in (a).
\end{itemize}
Thus, $\spds_0$ is non-projective in the case of $c=2$.

If $c=1$, then $E(v)$ is both an injective module and a directed string module. We have the following two cases.
\begin{itemize}
  \item[(c)]  there is a unique path lying in $\{s_1', s_2', \ldots, s_d'\}$,
  assuming $s_1'$ without loss of generality, such that $s_1s_1'\notin \I$ holds.
  In this case, $\spds_0$ is a string module corresponding to the string $s_1''$ which is obtained by
  deleting the first arrow $a_{1,1,\ell_{11}+1}$ of $s'$
  where $\ell_{11}+1\geqslant 1$, $v_{1,1,\ell_{11}+1}=v$, see (\ref{formula:1st-syzygy of E II}).
  Then the following two subcases need to be considered.

  \begin{itemize}
    \item[(c-1)] $\ell(s_1')=1$. In this subcase, $s_1''$ is a path of length zero given
    by the vertex $v_{1,1,\ell_{11}+2}$. It follows that $\spds_0 = S(v_{1,1,\ell_{11}+2})$
    is projective if and only if $v_{1,1,\ell_{11}+2} = \target(s_1'') = \target(s_1')$ is a sink.  We obtain (1).

    \item[(c-2)] $\ell(s_1')\geqslant 2$. In this subcase,
    \[s_1'' =\ \  v_{1,1,\ell_{11}+2} \To{a_{1,1,\ell_{11}+2}}{} \cdots \To{a_{1,1,l_{11}}}{} v_{1,1,l_{11}+1} \]
    is a right maximal directed string of length $\geqslant 1$. We have $\spds_0 \cong \M(s_1'')$.
    Thus, $\spds_0$ is non-projective if and only if there is at least one arrow $\alpha$ starting at $v_{1,1,\ell_{i1+2}}$,
    and by the definition of almost gentle algebra, we have $a_{1,1,\ell_{11}+1}\alpha\in\I$.
    That is, $\spds_0$ is non-projective if and only if $v_{1,1,\ell_{i1+2}}$ is $(a_{1,1,\ell_{11}+1},\bfda)$-relational. We obtain (2).
  \end{itemize}

  \item[(d)] For all path $s_j' \in \{s_1', s_2', \ldots, s_d'\}$, we have $s_1s_j'\in \I$.
  In this case, $\spds_0 = 0$ is projective. We obtain (3).
\end{itemize}

Therefore, $\spds_0$ is projective if and only if the $(c^{\inner}, d^{\out})$-type vertex $v$ lies in one of the cases (c-1), (c-2) and (d).
\end{proof}

\begin{example} \rm
(1) Consider the almost gentle pair $(\Q,\I')$ given by the quiver $\Q$ provided in Example \ref{exp:almost gent}
and the admissible ideal $\I'=\I\backslash\{a_{1,2}a_{2,4}\}\cap\{a_{2,4}a_{2,5}\}$
see \Pic \ref{fig:almost gent II}.
\begin{figure}[htbp]
\centering
\begin{tikzpicture}[scale=1.3]
\draw ( 0  , 3  ) node{$1$};
\draw ( 0  , 1.5) node{$2$};
\draw (-1.2, 0  ) node{$3$};
\draw ( 1.2, 0  ) node{$3'$};
\draw ( 0  ,-1.5) node{$4$};
\draw ( 0  ,-3  ) node{$5$};
\draw (-2.1, 2.1) node{$2_{\Left}$};
\draw (-3.0, 0  ) node{$3_{\Left}$};
\draw (-2.1,-2.1) node{$4_{\Left}$};
\draw ( 2.1, 2.1) node{$2_{\Right}$};
\draw ( 3.0, 0  ) node{$3_{\Right}$};
\draw ( 2.1,-2.1) node{$4_{\Right}$};
\draw[line width = 1pt][->] ( 0  ,-1.7) -- (   0,-2.8); \draw ( 0  ,-2.2) node[right]{$a_{4,5}$};
\draw[line width = 1pt][->] ( 0  , 1.2) -- (   0,-1.2); \draw ( 0  , 0  ) node[ left]{$a_{2,4}$};
\draw[line width = 1pt][->] ( 0  , 2.8) -- (   0, 1.7); \draw ( 0  , 2.2) node[ left]{$a_{1,2}$};
\draw[line width = 1pt][->] (-0.2, 1.5) to[out= 180, in=  90] (-1.2, 0.2);
\draw[line width = 1pt][->] (-1.2,-0.2) to[out= -90, in= 180] (-0.2,-1.5);
\draw[line width = 1pt][->] ( 0.2, 1.5) to[out=   0, in=  90] ( 1.2, 0.2);
\draw[line width = 1pt][->] ( 1.2,-0.2) to[out= -90, in=   0] ( 0.2,-1.5);
\draw (-1. , 1.0) node[ left]{$a_{2,3}$};
\draw ( 1. , 1.0) node[right]{$a_{2,3'}$};
\draw (-1. ,-1.0) node[ left]{$a_{3,4}$};
\draw ( 1. ,-1.0) node[right]{$a_{3',4}$};
\draw[line width = 1pt][->][rotate= 5+  0] (0,3) arc(90:125:3);
\draw[line width = 1pt][->][rotate= 5+ 45] (0,3) arc(90:125:3);
\draw[line width = 1pt][->][rotate= 5+ 90] (0,3) arc(90:125:3);
\draw[line width = 1pt][->][rotate= 5+135] (0,3) arc(90:125:3);
\draw[line width = 1pt][->][rotate=-5-  0] (0,3) arc(90:55:3);
\draw[line width = 1pt][->][rotate=-5- 45] (0,3) arc(90:55:3);
\draw[line width = 1pt][->][rotate=-5- 90] (0,3) arc(90:55:3);
\draw[line width = 1pt][->][rotate=-5-135] (0,3) arc(90:55:3);
\draw[line width = 1pt][->][rotate=-5+  0] (0,2.8) arc(90:55:2.8);
\draw[line width = 1pt][<-][rotate=-5+180] (0,2.8) arc(90:55:2.8);
\draw[rotate=   0] ( 1.06, 3.21) node{$a_{1,2_{\Right}}$};
\draw[rotate=   0] ( 0.96, 2.35) node{$b_{1,2_{\Right}}$};
\draw[rotate=  45] ( 1.06, 3.21) node{$a_{1,2_{\Left}}$};
\draw[rotate=  90] ( 1.06, 3.21) node{$a_{2_{\Left},3_{\Left}}$};
\draw[rotate= 135] ( 1.06, 3.21) node{$a_{3_{\Left},4_{\Left}}$};
\draw[rotate= 180] ( 1.06, 3.21) node{$a_{4_{\Left},5}$};
\draw[rotate= -45] ( 1.06, 3.21) node{$a_{2_{\Right},3_{\Right}}$};
\draw[rotate= -90] ( 1.06, 3.21) node{$a_{3_{\Right},4_{\Right}}$};
\draw[rotate=-135] ( 1.06, 3.21) node{$a_{4_{\Right},5}$};
\draw[rotate= 180] ( 0.96, 2.35) node{$b_{4_{\Left},5}$};
\draw[ red][rotate=   0] (-1.72, 2.45) arc(  45:-135:0.5) [line width = 1pt][dotted];
\draw[ red][rotate=   0] (-1.21, 0.50) arc(  90: 270:0.5) [line width = 1pt][dotted];
\draw[ red][rotate=   0] ( 0.00, 2.00) arc(  90: 185:0.5) [line width = 1pt][dotted];
\draw[blue][rotate=   0] ( 0.00, 2.20) arc(  90:  -5:0.7) [line width = 1pt][dotted];
\draw[blue][rotate= 180] (-1.72, 2.45) arc(  45:-135:0.5) [line width = 1pt][dotted];
\draw[blue][rotate= 180] (-1.21, 0.50) arc(  90: 270:0.5) [line width = 1pt][dotted];
\draw[blue][rotate= 180] ( 0.00, 2.00) arc(  90: 185:0.5) [line width = 1pt][dotted];
\draw[ red][rotate= 180] ( 0.00, 2.20) arc(  90:  -5:0.7) [line width = 1pt][dotted];
\draw[green][rotate=   0] (-1.62,-2.25) arc(-25: 135:0.5) [line width = 1pt][dotted];
\draw[green][rotate= 180] (-1.62,-2.25) arc(-25: 135:0.5) [line width = 1pt][dotted];
\draw[green][shift={(0,-3)}] ( 0.00, 2.10) to[out=0,in=0] ( 0.00, 0.90) [line width = 1pt][dotted];
\end{tikzpicture}
\caption{An almost gentle pair $(\Q,\I')$ ($\I'=\I\backslash\{a_{1,2}a_{2,4}\}\cap\{a_{2,4}a_{2,5}\}$)}
\label{fig:almost gent II}
\end{figure}
Then the vertex $2$ is a $(1^{\inner},3^{\out})$-type vertex such that the conditions given in Lemma \ref{lemm:spds III}(1) are satisfied, except the condition $\target(s_j')$ to be a sink.
That is, $c=1$, $d=3 (> 2)$, $s_1=a_{1,2}$, $s_1'=a_{2,3}$, $s_2'=a_{2,4} (=s_j')$, $s_3'=a_{2,3'}$,
$\ell(s_2')=1$, and $s_1s_2'\notin\I$ hold.
In this instance, $\spds_0$ is non-projective.
Indeed, one can check that $\spds_0 \cong S(4)$ is non-projective simple.

(2) Consider the almost gentle pair $(\Q,\I'')$ given by the quiver $\Q$ provided in Example \ref{exp:almost gent}
and the admissible ideal $\I''=\I\backslash\{a_{1,2}a_{2,4}\}$,
Then the vertex $2$ is a $(1^{\inner},3^{\out})$-type vertex such that the conditions given in Lemma \ref{lemm:spds III}(2) are satisfied.
Thus, for $E(2) = (^1_2)$, we have that $\spds_0$ is projective. Indeed, one can check that $\spds_0 \cong ({^4_5}) \cong P(4)$
since $P(1)$ is the indecomposable projective module corresponding to the claw
\[ a_{1,2_{\Left}} \claw a_{1,2}a_{2,4}a_{4,5} \claw
a_{1,2_{\Right}}a_{2_{\Right},3_{\Right}}a_{3_{\Right},4_{\Right}}
\claw b_{1,2_{\Right}}. \]

(3) The indecomposable injective module $E(2)$ over the almost gentle algebra $A=\kk\Q/\I$ shown in Example \ref{exp:almost gent}
is the one corresponding to the $(1^{\inner}, 3^{\out})$-type vertex $2$.
We have $c=1$, $d=3$ $(\geqslant 2)$, $s_1=a_{1,2}$, $s_1'=a_{2,3}$, $s_2'=a_{2,4}a_{4,5}$, and $s_3'=a_{2,3'}$.
Obviously,  $4 \in 2^{\da}=\{3,4,3'\}$ is a vertex on the path $s_2'$, and the length of $s_2'$ is two.
Notice that $s_1s_j'\in\I$ holds for all $j\in\{1,2,3\}$, then $\spds_0$ is projective by Lemma \ref{lemm:spds III} (3).
In fact, we can check that
\begin{align*}
 \Omega_1(E(2)) \cong (2_{\Left})\oplus (2_{\Right}) \oplus
  \left(
  {
  \begin{smallmatrix}
    2_{\Right} \\ 3_{\Right} \\ 4_{\Right}
  \end{smallmatrix}
  }
  \right) = \bigoplus_{\imath\in I = \{1,2,3\}} M_{\imath}
\end{align*}
\begin{align*}
  \spds_0 = 0, \ \text{ and } \bigoplus_{\jmath\in J = \emptyset} S_{\jmath} = 0
\end{align*}
by using Lemma \ref{lemm:1st-syzygy of E}. Thus $\spds_0 = 0$ is projective.
\end{example}

\begin{definition} \label{def:noname} \rm
For a vertex $v$ of almost gentle pair $(\Q,\I)$, assume that $v$ is a $(c^{\inner},d^{\out})$-type vertex.
We call it an {\defines \noname vertex} if one of the following conditions holds.
\begin{itemize}
  \item[\rm(1)] $v$ is a gentle vertex with $c=2$; 
  \item[\rm(2)] $c$ is arbitrary and $d=0$;
  \item[\rm(3)] $c=1$, there is a unique right maximal path $s_j'$ $(1\leqslant j \leqslant d)$ such that $s_1s_j'\notin\I$, $\ell(s_j')=1$, and $\target(s_j')$ is a sink of the quiver $\Q$;
  \item[\rm(4)] $c=1$, there is a unique right maximal path $s_j' = $
    \[ v=v_{1,1,\ell_{11}+1} \To{a_{1,1,\ell_{11}+1}}{}
         v_{1,1,\ell_{11}+2} \To{a_{1,1,\ell_{11}+2}}{}
         \cdots \To{a_{1,1,l_{11}}}{} v_{{1,1,l_{11}+1}}
    \]
    $(1\leqslant j \leqslant d)$ such that $s_1s_j'\notin\I$, $\ell(s_j')\geqslant 2$, and the vertex $v_{1,1,\ell_{11}+2}$ is not an  $(a_{1,1,\ell_{11}+1},\bfda)$-relational vertex;
  \item[\rm(5)] $c=1$ and $s_1s_j' \in \I$ for any $1\leqslant j \leqslant d$;
\end{itemize}
\end{definition}

\begin{proposition} \label{prop:projectivity}
Keep the notation in Lemma \ref{lemm:1st-syzygy of E} and Corollary \ref{coro:1st-syzygy of E}.
For the indecomposable injective module $E(v)$ corresponding to the $(c^{\inner}, d^{\out})$-type vertex $v$,
if $c\geqslant 1$, then the direct summand $\spds_0$ $(\le_{\oplus}\Omega_1(E(v)))$ is projective if and only if $v$ is \noname.
\end{proposition}

\begin{proof}
In the case for $1\leqslant c\leqslant 2$ and $d=0$, we have that
$\spds_0 \cong S(v)^{\oplus(c-1)}$ is projective since $v$ is a sink. Thus, in the condition (2), we assume $c\geqslant 3$.

The conditions (1)--(5) given in Definition \ref{def:noname} are given by Lemmata \ref{lemm:spds I}, \ref{lemm:spds II}
and \ref{lemm:spds III}, respectively.
To be more precisely, if $c=2$, $d\leqslant 2$ and $v$ is a gentle vertex, then Lemma \ref{lemm:spds I} provides the condition (1);
if $c\geqslant 3$, then Lemma \ref{lemm:spds II} provides the condition (2);
if $c=1$ and $d\geqslant 3$, then $v$ satisfying the conditions given in Lemma \ref{lemm:spds III} is a vertex such that the conditions (3)--(5) hold.
Thus, we only need to consider the following two cases:
\begin{itemize}
  \item[(A)] $1\leqslant c \leqslant 2$, $d\leqslant 2$ and $v$ is not a gentle vertex;
  \item[(B)] $c=1$, $d\leqslant 2$, and $v$ is a gentle vertex.
\end{itemize}

For the case (A), by the definition of gentle vertices, we have that one of the following cases occurs:
\begin{itemize}
\item[(a)] $c=1$, $d=2$, $s_1s_1' \in \I$, and $s_1s_2' \in \I$;
\item[(b)] $c=2$, $d=1$, $s_1s_1' \in \I$, and $s_2s_1' \in \I$;
\item[(c)] $c=d=2$ such that $s_1s_2'$ and $s_1s_2'$ lie in $\I$, and at least one of $s_1s_1'$ and $s_2s_2'$ lies in $\I$.
\end{itemize}

In the case (a), we have $\spds_0=0$ which is a trivial projective module, we obtain the condition (5).

In the case (b), we have $\spds_0 \cong S(v)$ which is non-projective since $v$ does not be a sink of $\Q$.

In the case (c), if $s_1s_1'\in \I$ and $s_2s_2'\notin \I$,
then $\spds_0$ is a directed string module corresponding to the string $s_2'$.
In this case, the kernel of the projective cover of $\spds_0$, which is of the form $P_0\to \spds_0$, is non-zero
since the following formula
\[\dim_{\kk}P_0 -\dim_{\kk} \spds_0 = (\ell(s_1')+\ell(s_2')+1) - (\ell(s_2')+1) = \ell(s_1') >0\]
holds. Thus, $\spds_0$ is non-projective. Similarly,
We have that $\spds_0$ is non-projective in other cases.

For the case (B), if $d=0$ then it is trivial that $\spds_0 = 0$.
Thus, we have $1\leqslant d \leqslant 2$, and obtain the following two cases.
\begin{itemize}
\item[(d)] $c=1$ and $d=1$.
\item[(e)] $c=1$ and $d=2$.
\end{itemize}

In the case (d), we have that $v$ is a $(1^{\inner},1^{\out})$-type vertex and $s_1s_1'$ equals to
\[ \overbrace{v_{1,1,1} \To{a_{1,1,1}}{} \cdots \To{a_{1,1,\ell_{11}}}{}
    v}^{s_1}
    = \overbrace{v_{1,1,\ell_{11}+1} \To{a_{1,1,\ell_{11}+1}}{}
      \underbrace{v_{1,1,\ell_{11}+2} \To{a_{1,1,\ell_{11}+2}}{}
     \cdots \To{a_{1,1,l_{11}}}{} v_{{1,1,l_{11}+1}} }_{s_1''} }^{s_1'}. \]
Then $\spds_0 \cong \M(s'')$ is projective if and only if one of the following conditions holds:
\begin{itemize}
  \item[(d-1)] $s_1s_1'\notin\I$, $\ell(s_1')=1$, and $v_{1,1,\ell_{11}+2}=\target(s_1')$ is a sink,
  that is, $v_{1,1,\ell_{11}+2}$ is a vertex lying in $v^{\da}$ satisfying the condition (3);
  \item[(d-2)] $s_1s_1'\notin\I$, $\ell(s_1')\geqslant 2$, and $v_{1,1,\ell_{11}+2}$ is not a $(a_{1,1,\ell_{11}+2},\bfda)$-relational vertex,
  that is, the condition (4) holds;
  \item[(d-3)] $s_1s_1'\in\I$. In this case, the vertex $v$ satisfies the condition (5).
\end{itemize}

If $v$ is a $(1^{\inner},2^{\out})$-type vertex, then at least one of $s_1s_1'$ and $s_1s_2'$ lies in $\I$.
If $s_1s_1'\notin \I$, then $s_1s_2'\in \I$. Thus, similar to the case (d),
$\spds_0$ is projective if and only if one of the following conditions holds:
\begin{itemize}
  \item[(e-1)] $s_1s_1'\notin\I$, $\ell(s_1')=1$, and $v_{1,1,\ell_{11}+2}=\target(s_1')$ is a sink, this satisfies the condition (3);
  \item[(e-2)] $s_1s_1'\notin\I$, $\ell(s_1')\geqslant 2$, and $v_{1,1,\ell_{11}+2}$ is not a $(a_{1,1,\ell_{11}+2},\bfda)$-relational vertex, this satisfies the condition (4);
\end{itemize}
The cases for $s_1s_1'\in\I$ and $s_1s_2'\notin \I$ are dual.
\begin{itemize}
  \item[(e-3)] $s_1s_1'$ and $s_1s_2'$ lie in $\I$. Then we have $\spds_0=0$ is projective, this satisfies the condition (5).
\end{itemize}
\end{proof}

\subsection{Self-injective dimension}

We divide all vertices of an almost gentle pair $(\Q,\I)$ to three parts:
sources; $(1^{\inner},d^{\out})$-vertices, and $((\geqslant 2)^{\inner},d^{\out})$-vertices.
If a vertex $v$ is a source, then $E(v)$ is a simple module corresponding to $v$,
and in this situation we have computed the projective dimension $\pdim E(v)$ in Proposition \ref{prop:forb-Omega}, i.e.,
\[ \pdim E(v) = \pdim S(v) = \ell(F_v), \]
where $F_v$ is the right maximal forbidden path starting at the sink $v$.

Next, we compute $\pdim E(v)$ in the case for $v$ being a vertex lies in the second and third situations.
By Lemma \ref{lemm:1st-syzygy of E}, we have that $\Omega_1(E(v))$ is a direct sum of some right maximal directed string modules and the module $\spds_0$.
The projective resolution of any right maximal directed string module can be computed by using Lemma \ref{lemm:nth-syzygy}.
Therefore, we need to construct the projective resolution of $\spds_0$.

\begin{lemma} \label{lemm:n-th syzygy of E}
For any $n\geqslant 1$, the $n$-th syzygy $\Omega_n(\spds_0)$ of $\spds_0$ is a direct sum of some right maximal directed string modules and some simple module.
\end{lemma}

\begin{proof}
Let $v$ be a $(c^{\inner},d^{\out})$-type vertex shown in \Pic \ref{fig:c-in d-out II},
and assume that $s_is_j'\notin\I$ if and only if $0\leqslant i=j \leqslant t$ $(<\min\{c,d\})$.
Recall that any $A$-module $M$ over the almost gentle algebra $A=\kk\Q/\I$ can be described by the bound quiver $(\Q,\I)$ as the quiver representation $(V_w, \varphi_{a})_{w\in\Q_0, a\in\Q_1}$,
where $V_w$ is the $\kk$-linear space which is $\kk$-linear isomorphic to $Me_w$
and $\varphi_{a}$ is a $\kk$-linear space $V_{\source(a)} \to V_{\target(a)}$
such that $\varphi_{a}\varphi_{b} = 0$ ($\target(a)=\source(b)$) if and only if $ab\in\I$.

Computing the kernel of the projective cover of $E(v)$, we obtain that the quiver representation of $\spds$ is of the form
\[ (\spds e_w, \varphi_{a})_{w\in\Q_0,a\in\Q_1}, \]
where for any vertex $w$ on the path $s_j''$ except $v$ ($s_j''$ a string obtained by deleting the first arrow of $s_j'$), and
\[ \spds e_w \cong \M(s_j'')e_w, \]
to be precise, we have $\rad(\spds) \cong \bigoplus_{j=1}^d \M(s_j'')$.
\Pic \ref{fig:spds repr} provides an example for the above, and for simplicity we assume that
each path $s_j'$ has no self-injection in this figure.
\begin{figure}[htbp]
\centering
\begin{tikzpicture}
\draw[top color = white, middle color = blue!50]
  (-3.8,-1.7) -- (3.8,-1.7) to[out=  0,in=  0] (3.8,-0.8) -- (-3.5,-0.8)
  -- (-4.0,-0.2) -- (-3.8,-0.8) to[out=180,in=180] (-3.8,-1.7);
\draw (0.1,0) node{
\xymatrix@R=2cm{
  & &
\kk^{\oplus (t-1)} \ar@{~>}[lld]_{s_1'}
  \ar@{~>}[ld]^{s_2'}
  \ar@{~>}[rd]_{s_{d-1}'}
  \ar@{~>}[rrd]^{s_d'}
  & &
\\
\kk_1 & \kk_2 & \cdots & \kk_{d-1} & \kk_d
}
};
\draw[blue] (-4.3,0) node{$\kk_j\cong \kk$, $1\leqslant j\leqslant d$};
\end{tikzpicture}
\caption{The quiver representation of $\spds$}
\begin{center}
  (all $\kk$-linear spaces corresponding to vertices on each $s_j'$ is one-dimensional,

  except the $\kk$-linear space $\kk^{\oplus (t-1)}$ corresponding to the vertex $v$)
\end{center}
\label{fig:spds repr}
\end{figure}
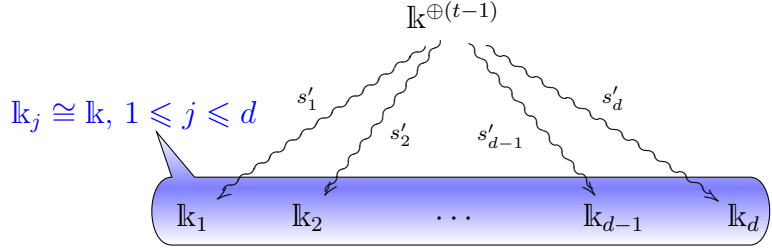
Furthermore, $\top(\spds) \cong S(v)^{\oplus (t-1)}$, and then the projective cover of $\spds$ is of the form $P(v)^{\oplus (t-1)} \to \spds$.
Now, we have $\Omega_1(\spds) \cong \bigoplus_{j=1}^d \M(s_j'')$,
then it is clear that each direct summand of $\Omega_1(\spds)$ is either a right maximal directed string module or a simple module,
and so is $\Omega_{n}(\spds_1)$ by using Proposition \ref{prop:nth-syzygy} and induction.

On the other hand, by Lemma \ref{lemm:1st-syzygy of E}, we have
\[ \spds_0 = \spds \oplus \bigoplus_{\jmath\in J}S_{\jmath}, \]
it follows that $\displaystyle \Omega_n(\spds_0) = \Omega_n(\spds) \oplus \bigoplus_{\jmath\in J}\Omega_n(S_{\jmath})$.
Notice that $S_{\jmath}$ is a direct string module corresponding to a simple string,
so each direct summand of $\Omega_n(S_{\jmath})$ is either a right maximal directed string module
or a simple module by using Proposition \ref{prop:nth-syzygy} and induction.
\end{proof}

For a directed string $\ds=a_1a_2\cdots a_l$, a forbidden path $F=b_1b_2\cdots b_l$ is said to
be a {\defines $\ds$-forbidden} if one of the following conditions holds:
\begin{itemize}
  \item $\source(F)=\target(\ds)$ and $\ds b_1\notin\I$;
  \item $\source(F)=\source(\ds)$ and $a_1\ne b_1$.
\end{itemize}

\begin{example} \rm
The indecomposable injective module $E(2) = ({^1_2})$ over the almost gentle algebra $A=\kk\Q/\I$ given in Example \ref{exp:almost gent} is a directed string module corresponding to the directed string $a_{1,2}$.
Then the forbidden paths $b_{1,2_{\Right}}a_{2_{\Right},3_{\Right}}$, $a_{1,2_{\Right}}$ and $b_{1,2_{\Left}}a_{2_{\Left},3_{\Left}}$ are $a_{1,2}$-forbidden paths.
But $a_{2,4}$ is not an $a_{1,2}$-forbidden path because $a_{1,2}a_{2,4}\in\I$.
\end{example}

\begin{notation} \label{nota:F(ds)}
\rm
Let $\forb(\ds)$ be the set of all $\ds$-forbidden paths.
If $\ell(\ds)=0$, then $\forb(\ds)=\forb(v)$, where $v=\source(\ds)=\target(\ds)$.
\end{notation}

Now we show the following proposition.

\begin{proposition} \label{prop:forb-Omega:ds case}
Let $\ds$ be either a right maximal directed string or a simple string. Then $\Omega_{n-1}(\M(\ds))$
is non-projective if and only if there is a forbidden path $F\in\forb(\ds)$ whose length is $n$.
\end{proposition}

\begin{proof}
The case for $\ds$ to be simple is shown in Proposition \ref{prop:forb-Omega}.
Now, we assume that $\ds$ is the right maximal directed string
\[ v_{1,1} \To{a_{1,1}}{} v_{1,2} \To{a_{1,2}}{} \cdots \To{a_{1,l_1}}{} v_{1,l_1+1} \]
with length $l_1\geqslant 1$ and $v_{1,1}$ a $(c^{\inner},d^{\out})$-type vertex ($d\geqslant 1$),
and the claw corresponding to $P(v_{1,1})$ is
\[ \clawnota = \ds_1 \claw \ds_2 \claw \cdots \claw \ds_d, \]
where $\ds=\ds_1$, $\ds_i = \ v_{i,1} \To{a_{i,1}}{} v_{i,2} \To{a_{i,2}}{} \cdots \To{a_{i,l_i}}{} v_{i,l_i+1}$
($l_i\geqslant 1$, $1\leqslant i\leqslant d$) is right maximal,
and $v_{1,1}=v_{2,1}=\cdots=v_{d,1}$.
Then we have
\[ \Omega_1(\M(\ds)) \cong \bigoplus_{i=1}^{d} \M(\ds_i^{*}), \]
where
\[ \ds_i^{*} =
   \begin{cases}
    0, &  \text{ if } i=1; \\
     v_{i,2} \To{a_{i,2}}{} \cdots \To{a_{i,l_i}}{} v_{i,l_i+1}, & \text{ if } i\geqslant 2.
   \end{cases}
\]

In the case for $d=1$, $\M(\ds)$ is projective, and so $\Omega_1(\M(\ds)) = 0$ is projective.

In the case for $d\geqslant 2$, for any $i\geqslant 2$, we have $\M(\ds_i^{*})\ne 0$.
Furthermore, $\M(\ds_i^{*})$ is non-projective if and only if $v_{i,2}$ is $(a_{i,1},\bfda)$-relational.
It follows that
\begin{itemize}
  \item
    there is an arrow $a_{i,2}^{(1)}$ such that $a_{i,1}a_{i,2}^{(1)}$ is a forbidden path lying in $\forb(\ds)$;
  \item
    the $1$-st syzygy $\Omega_1(\M(\delta))$ is non-projective;
  \item
    there is a directed string $\delta_i^{(1)}$ which is of the form $a_{i,2}^{(1)}\cdots$
    such that
    \[ \M(\delta_i^{(1) *}) \le_{\oplus} \Omega_1(\M(\delta_i^{*})) \le_{\oplus} \Omega_2(\M(\ds)),\]
    where $\delta_i^{(1) *}$ is the string obtained by deleting the first arrow $a_{i,2}^{(1)}$ of $\delta_i^{(1)}$,
    see \Pic \ref{fig:1st syzygy non-proj}.
\end{itemize}
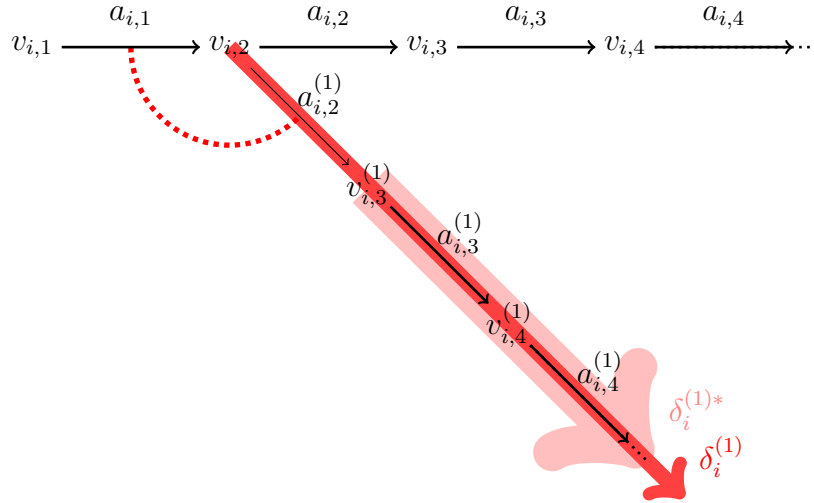
\begin{figure}[htbp]
\centering
\begin{tikzpicture}[scale=1.3]
\draw[rotate around={-45:(0,0)}][line width =18pt][red!25][->] (2,0) -- (6,0);
\draw[rotate around={-38:(0,0)}][line width = 6pt][red!50]  (6,0) node{$\ds_i^{(1)*}$};
\draw[rotate around={-45:(0,0)}][line width = 6pt][red!75][->] (0,0) -- (6.5,0);
\draw[rotate around={-40:(0,0)}][line width = 6pt][red]  (6.5,0) node{$\ds_i^{(1)}$};
\draw[line width = 2pt][red][dotted] (-1,0) arc (180:315:1);
\draw[shift={(-2, 0)}] (0,0) node{$v_{i,1}$};
\draw[shift={( 0, 0)}] (0,0) node{$v_{i,2}$};
\draw[shift={( 2, 0)}] (0,0) node{$v_{i,3}$};
\draw[shift={( 4, 0)}] (0,0) node{$v_{i,4}$};
\draw[shift={(-2, 0)}] (1,0.3) node{$a_{i,1}$};
\draw[shift={( 0, 0)}] (1,0.3) node{$a_{i,2}$};
\draw[shift={( 2, 0)}] (1,0.3) node{$a_{i,3}$};
\draw[shift={( 4, 0)}] (1,0.3) node{$a_{i,4}$};
\draw[shift={(-2, 0)}][->] (0.3,0) -- (1.7,0) [line width = 1pt];
\draw[shift={( 0, 0)}][->] (0.3,0) -- (1.7,0) [line width = 1pt];
\draw[shift={( 2, 0)}][->] (0.3,0) -- (1.7,0) [line width = 1pt];
\draw[shift={( 4, 0)}][->] (0.3,0) -- (1.7,0) [line width = 1pt];
\draw[shift={( 4, 0)}][dotted] (0.3,0) -- (2,0) [line width = 1pt];
\draw[rotate around={-45:(0,0)}][shift={(0,0)}][->] (0.3,0) -- (1.7,0);
\draw[rotate around={-45:(0,0)}][shift={(0,0)}] (1,0.3) node{$a_{i,2}^{(1)}$};
\draw[rotate around={-45:(0,0)}][shift={(2,0)}] (0,0) node{$v_{i,3}^{(1)}$};
\draw[rotate around={-45:(0,0)}][shift={( 4, 0)}] (0,0) node{$v_{i,4}^{(1)}$};
\draw[rotate around={-45:(0,0)}][shift={( 2, 0)}] (1,0.3) node{$a_{i,3}^{(1)}$};
\draw[rotate around={-45:(0,0)}][shift={( 4, 0)}] (1,0.3) node{$a_{i,4}^{(1)}$};
\draw[rotate around={-45:(0,0)}][shift={( 2, 0)}][->] (0.3,0) -- (1.7,0) [line width = 1pt];
\draw[rotate around={-45:(0,0)}][shift={( 4, 0)}][->] (0.3,0) -- (1.7,0) [line width = 1pt];
\draw[rotate around={-45:(0,0)}][shift={( 4, 0)}][dotted] (0.3,0) -- (2,0) [line width = 1pt];
\end{tikzpicture}
\caption{$\Omega_1(\M(\ds))$ is non-projective}
\label{fig:1st syzygy non-proj}
\end{figure}
Therefore, if $\M(\ds_i)$ is non-projective, then so is $\Omega_1(\M(\ds))$,
and in this case, we get a forbidden path $a_{i,1}a_{i,2}^{(1)}$ lying in $\forb(\ds)$;
if $\M(\ds_i^{(1)*})$ is non-projective, then so is $\Omega_2(\M(\ds))$,
and in this case, by the method similar to above, there is a directed string $\ds_i^{(2)}$
which is of the form $a_{i,3}^{(2)}\cdots $ such that $a_{i,2}^{(1)}a_{i,3}^{(2)}$ is a forbidden path,
then we get a forbidden path $a_{i,1}a_{i,2}^{(1)}a_{i,3}^{(2)}$ lying in $\forb(\ds)$,
see \Pic \ref{fig:2nd syzygy non-proj}.
\begin{figure}[htbp]
\centering
\begin{tikzpicture}[scale=1.3]
\draw[rotate around={-45:(1.41,-1.41)}]
     [rotate around={-45:(0,0)}][line width =18pt][blue!25][->] (4,0) -- (6,0);
\draw[rotate around={-45:(1.41,-1.41)}]
     [rotate around={-38:(0,0)}][blue!50]  (5.5,0.3) node{$\ds_i^{(2)*}$};
\draw[rotate around={-45:(1.41,-1.41)}]
     [rotate around={-45:(0,0)}][line width = 6pt][blue!50][->] (2,0) -- (6.2,0);
\draw[rotate around={-45:(1.41,-1.41)}]
     [rotate around={-40:(0,0)}][blue]  (6.0,0) node{$\ds_i^{(2)}$};
\draw[rotate around={-45:(0,0)}][line width =18pt][red!25][->] (2,0) -- (6,0);
\draw[rotate around={-38:(0,0)}][red!50]  (5.5,0.3) node{$\ds_i^{(1)*}$};
\draw[rotate around={-45:(0,0)}][line width = 6pt][red!75][->] (0,0) -- (6.2,0);
\draw[rotate around={-40:(0,0)}][red]  (6.0,0) node{$\ds_i^{(1)}$};
\draw[line width = 2pt][ red][dotted] (-1,0) arc (180:315:1);
\draw[line width = 2pt][blue][dotted] (1.41*0.5,-1.41*0.5) arc (135:270:1);
\draw[shift={(-2, 0)}] (0,0) node{$v_{i,1}$};
\draw[shift={( 0, 0)}] (0,0) node{$v_{i,2}$};
\draw[shift={( 2, 0)}] (0,0) node{$v_{i,3}$};
\draw[shift={( 4, 0)}] (0,0) node{$v_{i,4}$};
\draw[shift={(-2, 0)}] (1,0.3) node{$a_{i,1}$};
\draw[shift={( 0, 0)}] (1,0.3) node{$a_{i,2}$};
\draw[shift={( 2, 0)}] (1,0.3) node{$a_{i,3}$};
\draw[shift={( 4, 0)}] (1,0.3) node{$a_{i,4}$};
\draw[shift={(-2, 0)}][->] (0.3,0) -- (1.7,0) [line width = 1pt];
\draw[shift={( 0, 0)}][->] (0.3,0) -- (1.7,0) [line width = 1pt];
\draw[shift={( 2, 0)}][->] (0.3,0) -- (1.7,0) [line width = 1pt];
\draw[shift={( 4, 0)}][->] (0.3,0) -- (1.7,0) [line width = 1pt];
\draw[shift={( 4, 0)}][dotted] (0.3,0) -- (2,0) [line width = 1pt];
\draw[rotate around={-45:(0,0)}][shift={(0,0)}][->] (0.3,0) -- (1.7,0) [line width = 1pt];
\draw[rotate around={-45:(0,0)}][shift={(0,0)}] (1,0.3) node{$a_{i,2}^{(1)}$};
\draw[rotate around={-45:(0,0)}][shift={(2,0)}] (0,0) node{$v_{i,3}^{(1)}$};
\draw[rotate around={-45:(0,0)}][shift={( 4, 0)}] (0,0) node{$v_{i,4}^{(1)}$};
\draw[rotate around={-45:(0,0)}][shift={( 2, 0)}] (1,0.3) node{$a_{i,3}^{(1)}$};
\draw[rotate around={-45:(0,0)}][shift={( 4, 0)}] (1,0.3) node{$a_{i,4}^{(1)}$};
\draw[rotate around={-45:(0,0)}][shift={( 2, 0)}][->] (0.3,0) -- (1.7,0) [line width = 1pt];
\draw[rotate around={-45:(0,0)}][shift={( 4, 0)}][->] (0.3,0) -- (1.7,0) [line width = 1pt];
\draw[rotate around={-45:(0,0)}][shift={( 4, 0)}][dotted] (0.3,0) -- (2,0) [line width = 1pt];
\draw[rotate around={-45:(1.41,-1.41)}]
     [rotate around={-45:(0,0)}][shift={( 2, 0)}] (1,0.3) node{$a_{i,3}^{(2)}$};
\draw[rotate around={-45:(1.41,-1.41)}]
     [rotate around={-45:(0,0)}][shift={( 4, 0)}] (0,0) node{$v_{i,4}^{(2)}$};
\draw[rotate around={-45:(1.41,-1.41)}]
     [rotate around={-45:(0,0)}][shift={( 4, 0)}] (1,0.3) node{$a_{i,4}^{(2)}$};
\draw[rotate around={-45:(1.41,-1.41)}]
     [rotate around={-45:(0,0)}][shift={( 2, 0)}][->]
     (0.3,0) -- (1.7,0) [line width = 1pt];
\draw[rotate around={-45:(1.41,-1.41)}]
     [rotate around={-45:(0,0)}][shift={( 4, 0)}][->]
     (0.3,0) -- (1.7,0) [line width = 1pt];
\draw[rotate around={-45:(1.41,-1.41)}]
     [rotate around={-45:(0,0)}][shift={( 4, 0)}][dotted]
     (0.3,0) -- (2,0) [line width = 1pt];
\end{tikzpicture}
\caption{$\Omega_2(\M(\ds))$ is non-projective (note that $a_{i,3}^{(1)}$ and $a_{i,3}^{(2)}$ may be the same)}
\label{fig:2nd syzygy non-proj}
\end{figure}
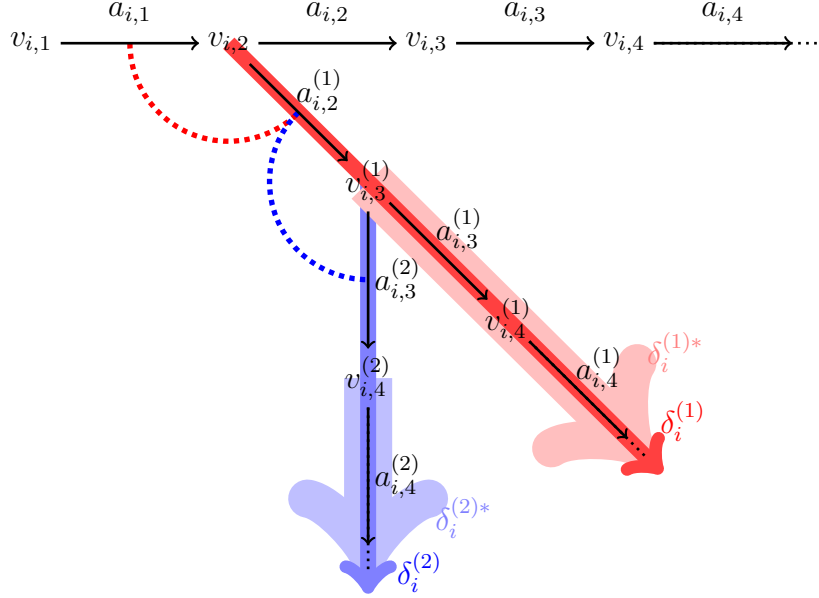
Repeating the above step, we can find a forbidden path of length $n$ which lies in $\forb(\ds)$ if $\Omega_{n-1}(\M(\ds))$ is non-projective.

Conversely, if $\ds$ is a simple string corresponding to $v_1$, and by assumption there is a forbidden path of length $n$
lying in $\forb(\ds)=\forb(v_1)$. Then $\Omega_{n-1}(\M(\ds))$ is non-projective by Proposition \ref{prop:forb-Omega}.
Next, consider the case for $\ds$ to be a right maximal directed string
\[ v_1 \To{a_1}{} v_2 \To{a_2}{} \cdots \To{a_l}{} v_{l+1} \]
with length $l \geqslant 1$.
Since $\forb(\ds)$ contains a forbidden path of length $n$ and $\ds$ is a right maximal directed string, all forbidden paths
lying in $\forb(\ds)$ share the same source $v_{1,1}$ by the definition of $\forb(\ds)$.
Suppose that there is a forbidden path $F=b_1b_2\cdots b_n$ which is of the form
    \begin{align*}
       w_1  \To{b_1}{} w_2 \To{b_2}{} \cdots \To{b_n}{} w_{n+1}
    \end{align*}
    such that $w_1 = v_1$ and $a_1\neq b_1$.
The $1$-st syzygy $\Omega_1(\M(\ds))$ contains a direct summand $M_1$
which is a directed string module (see Lemma \ref{lemm:1st-syzygy}) such that
$M_1b_2=0$ and $\top(M_1) \cong S(w_2)$ by computing the kernel of the projective cover of $\M(\ds)$.
Note that $M_1$ is non-projective. Otherwise, $M_1\cong e_{w_2}A$.
It follows that $M_1b_2 = e_{w_2}Ab_2 \ne 0$ since $e_{w_2}Ab_2$ contains at least one element $b_2 \ne 0$. This is a contradiction.
Then $\Omega_1(\M(\ds))$ is non-projective.

Since $w_2$ is $(b_1,\bfda)$-relational, the projective cover of $M_1$ is $P(w_2) \to M_1$,
then $\Omega_1(M_1)$ has a direct summand $M_2$ which is a directed string module (see Lemma \ref{lemm:1st-syzygy})
such that $M_2b_3=0$ and $\top(M_2) \cong S(w_3)$ by computing the kernel of the projective cover of $M_1$.
Similarly, we have $M_2\le_{\oplus}\Omega_1(M_1)$ ($\le_{\oplus} \Omega_2(\M(\ds))$) is non-projective,
and so $\Omega_2(\M(\ds))$ is non-projective.

Therefore $\Omega_{n-1}(\M(\ds))$ is non-projective by induction.
\end{proof}

\begin{example} \rm
Recall that the minimal projective resolution of the indecomposable injective module $E(2)$
over the almost gentle algebra $A=\kk\Q/\I$ given in Example \ref{exp:almost gent} is
\begin{align*}
 0 & \To{}{} P_2 = P(3_{\Left}) \oplus P(3_{\Right})
     \xymatrix@C=2cm{ \ar[r]^{
     \left(
       \begin{smallmatrix}
        a_{2_{\Left},3_{\Left}} & 0 \\
        0 & a_{2_{\Right},3_{\Right}} \\
        0 & 0
       \end{smallmatrix}
     \right)
     }_{p_2} &}
     P_1 = P(2_{\Left})\oplus P(2_{\Right}) \oplus P(2_{\Right}) \\
   & \xymatrix@C=2cm{ \ar[r]^{(a_{1,2_{\Left}}\ b_{1,2_{\Right}}\ a_{1,2_{\Right}})}_{p_1}  &}
     P_0=P(1) \To{}{p_0} \M(a_{1,2}) \cong E(2) \To{}{} 0,
\end{align*}
see Example \ref{exp:proj resol}.
Then we have
\begin{center}
  $\Omega_1(E(2)) \cong (2_{\Left}) \oplus (2_{\Right}) \oplus
  \left(\begin{smallmatrix}
    2_{\Right} \\ 3_{\Right} \\ 4_{\Right}
  \end{smallmatrix}\right) = S(2_{\Left}) \oplus S(2_{\Right}) \oplus P(2_{\Right})$
\end{center}
is non-projective since $S(2_{\Left})$ and $S(2_{\Right})$ are non-projective.
Here, $E(2)$ is also a right maximal directed string module, and the directed string corresponding to it is $a_{1,2}$.
We have three $a_{1,2}$-forbidden paths of length $\geqslant 1$ as follows:
\begin{center}
 $a_{1,2_{\Left}}a_{2_{\Left},3_{\Left}}$, $b_{1,2_{\Right}}a_{2_{\Right},3_{\Right}}$, and $a_{1,2_{\Right}}$.
\end{center}
We have that $S(2_{\Left})$ and $S(2_{\Right})$ correspond to the forbidden paths $a_{1,2_{\Left}}a_{2_{\Left},3_{\Left}}$ and $b_{1,2_{\Right}}a_{2_{\Right},3_{\Right}}$ of length two, respectively.
The fact that $P(2_{\Right}) \le_{\oplus} \Omega_1(E(2))$ yields $\Omega_1(P(2_{\Right})) = 0$.
Proposition \ref{prop:forb-Omega:ds case} shows that the forbidden paths $a_{1,2_{\Left}}a_{2_{\Left},3_{\Left}}$ and $b_{1,2_{\Right}}a_{2_{\Right},3_{\Right}}$, as two elements in $\forb(a_{1,2})$, exist.
\end{example}

In Lemma \ref{lemm:1st-syzygy of E}, we show that the $1$-st syzygy of an indecomposable injective $E(v)$ ($v$ is a $(c^{\inner},d^{\out})$)-type vertex module over an almost gentle algebra can be described by three parts:
\begin{itemize}
  \item a module $\spds$ whose top is a direct sum of some copies of $S(v)$;
  \item a semi-simple module which is of the form $S(v)^{\oplus (c-t)}$, where $t$ is an integer satisfying $0 \leqslant t \leqslant c$;
  \item a direct sum whose direct summand is either a right maximal directed module or a simple module;
\end{itemize}
Lemma \ref{lemm:n-th syzygy of E} shows that any direct summand of
the $n$-th ($n\geqslant 1$) syzygy $\Omega_n(\spds_0)$
of $\spds_0 = \spds \oplus S(v)^{\oplus (c-d)}$ is either a right maximal directed string module or a simple module.
Propositions \ref{prop:forb-Omega} and \ref{prop:forb-Omega:ds case} provide
a correspondence from any right maximal directed string module or simple module
to a forbidden path on the bound quiver of an almost gentle algebra.
Thus, for any $n\geqslant 1$, each non-projective direct summand of $\Omega_{n}(E(v))$
can be described by using a forbidden path of length $n+1$,
except the direct summand $\spds_0$ of $\Omega_1(E(v))$.
The projectivity of $\spds_0$ is described in Proposition \ref{prop:projectivity}, thus we can compute the projective resolution of any indecomposable injective module over almost gentle algebra.

\begin{definition} \label{def:F(aclawnota)} \rm
For an indecomposable injective module corresponding a $(c^{\inner},d^{\out})$-type vertex $v$,
assume that the anti-claw of $E(v)$, written as $\aclawnota$ in this section, is of the form shown in \Pic \ref{fig:c-in d-out II}.
A forbidden path $F$ is said to be a {\defines $\aclawnota$-forbidden path} if it satisfies one of the following conditions:
\begin{itemize}
  \item[(1)] there is an integer $1\leqslant i\leqslant c$ such that $F \in \forb(\source(s_i))$;
  \item[(2)] $F$ is a forbidden path starting at $v$ such that $v$ is not an \noname vertex.
\end{itemize}
The set of all $\aclawnota$-forbidden paths is written as $\forb(\aclawnota)$.
In particular, if $c=0$, then $E(v)=S(v)$ is simple, and we have $\forb(\aclawnota) = \forb(v)$.
\end{definition}

\begin{theorem} \label{thm:pdim inj}
Let $A=(\Q,\I)$ be an almost gentle algebra.
Then for any $(c^{\inner},d^{\out})$-type vertex $v$ {\rm(}keep the notation in
Lemma \ref{lemm:1st-syzygy of E} and \Pic \ref{fig:c-in d-out II}{\rm)},
we have
\[\pdim E(v) = \sup\limits_{F\in\forb(\aclawnota)} \ell(F), \]
where $\aclawnota$ is the anti-claw corresponding to $E(v)$.
\end{theorem}

  {\rm(}Note that if $\forb(\aclawnota)=\emptyset$,
  then take $\sup\limits_{F\in\forb(\aclawnota)} \ell(F)=0$.{\rm)}

\begin{proof}
If $c=0$, then $E(v)=S(v)$. Thus, we have
\[\pdim E(v) = \pdim S(v) = \sup\limits_{F\in\forb(v)} \ell(F).\]
When $c\geqslant 1$, we have two cases as follows:
\begin{itemize}
  \item[(1)] $v$ is \noname\!\!;
  \item[(2)] $v$ is not \noname\!\!.
\end{itemize}

Assume $\aclawnota = s_1\aclaw s_2 \aclaw \cdots \aclaw s_c$.
In the case (1), $\forb(\aclawnota)$ contains $c$ ($\geqslant 1$) right maximal forbidden paths
$F_1, \ldots, F_c$ which respectively correspond to $s_1,\ldots, s_c$ such that
$s_ib_{i1} \notin\I$, where $b_{i1}$ is the first arrow of $F_i = b_{i1}b_{i2}\cdots b_{il_i}$ ($l_i\geqslant 1$),
see Definition \ref{def:F(aclawnota)}(1).
By Proposition \ref{prop:forb-Omega:ds case}, we have that $\Omega_t(E(v))$ is non-projective
if and only if there is an integer $i$ with $1\leqslant i \leqslant c$ such that $\ell(F_i)=t+1$.
Without loss of generality, suppose
\begin{align} \label{formula:ell(F)}
  \ell(F_1) = \sup\limits_{1\leqslant i\leqslant c} \ell(F_i).
\end{align}
Since $\Omega_{\ell(F_1)-1}(E(v))$ is non-projective, we have $\pdim E(v) \geqslant \ell(F_1)$.
If $\pdim E(v) > \ell(F_1)$, then $\Omega_{\ell(F_1)}(E(v))$ is non-projective.
In this case,  by Proposition \ref{prop:forb-Omega:ds case}, there exists a forbidden path
$F_j\in \{F_1,\ldots, F_c\}$ such that $\ell(F_j) \geqslant \ell(F_1)+1$.
It contradicts the maximality of $\ell(F_1)$. Thus, we have
\begin{align} \label{formula:pdim(E) ell(F)}
  \pdim E(v) = \ell(F_1),
\end{align}
and therefore 
\begin{align} \label{formula:pdim(E) I}
  \pdim E(v) = \sup\limits_{1\leqslant i\leqslant c} \ell(F_i),
\end{align}
as required.

In the case (2), we obtain that $\forb(\aclawnota)$ contains $c$ ($\geqslant 1$) right maximal forbidden paths $F_1, \ldots, F_c$ and $F_{c+1}$,
where $F_1, \ldots, F_c$ respectively correspond to $s_1,\ldots, s_c$ such that $s_ib_{i1}\notin \I$ ($b_{i1}$ is the first arrow of $F_i$);
and $F_{c+1}$ corresponds to $v$, that is, $F_{c+1}$ is a forbidden path starting at $v$ (see Definition \ref{def:F(aclawnota)}(2)).

Next, we compute $\pdim\spds_0$. By Proposition \ref{prop:projectivity},
$v$ is not a sink since $d=0$ contradicts Definition \ref{def:noname}(2).
We have the following subcases:
\begin{itemize}
  \item[(2.1)] $c=1$ and $d\geqslant 1$;
  \item[(2.2)] $c=2$, then $v$ is not a gentle vertex;
  \item[(2.3)] $c\geqslant 3$, then $\spds_0$ is non-projective by Lemma \ref{lemm:spds II}.
\end{itemize}

We only prove the case (2.1). The proofs of (2.2) and (2.3) are similar.

In the subcase (2.1), there is a unique $s_j'$ ($1\leqslant j\leqslant d$) such that $s_1s_j'\notin \I$ by Definition \ref{def:noname}(5).
Without loss of generality, let $s_1= $
\[  v_{1,1,1} \To{a_{1,1,1}}{} v_{1,1,2} \To{a_{1,1,2}}{}
      \cdots \To{a_{1,1,\ell_{11}}}{} v_{1,1,\ell_{11}+1}=v \]
and assume $j=1$ and $s_j' = s_1' = $
\[
   v=v_{1,1,\ell_{11}+1} \To{a_{1,1,\ell_{11}+1}}{}
   v_{1,1,\ell_{11}+2} \To{a_{1,1,\ell_{11}+2}}{}
   \cdots \To{a_{1,1,l_{11}}}{} v_{{1,1,l_{11}+1}}.
\]
By Definition \ref{def:noname}(3)(4), we have that $v_{1,1,\ell_{11}+2}$ is $(a_{1,1,\ell_{11}+1}, \bfda)$-relational.
Otherwise,
\begin{itemize}
  \item[(2.1.1)] $v_{1,1,\ell_{11}+2}$ is a sink, and so $\ell(s_1')=1$, it contradicts Definition \ref{def:noname}(3);
  \item[(2.1.2)] $v_{1,1,\ell_{11}+2}$ is not a sink, and so $\ell(s_1')\geqslant 1$, it contradicts Definition \ref{def:noname}(2).
\end{itemize}
Thus, there is an arrow $\alpha_1$ such that $\alpha_1\ne a_{1,1,\ell_{11}+1}$,  $\source(\alpha_1) = v$, and $a_{1,1,\ell_{11}}\alpha_1\in \I$ is a forbidden path lying in $\forb(\aclawnota)$.
Thus, if $\target(\alpha_1)$ is $(\alpha_1,\bfda)$-relational,
that is, there is an arrow $\alpha_2$ such that $\target(\alpha_1) = \source(\alpha_2)$ and $\alpha_1\alpha_2\in \I$.
then $\Omega_1(\spds)$ ($\le_{\oplus} \Omega_2(E(v)$) has a direct summand $M_1$
such that $M_1\alpha_2=0$ and $\top(M_1) \cong S(\target(\alpha_1))$,
The direct summand $M_1$ is non-projective. Otherwise, $M_1\cong e_{\target(\alpha_1)}A$,
and so $M_1\alpha_2 \cong e_{\target(\alpha_1)}A\alpha_2 \ne 0$.
We find a forbidden path $a_{1,1,\ell_{11}}\alpha_1\alpha_2 \in \forb(\aclawnota)$
whose length equals to $3$.
Next, any direct summand of $\Omega_{\geqslant 2}(\spds)$ is right maximal directed string module. We get
\begin{align} \label{formula:pdim(E) II}
  \pdim E(v) = \sup\limits_{1\leqslant i\leqslant c+1} \ell(F_i)
\end{align}
by using an argument similar to the proof of (\ref{formula:pdim(E) I}).
\end{proof}

Now, we can prove the following result.

\begin{theorem} \label{thm:selfdim}
Let $A=\kk\Q/\I$ be an almost gentle algebra. Then
\[ \idim A=\pdim D(A) = \sup_{F\in \forb} \ell(F), \]
where $\forb := \bigcup_{v\in\Q_0} \forb(\aclawnota_v)$ and $\aclawnota_v$ is the anti-claw corresponding
to the indecomposable injective module $E(v)$.
\end{theorem}

\begin{proof}
By Theorem \ref{thm:pdim inj}, we have
\[ \pdim E(v) = \sup\limits_{F\in\forb(\aclawnota_v)} \ell(F) \]
for each $v$. Thus,
\[ \idim A = \pdim D(A) = \pdim \bigoplus_{v\in\Q_0} E(v) = \sup_{F\in \forb} \ell(F). \]
\end{proof}

\section{Auslander--Reiten Conjecture} \label{Sect:AGC}

In this section, we prove that the Auslander--Reiten conjecture ({\bf ARC} for short),
mentioned in Introduction, holds true for almost gentle algebras.

\subsection{The second description of self-injective dimension}

Recall that an oriented cycle $\C=a_1a_2\cdots a_n$ is said to be a {\defines forbidden cycle} if
$a_1a_2, \ldots, a_{n-1}a_n$ and $a_na_1$ are generators of $\I$.
For simplicity, assume that a forbidden cycle $\C$ of length $\ell$ is of the following form:
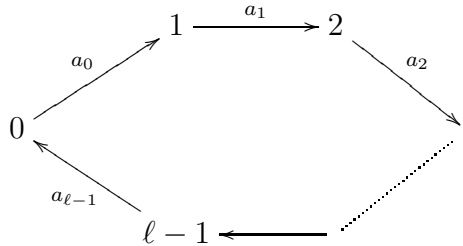
\[ \xymatrix@C=1.3cm{
 & 1 \ar[r]^{a_1} & 2 \ar[rd]^{a_2} & \\
0 \ar[ru]^{a_0} & & & \  \ar@{.}[ld] \\
 & \ell-1 \ar[lu]^{a_{\ell-1}} &\ar[l] &
} \]

%
%
%
%
%
%
%
%

\begin{theorem} \label{thm:forb cycle}
For an almost gentle algebra $A=\kk\Q/\I$ whose bound quiver $(\Q,\I)$ has at least one forbidden cycle,
then $\idim A = \infty$ if and only if there is a vertex $v$ on some forbidden cycle such that $v$ is not \noname\!\!.
\end{theorem}

\begin{proof}
If $\idim A = \infty$, then, by Theorem \ref{thm:selfdim}, there is an anti-claw $\aclawnota$ such that $\forb(\aclawnota)$
contains a forbidden path $F$ with $\ell(F)=\infty$ and $F$ being of the form:
\[ F = (\alpha_1\cdots \alpha_{n-1})
       (\overbrace{\alpha_{n}\alpha_{n+1}\cdots\alpha_{n+\ell-1}}
         ^{\C=a_0a_1\cdots a_{\ell-1}})
       (\overbrace{\alpha_{n}\alpha_{n+1}\cdots\alpha_{n+\ell-1}}
         ^{\C=a_0a_1\cdots a_{\ell-1}})\cdots, (n\geqslant 2). \]
We get $\target(\alpha_{n-1})=\target(\alpha_{n+\ell-1}) = \source(\alpha_n)$
($=\source(a_v)=v$ for some $0\leqslant v\leqslant \ell-1$).
Assume that $v$ is a $(c^{\inner},d^{\out})$-type vertex.
Obviously, $c \geqslant 2$ and $d\geqslant 1$.
\begin{itemize}
  \item[(a)] When $c>2$,
$v$ is \noname if and only if $d=0$, see Definition \ref{def:noname}.
Thus, $v$ is not \noname\!\!, a contradiction.
  \item[(b)] When $c=2$, we have
    \[ \top(\Omega_1(E(v))) \ge_{\oplus} \top(\spds_0) \cong S(v)^{\oplus c-1} = S(v). \]
    Since $v$ is a vertex on a forbidden cycle, one can check that $\pdim (\top(\spds_0)) = \infty$.
    It follows that $\spds_0$ is not projective,
    then $v$ is not \noname by Proposition \ref{prop:projectivity}.
\end{itemize}

Conversely, assume that every vertex on forbidden cycle is \noname\!\!.
Then for any vertex $v$ on each forbidden cycle,
the direct summand $\spds_0 = \spds \oplus \bigoplus_{\jmath\in J}S_{\jmath}$
of $\Omega_1(E(v))$ is projective by Lemma \ref{lemm:1st-syzygy of E}  and Proposition \ref{prop:projectivity}.
In this case, if $\idim A = \infty = \pdim D(A)$, then there is a right maximal directed string module $M \cong \M(\ds)$
such that $\pdim M = \infty$,
where $\ds=b_1b_2\cdots b_m$ is a right maximal directed string.
It follows that $\forb(\ds)$ contains a forbidden path $F$ such that $\ell(F)=\infty$.

If $\ell(\ds)\geqslant 1$, then, by the definition of $\forb(\ds)$, we have that
$\source(F)$ is a vertex on the claw
\begin{center}
$\clawnota = r_1\claw r_2\claw \cdots \claw r_n$

    (where $ r_j = \ w_{j,1} \To{\beta_{j,1}}{} w_{j,2} \To{\beta_{j,2}}{} \cdots \To{\beta_{j,l_j}}{} w_{j,l_j+1}$, $1\leqslant j\leqslant n$,

    $r_1 = \ds$, and $\source(b_1) = w_{1,1} = w_{2,1} = \ldots = w_{n,1}$)
\end{center}
corresponding to $P(\source(\ds))$,
the indecomposable projective module given by the projective cover of $\M(\ds)$,
such that $\source(F)$ is a $\clawnota$-relational vertex lying in
$\source(\ds)^{\da}\backslash\{\target(b_1)\} = \{ w_{j,2} \mid 2\leqslant j\leqslant n \}$.
Without loss of generality, assume $\source(F) = w_{2,2}$. Then $l_2=1$ (i.e., $r_2=\beta_{2,1}$)
and $F$ is of the following form:
\begin{align}\label{formula:forb cycle}
  F = (\alpha_1\cdots \alpha_{n-1})
       (\overbrace{\alpha_{n}\alpha_{n+1}\cdots\alpha_{n+\ell-1}}
         ^{\C=a_0a_1\cdots a_{\ell-1}})
       (\overbrace{\alpha_{n}\alpha_{n+1}\cdots\alpha_{n+\ell-1}}
         ^{\C=a_0a_1\cdots a_{\ell-1}})\cdots
\end{align}
\begin{center}
    (we take $\alpha_1\cdots \alpha_{n-1}=e_{\source(\alpha_n)}$ in the case for $n=0$)
\end{center}
such that $\beta_{2,1}\alpha_1 \in \I$ (we take $\beta_{2,1}\alpha_0 \in \I$ in the case for $n=0$).
By using an argument similar to the proof of the necessity, we get that
\begin{itemize}
  \item[(A)] If $n=0$, then $\beta_{2,1}F$ is a forbidden path which admits the vertex
  $\source(\alpha_0)=\target(\beta_{2,1}) = \target(\alpha_{\ell-1})$ on the forbidden cycle $\C$ which is not \noname
  \item[(B)] If $n\geqslant 1$, then $\source(\alpha_n) = \target(\alpha_{n-1}) = \target(\alpha_{n+\ell-1})$
  is a vertex on the forbidden cycle $\C$ which is not \noname
\end{itemize}

If $\ell(\ds)=0$, then $\ds=e_v$ is a path of length zero corresponding to some vertex $v$,
and $\forb(\ds)$ contains a forbidden path $F$ (which is of the form as shown in (\ref{formula:forb cycle})) with $\ell(F)=\infty$ such that
there is an arrow $\alpha$ starting at $v$ and ending at $\source(F)$ such that
$\alpha F$ is a forbidden path.
It follows the vertex $\source(\alpha_n)$ is not \noname by using an argument similar to the proof of the necessity.
\end{proof}

\begin{corollary} \label{coro:selfdim II}
Let $A=\kk\Q/\I$ be an almost gentle algebra. If $\idim A = \infty$, then the following statements hold.
\begin{itemize}
  \item[\rm(1)] The bound quiver $(\Q,\I)$ has at least one forbidden cycle.
  \item[\rm(2)] There is a vertex $v$ on some forbidden cycle $\C=a_0a_1\cdots a_{\ell-1}$ such that one of the following statements holds:
  \begin{itemize}
    \item[\rm(A)] there is an arrow $\alpha$ ending at $v$ satisfying $\alpha a_v \in \I$;
    \item[\rm(B)] there is an arrow $\beta$ starting at $v$ satisfying $a_{\overline{v-1}}\beta \in \I$,
    where $\overline{v-1}$ is $v-1$ modulo $\ell$.
  \end{itemize}
\end{itemize}
\end{corollary}

\begin{proof}
The statement (1) holds by Theorem \ref{thm:forb cycle}. Now we prove (2).

By Theorem \ref{thm:forb cycle}, we can find a vertex $v$ on some forbidden cycle $\C=a_0a_1\cdots a_{\ell-1}$ such that $v$ is not \noname\!\!.
Assume that $v$ is a $(c^{\inner},d^{\out})$-type vertex ($c\geqslant 1$, $d \geqslant 1$),
and for all arrows $\beta_1,\ldots, \beta_d$ starting at $v$ (see \Pic \ref{fig:c-in d-out}),
we have $a_{\overline{v-1}}\beta_j \notin \I$ holds for any $1\leqslant j\leqslant d$.
Then, by the definition of almost gentle algebras, we have $d=1$.
\begin{itemize}
  \item[(a)] When $c\geqslant 2$,
if there are two integers $i_1$ and $i_2$ with $1\leqslant i_1, i_2 \leqslant c$ such that $\alpha_{i_1}a_v$
and $\alpha_{i_2}a_v$ are paths of length two lying in $\I$,
  then it contradicts the definition of almost gentle pairs.
  \item[(b)] When $c=1$, there is no arrow $\alpha$ starting at $v$ such that $\alpha a_v$ lies in $\I$.
  Then $v$ is \noname by Definition \ref{def:noname}(5), a contradiction.
\end{itemize}

(Notice that if all vertices on forbidden cycle satisfy the case (b), then $(\Q,\I)$ is a forbidden cycle
in the case for $\Q$ to be connected. Furthermore, $A$ is self-injective in this case.)
\end{proof}

\begin{lemma} \label{lemm:selfdim II}
Let $A=\kk\Q/\I$ be an almost gentle algebra such that the bound quiver $(\Q,\I)$ contains at least one forbidden cycle.
Assume that there is a vertex $v$ on some forbidden cycle $\C=a_0a_1\cdots a_{\ell-1}$ such that one of the following conditions holds:
  \begin{itemize}
    \item[\rm(A)] there is an arrow $\alpha$ ending at $v$ satisfying $\alpha a_v \in \I$;
    \item[\rm(B)] there is an arrow $\beta$ starting at $v$ satisfying $a_{\overline{v-1}}\beta \in \I$.
  \end{itemize}
Then $\idim A = \infty$.
\end{lemma}

\begin{proof}
Assume that $v$ is a $(c^{\inner},d^{\out})$-type vertex as shown in \Pic \ref{fig:c-in d-out II}.
In Case (A), we have $c\geqslant 2$ and $d\geqslant 1$,  see \Pic \ref{fig:enter-label}.
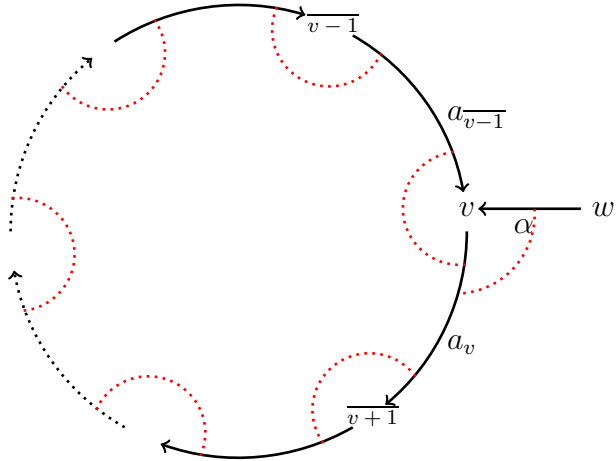
\begin{figure}[htbp]
\centering
\begin{tikzpicture}[scale=1.5]
\draw[rotate=   0][->] (2,0) arc(0:-50:2) [line width=1pt];
\draw[rotate= -60][->] (2,0) arc(0:-50:2) [line width=1pt];
\draw[rotate=-120][->] (2,0) arc(0:-50:2) [line width=1pt][dotted];
\draw[rotate=-180][->] (2,0) arc(0:-50:2) [line width=1pt][dotted];
\draw[rotate=-237][->] (2,0) arc(0:-50:2) [line width=1pt];
\draw[rotate=-300][->] (2,0) arc(0:-50:2) [line width=1pt];
\draw[rotate=  60]     (2,0.2) node{\tiny$\overline{v-1}$};
\draw[rotate=   0]     (2,0.2) node{$v$};
\draw[rotate= -60]     (2,0.2) node{\tiny$\overline{v+1}$};
\draw[rotate=   0] (1.73,1) node[right]{$a_{\overline{v-1}}$};
\draw[rotate= -60] (1.73,1) node[right]{$a_{v}$};
\draw[<-] (2.1, 0.2) -- (3,0.2) [line width=1pt];
\draw ( 2.5, 0.2) node[below]{$\alpha$};
\draw ( 3  , 0.2) node[right]{$w$};
\draw[red][rotate=6-  0] (1.90,-0.75) arc(-90:  -2:0.7) [line width=1pt][dotted];
\draw[red][rotate=6-  0] (1.95,-0.5 ) arc(-90:-270:0.5) [line width=1pt][dotted];
\draw[red][rotate=6- 60] (1.95,-0.5 ) arc(-90:-270:0.5) [line width=1pt][dotted];
\draw[red][rotate=6-120] (1.95,-0.5 ) arc(-90:-270:0.5) [line width=1pt][dotted];
\draw[red][rotate=6-180] (1.95,-0.5 ) arc(-90:-270:0.5) [line width=1pt][dotted];
\draw[red][rotate=6-240] (1.95,-0.5 ) arc(-90:-270:0.5) [line width=1pt][dotted];
\draw[red][rotate=6-300] (1.95,-0.5 ) arc(-90:-270:0.5) [line width=1pt][dotted];
\end{tikzpicture}
\caption{There exists an arrow $\alpha$ starting with $v$ such that $\alpha a_{v}\in \I$}
\label{fig:enter-label}
\end{figure}
If $c>2$, then $v$ is \noname if and only if $d=0$, see Definition \ref{def:noname}(2).
If $c=2$, then $v$ is \noname if and only if $v$ is gentle, see Definition \ref{def:noname}(1).
Thus, $v$ is not \noname in Case (A).
Let $\aclawnota$ be the anti-claw corresponding to $E(v)$. Then $\forb(\aclawnota)$ contains the forbidden path
\[ F = (a_va_{\overline{v+1}}\cdots a_{\ell-1})
   (\overbrace{a_0a_1\cdots a_{\ell-1}}^{\C})
   (\overbrace{a_0a_1\cdots a_{\ell-1}}^{\C})
   \cdots\]
whose length is infinite, and thus $\pdim E(v) = \infty$ by Theorem \ref{thm:pdim inj}.

In Case (B), we have $c\geqslant 1$ and $d\geqslant 2$.
If $c\geqslant 2$, then one can check that $v$ is not \noname by Definition \ref{def:noname}(1)(2).
In this case, we obtain $\pdim E(v)$ is infinite by an argument similar to the proof of Case (A),
and so is $\pdim D(A)$ ($=\idim A$), as required.
Now, consider the case for $c=1$ in Case (B).
Assume $\target(\beta)=w$. Then we have $S(v)\le_{\oplus}\top(E(w))$.
It follows that $\Omega_1(E(w))$ has a direct summand, written as $M$, such that
\begin{itemize}
\item $M$ is a right maximal directed string module;
\item $\top M = S(\overline{v+1})$.
\end{itemize}
Then $\pdim M=\infty$ since $\overline{v+1}$ is a vertex on $\C$, and so is $\pdim E(w)$.
Thus $\pdim D(A) = \infty$, as required.
\end{proof}

The following result provides a description of the infiniteness of $\idim A$.

\begin{theorem} \label{thm:selfdim II}
Let $A=\kk\Q/\I$ be an almost gentle algebra such that the bound quiver $(\Q,\I)$ contains at least one forbidden cycle.
$\C=a_0a_1\cdots a_{\ell-1}$. Then $\idim A = \infty$ if and only if there is a vertex $v$ on $\C$ such that
one of the conditions (A) and (B) given in Lemma \ref{lemm:selfdim II} holds.
\end{theorem}

\begin{proof}
It follows from Lemma \ref{lemm:selfdim II} and Corollary \ref{coro:selfdim II}.
\end{proof}

\subsection{Auslander--Reiten Conjecture}

\begin{lemma} \label{lemm:injenv(P)}
Let $A=\kk\Q/\I$ be an almost gentle algebra whose bound quiver $(\Q,\I)$ has at least one forbidden cycle.
If $\idim A = \infty$, then the injective envelope $E^0$ of $A$ has an infinite projective dimension.
\end{lemma}

\begin{proof}
By Theorem \ref{thm:selfdim II}, there exists a vertex $v$ on some forbidden cycle $\C=a_0a_1\cdots a_{\ell-1}$
such that one of the conditions (A) and (B) given in Lemma \ref{lemm:selfdim II} holds.
If $A$ satisfies the condition (A), that is, there is an arrow $\alpha$ ending at $v$ satisfying $\alpha a_v \in \I$,
then there are two cases as follows:
\begin{itemize}
  \item[\rm(1)] For any arrow $\beta$ starting at $v$, we have $\alpha\beta\in\I$.
  \item[\rm(2)] There exists a unique arrow $\beta$ with $\source(\beta)=v$, we have $\alpha\beta\notin \I$.
\end{itemize}

In Case (1), assume $\source(\alpha)=w$  (cf. \Pic \ref{fig:enter-label}).
Then the injective envelope $E^0_{P(w)}$ of $P(w)$ has a direct summand which is isomorphic to $E(v)$.
We have that the projective dimension $\pdim E(v)$ of $E(v)$ is infinite, see the proof of Lemma \ref{lemm:selfdim II}(A),
thus $\pdim E^0_{P(w)} = \infty$, it follows that $\pdim E^0 = \infty$.

In Case (2), there is a right maximal directed string $p = \beta_1\beta_2\cdots \beta_l$ ($\beta=\beta_1$) such that $\alpha p\notin\I$, cf. \Pic \ref{fig:enter-label II}.
\begin{figure}[htbp]
\centering
\begin{tikzpicture}[scale=1.5]
\draw[rotate=   0][->] (2,0) arc(0:-50:2) [line width=1pt];
\draw[rotate= -60][->] (2,0) arc(0:-50:2) [line width=1pt];
\draw[rotate=-120][->] (2,0) arc(0:-50:2) [line width=1pt][dotted];
\draw[rotate=-180][->] (2,0) arc(0:-50:2) [line width=1pt][dotted];
\draw[rotate=-237][->] (2,0) arc(0:-50:2) [line width=1pt];
\draw[rotate=-300][->] (2,0) arc(0:-50:2) [line width=1pt];
\draw[rotate=  60]     (2,0.2) node{\tiny$\overline{v-1}$};
\draw[rotate=   0]     (2,0.2) node{$v$};
\draw[rotate= -60]     (2,0.2) node{\tiny$\overline{v+1}$};
\draw[rotate=   0] (1.73,1) node[right]{$a_{\overline{v-1}}$};
\draw[rotate= -60] (1.73,1) node[right]{$a_{v}$};
\draw[<-] (2.1, 0.2) -- (3,0.2) [line width=1pt];
\draw (2.75, 0.2) node[above]{$\alpha$};
\draw ( 3  , 0.2) node[right]{$w$};
\draw[->][line width=1pt] (2.1,0.1) -- (2.7,-0.6);
\draw[->][shift={(0.8,-0.9)}][line width=1pt] (2.1,0.1) -- (2.7,-0.6)[dotted];
\draw[->][shift={(0.8*2,-0.9*2)}][line width=1pt] (2.1,0.1) -- (2.7,-0.6);
\draw (2.45,-0.3) node[right]{$\beta_1=\beta$};
\draw[shift={(0.8*2,-0.9*2)}] (2.45,-0.3) node[right]{$\beta_l$};
\draw (2.8,-0.7) node{$u_1$};
\draw[shift={(0.8,-0.9)}] (2.85,-0.75) node{$u_{l-1}$};
\draw[shift={(0.8*2,-0.9*2)}] (2.8,-0.7) node{$u_1$};
\draw[red][rotate=6-  0] (1.90,-0.75) arc(-90:  -2:0.7) [line width=1pt][dotted];
\draw[red][rotate=6-  0] (1.95,-0.5 ) arc(-90:-270:0.5) [line width=1pt][dotted];
\draw[red][rotate=6- 60] (1.95,-0.5 ) arc(-90:-270:0.5) [line width=1pt][dotted];
\draw[red][rotate=6-120] (1.95,-0.5 ) arc(-90:-270:0.5) [line width=1pt][dotted];
\draw[red][rotate=6-180] (1.95,-0.5 ) arc(-90:-270:0.5) [line width=1pt][dotted];
\draw[red][rotate=6-240] (1.95,-0.5 ) arc(-90:-270:0.5) [line width=1pt][dotted];
\draw[red][rotate=6-300] (1.95,-0.5 ) arc(-90:-270:0.5) [line width=1pt][dotted];
\draw[blue][shift={(-0.12,0.1)}][rotate=6][rotate around={47:(2,0)}]
     (1.90,-0.75) arc(-90:40:0.6) [line width=1.5pt][dotted];
\end{tikzpicture}
\caption{There exists an path $p=\beta_1\cdots\beta_l$ starting with $v$ such that $\alpha p\notin \I$ (i.e., such that $\alpha\beta_1\notin\I$)}
\label{fig:enter-label II}
\end{figure}
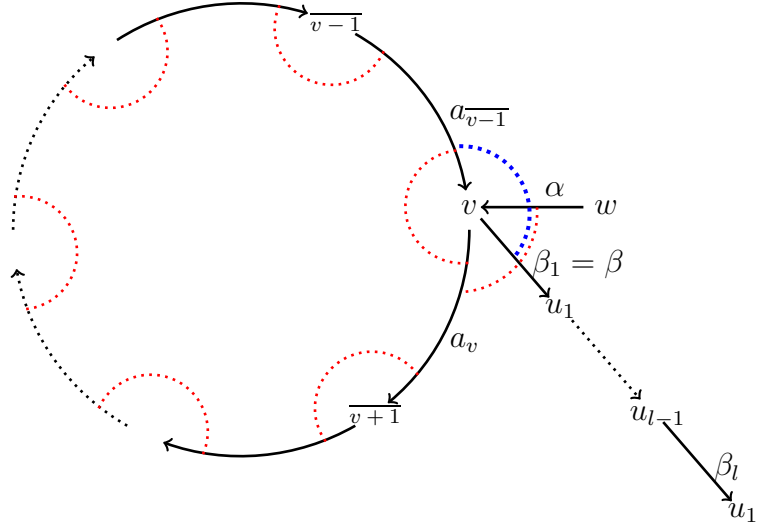
In this case, the injective envelope $E^0_{P(v)}$ of $P(v)$ has a direct summand $E(\target(\beta))$
(i.e., $E(u_1)$, where $u_1$ is shown in \Pic \ref{fig:enter-label II}).
By the definition of almost gentle algebras, we have $a_{\overline{v-1}}\beta \in \I$,
and then the anti-claw $\aclawnota$ corresponding to $E(u_1)$ is of the form
\[ \aclawnota = \ r_1\aclaw r_2\aclaw \cdots \aclaw r_m \]
such that $r_1 = p$.
Then
\[ F = (a_{\overline{v+1}}\cdots a_{\ell-1})
       (\overbrace{a_0a_1\cdots a_{\ell-1}}^{\C})
       (\overbrace{a_0a_1\cdots a_{\ell-1}}^{\C})\cdots \]
is a forbidden path lying in $\forb(\aclawnota)$ whose length is infinite.
By Theorem \ref{thm:pdim inj}, we have $\pdim E(u_1)=\infty$,
and so $\pdim E^0 = \infty$.
\end{proof}

\begin{corollary} \label{coro:injenv(P)}
Let $A$ be an almost gentle algebra with $\idim A = \infty$,
then $\pdim E^0 =\infty$, where $E^0$ is the injective envelope of $A_A$.
\end{corollary}

\begin{proof}
Since $\idim A = \infty$, there is an indecomposable injective module $E(v)$ such that $\pdim E(v) = \infty$.
Then by Theorem \ref{thm:pdim inj}, the set $\forb(\aclawnota)$ given by
the anti-claw $\aclawnota$ corresponding to $E(v)$ contains a forbidden path $F$ whose length is infinite.
Thus $F$ provides a forbidden cycle $\C$, that is,
$(\Q,\I)$ contains a forbidden cycle. Then $\pdim E^0 =\infty$ by Lemma \ref{lemm:injenv(P)},
\end{proof}

As a consequence, we obtain the following result.

\begin{theorem} \label{thm:AGC}
{\bf ARC} holds true for almost gentle algebras.
\end{theorem}

\begin{proof}
Let $A$ be an almost gentle algebra. If $\idim A <\infty$, then $A$ is Gorenstein.
If $\idim A = \infty$, then $A$ does not satisfy the Auslander condition by Corollary \ref{coro:injenv(P)}. The proof is finished.
\end{proof}

\section*{Acknowledgements}

This paper is supported by the National Natural Science Foundation of China (Grant Nos. 12371038, 12171207, and 12401042),
Guizhou Provincial Basic Research Program (Natural Science) (Grant Nos. ZK[2025]085 and ZK[2024]YiBan066)
and Scientific Research Foundation of Guizhou University (Grant Nos. [2023]16, [2022]53, and [2022]65).







\def\cprime{$'$}

\end{document}